\documentclass[preprint,aihp]{imsart}

\startlocaldefs

\RequirePackage[OT1]{fontenc}
\RequirePackage{amsmath,amssymb,amsthm,natbib}
\usepackage{epsfig,color,colordvi}
\definecolor{my-blue}{rgb}{0.0,0.0,0.6}
\definecolor{my-red}{rgb}{0.5,0.0,0.0}
\definecolor{my-green}{rgb}{0.0,0.5,0.0}
\RequirePackage[colorlinks,urlcolor=my-blue,linkcolor=my-red,citecolor=my-green]{hyperref}
\RequirePackage{hypernat}

\arxiv{arxiv:0705.4116}

\setattribute{journal}{name}{To appear at the Annales de l'Institut Henri Poincar\'e - Probabilit\'es et Statistique}  

\numberwithin{equation}{section}

\allowdisplaybreaks[1]

\theoremstyle{definition}

\theoremstyle{remark}

\theoremstyle{plain}
\newtheorem{theorem}{Theorem}[section]
\newtheorem{lemma}{Lemma}[section]
\newtheorem{proposition}[lemma]{Proposition}
\newtheorem{corollary}{Corollary}[section]

\newcommand*{\beq}{\begin{equation}}
\newcommand*{\eeq}{\end{equation}}
\newcommand*{\bal}{\begin{aligned}}
\newcommand*{\eal}{\end{aligned}}
\newcommand*{\nn}{\nonumber}

\providecommand{\abs}[1]{\left\vert#1\right\vert}

\providecommand{\pp}[1]{\langle#1\rangle}

\newcommand*{\one}{{{\rm 1\mkern-1.5mu}\!{\rm I}}}

\newcommand*{\vard}{d_{\scriptscriptstyle\rm Var}}

\newcommand*{\lsup}{\varlimsup}

\newcommand*{\ord}{\kern0.1em o\kern-0.02em{}_{\ds\breve{}}\kern0.1em}
\newcommand*{\Ord}{\kern0.1em O\kern-0.02em{\ds\breve{}}\kern0.1em}

\newcommand*{\ds}{\displaystyle}

\newcommand*{\e}{\varepsilon}

\newcommand*{\si}{\sigma}

\newcommand*{\w}{\omega}

\newcommand*{\cA}{{\mathcal A}}

\newcommand*{\cE}{{\mathcal E}}
\newcommand*{\cF}{{\mathcal F}}
\newcommand*{\cG}{{\mathcal G}}
\newcommand*{\cH}{{\mathcal H}}
\newcommand*{\cI}{{\mathcal I}}

\newcommand*{\cL}{{\mathcal L}}

\newcommand*{\cP}{{\mathcal P}}

\newcommand*{\cS}{{\mathcal S}}

\newcommand*{\kD}{{\mathfrak D}}

\newcommand*{\kS}{{\mathfrak S}}

\newcommand*{\E}{{\mathbb E}}

\newcommand*{\M}{{\mathbb M}}
\newcommand*{\N}{{\mathbb N}}
\renewcommand*{\P}{{\mathbb P}} 

\newcommand*{\R}{{\mathbb R}}

\newcommand*{\V}{{\mathbb V}}

\newcommand*{\Z}{{\mathbb Z}}

\newcommand*{\bfE}{{\mathbf E}}

\newcommand*{\bfP}{{\mathbf P}}

\newcommand*{\bhat}{{\hat b}}

\newcommand*{\uhat}{{\hat u}}
\newcommand*{\vhat}{{\hat v}}
\newcommand*{\what}{{\hat w}}

\newcommand*{\Lbar}{{\bar L}}
\newcommand*{\Mbar}{{\bar M}}

\newcommand*{\Tbar}{{\bar T}}

\newcommand*{\Xbar}{{\bar X}}
\newcommand*{\Ybar}{{\bar Y}}
\newcommand*{\Zbar}{{\bar Z}}

\newcommand*{\Gammabar}{{\bar\Gamma}}

\newcommand*{\qbar}{{\bar q}}

\newcommand*{\wbar}{{\bar w}}

\newcommand*{\zbar}{{\bar z}}

\newcommand*{\betabar}{{\bar\beta}}
\newcommand*{\gabar}{{\bar\gamma}}

\newcommand*{\zetabar}{{\bar\zeta}}

\newcommand*{\rhobar}{{\bar\rho}}

\newcommand*{\ombar}{{\bar\omega}}

\newcommand*{\Btil}{{\widetilde B}}

\newcommand*{\Etil}{{\widetilde E}}

\newcommand*{\Mtil}{{\widetilde M}}

\newcommand*{\Xtil}{{\widetilde X}}

\newcommand*{\ntil}{{\tilde n}}

\newcommand*{\wtil}{{\tilde w}}

\newcommand*{\ztil}{{\tilde z}}
\newcommand*{\altil}{{\tilde\alpha}}

\newcommand*{\betil}{{\tilde\beta}}

\newcommand*{\gatil}{{\tilde\gamma}}

\newcommand*{\mutil}{{\tilde\mu}}

\newcommand*{\sitil}{{\tilde\sigma}}

\newcommand*{\tautil}{{\tilde\tau}}

\newcommand*{\wwtil}{{{\tilde\omega}}}
\newcommand*{\elltil}{{\tilde\ell}}

\def\fmonth{\ifcase\month\or Jan\or Feb\or Mar\or Apr
\or May\or Jun\or Jul\or Aug\or Sep
\or Oct\or Nov\or Dec\fi\ }
\def\mmddyyyy{\the\month.\the\day.\the\year}
\def\ddmmyyyy{\the\day.\the\month.\the\year}
\def\Mddyyyy{\fmonth~\the\day,~\the\year}

\renewcommand{\Ord}{{\mathcal O}}

\def\HM{S} 
\def\M{r_0}  
\def\HT{M} 
\def\HR{R} 
\def\wiener{W}  
\def\cona{b_0}
\def\conb{b}
\def\mom{(\HT) }
\def\step{(\HM) }
\def\reg{(\HR) }
\def\amom{(\HT)} 
\def\astep{(\HM)}
\def\areg{(\HR)}
\newcommand*\momp{{p_0}}  
\newcommand*\ppath{\sigma}  
\newcommand*\group{\mathbb{S}}  
\newcommand*{\h}{h}   
\endlocaldefs

\begin{document}

\begin{frontmatter}
\title{Almost sure functional central limit theorem for ballistic 
random walk in random
environment}
\atltitle{Th\'eor\`eme limite central fonctionnel presque sure pour une marche al\'eatoire ballistique en 
milieu al\'eatoire} 
\runtitle{Quenched CLT for RWRE}

\begin{aug}
\author{\fnms{Firas} \snm{Rassoul-Agha$^1$}\ead[label=e1]{firas@math.utah.edu} \ead[label=u1,url]{http://www.math.utah.edu/$\sim$firas}}
\and
\author{\fnms{Timo} \snm{Sepp\"al\"ainen$^2$}\ead[label=e2]{seppalai@math.wisc.edu}\ead[label=u2,url]{http://www.math.wisc.edu/$\sim$seppalai}
}
\runauthor{F. Rassoul-Agha and T. Sepp\"al\"ainen}

\affiliation{University of Utah and University of Wisconsin-Madison}

\address{Department of Mathematics\\
University of Utah\\
155 South 1400 East\\
Salt Lake City, UT 84109\\
USA\\
\printead{e1}\\
\printead{u1}}

\address{Department of Mathematics\\
University of Wisconsin-Madison\\
419 Van Vleck Hall\\
Madison, WI 53706\\
USA\\
\printead{e2}\\
\printead{u2}}

\end{aug}
\footnotetext{Received January 2008; accepted February 2008.}
\footnotetext{$^1$Department of Mathematics, University of Utah.}
\footnotetext{$^1$Supported in part by NSF Grant DMS-0505030.}
\footnotetext{$^2$Mathematics Department, University of Wisconsin at Madison.}
\footnotetext{$^2$Supported in part by NSF Grants DMS-0402231 and
DMS-0701091.}

\begin{abstract}
We consider a multidimensional  random walk in a product
random environment with bounded steps,  transience in some
spatial direction, and high enough  
  moments on the  regeneration time.  
We prove  an invariance principle, or functional central
limit theorem,  under almost every  environment for the 
diffusively scaled centered walk.
The main point behind the invariance principle is that the quenched mean of the
 walk behaves subdiffusively.
\end{abstract}

\begin{abstract}[language=french]
Nous consid\'erons une marche al\'eatoire multidimensionnelle  en
environnement al\'eatoire produit. La marche est \`a pas born\'es, transiente dans
une direction spatiale donn\'ee, et telle que le temps
de r\'eg\'en\'eration poss\'ede un moment suffisamment haut. Nous prouvons un principe
d'invariance, ou un th\'eor\`eme limite central fonctionnel, sous presque tout
environnement pour la marche centr\'ee et diffusivement normalis\'ee. Le point principal
derri\`ere le principe d'invariance est que la moyenne tremp\'ee (quenched) de la marche
est sous-diffusive.
\end{abstract}

\begin{keyword}[class=AMS]
\kwd{60K37}
\kwd{60F05}
\kwd{60F17}
\kwd{82D30.}
\end{keyword}

\begin{keyword}
\kwd{random walk}
\kwd{ballistic}
\kwd{random environment}
\kwd{central limit theorem}
\kwd{invariance principle}
\kwd{point of view of the particle}
\kwd{environment process}
\kwd{Green function.}
\end{keyword}

\end{frontmatter}

\section{Introduction and main result}
We prove a quenched functional central limit theorem (CLT) for
ballistic random walk in random environment (RWRE) on the  $d$-dimensional
integer lattice $\Z^d$ in dimensions $d\ge 2$.
Here is a  general description of  the model,
fairly standard since quite a while.
An environment $\w$  is a configuration of probability vectors
$\w=(\w_x)_{x\in\Z^d}\in\Omega={\mathcal P}^{\Z^d},$  where
$\mathcal P=\{(p_z)_{z\in\Z^d}:p_z\ge0,\sum_z p_z=1\}$ is the simplex of all
probability  vectors on $\Z^d$.
Vector $\w_x=(\w_{x,z})_{z\in\Z^d}$ gives
the probabilities of jumps out of state $x$, and the transition
probabilities are 
denoted by $\pi_{x,y}(\w)=\w_{x,y-x}$.
To run the random walk, fix an environment $\w$ and
 an initial state $z\in\Z^d$.
 The random walk  $X_{0,\infty}=(X_n)_{n\geq0}$ in environment $\w$
started at $z$ is then the canonical
Markov chain with state space $\Z^d$
whose path measure  $P_z^\w$ satisfies
\begin{align*}
P_z^\w\{X_0=z\}=1\quad\text{and}\quad
P_z^\w\{X_{n+1}=y|X_n=x\}=\pi_{x,y}(\w).
\end{align*}

On the space $\Omega$ we put its  product
$\si$-field $\kS$,  natural shifts
$\pi_{x,y}(T_z\w)=\pi_{x+z,y+z}(\w)$, and
 a $\{T_z\}$-invariant probability measure $\P$ that makes the system
$(\Omega,\kS,(T_z)_{z\in\Z^d},\P)$ ergodic.  In this paper
 $\P$ is an  i.i.d.\  product  measure on ${\mathcal P}^{\Z^d}$.
In other words,    the vectors
$(\w_x)_{x\in\Z^d}$ are i.i.d.\ across the sites $x$ under $\P$.

Statements, probabilities and expectations
 under a fixed environment, such as the distribution  $P_z^\w$ above, are
called {\sl quenched}.  When also the environment is averaged out,
the notions are called {\sl averaged}, or also {\sl annealed}.
In particular,  the averaged distribution  $P_z(d x_{0,\infty})$ of the walk
is the marginal of the joint  distribution
$P_z(d x_{0,\infty},d\w)=P_z^\w(d{x}_{0,\infty})\P(d\w)$
on paths and environments.

Several excellent expositions on RWRE exist,
and we refer the reader to the lectures \citep{bolt-szni-dmv},
\citep{szni-trieste} and
\citep{zeit-stflour}. 

This paper  investigates the directionally transient situation. 
That is, we assume that there exists a vector  $\uhat\in\Z^d$ such that  
\begin{align}
\label{dir-trans}
P_0\{X_n\cdot\uhat\to\infty\}=1.
\end{align}
The key moment assumption (\HT) below is also 
expressed in terms of $\uhat$ so 
this vector needs to be fixed for the rest of the paper. 
There is no essential harm in assuming $\uhat\in\Z^d$ and this
is convenient. Appendix
\ref{vectorapp}  shows that at the expense of a larger
moment, an arbitrary $\uhat$ can be replaced by an integer
vector $\uhat$. 

The transience assumption provides regeneration times,
first defined and studied in the multidimensional setting by 
Sznitman and Zerner \citep{szni-zern-99}.
As a function of the path $X_{0,\infty}$ regeneration
time  $\tau_1$ is 
the first time at which 
\begin{align}
\label{tau}
\sup_{n<\tau_1} X_n\cdot\uhat<X_{\tau_1}\cdot\uhat\,=\inf_{n\ge\tau_1} X_n\cdot\uhat.
\end{align}
The benefit here is that 
 the past and the future of the walk lie in separate half-spaces.
Transience \eqref{dir-trans} is equivalent to $P_0(\tau_1<\infty)=1$
 \citep[Proposition 1.2]{szni-zern-99}. 

To be precise, 
\citep{szni-zern-99} is written under assumptions
of uniform ellipticity and  nearest-neighbor jumps.
  In an i.i.d.\ environment many properties established for
uniformly elliptic nearest-neighbor walks extend immediately to 
walks with bounded steps without ellipticity assumptions, 
the above mentioned equivalence among them.  
In such cases we treat the point simply as having been established
in earlier literature. 

In  addition
to the product form of $\P$, the following 
 three assumptions are used in this paper: a high
 moment \mom on $\tau_1$,
 bounded steps (S),
and some regularity (R).

\newtheorem*{hypothesisT}{\sc Hypothesis (\HT)}
\begin{hypothesisT}   
  $E_0(\tau_1^\momp)<\infty$ for some $\momp>176d$. 
\end{hypothesisT}

\newtheorem*{hypothesisM}{\sc Hypothesis (\HM)}
\begin{hypothesisM}
There exists a finite, deterministic, positive constant $\M$ such that 
$\P\{\pi_{0,z}=0\}=1$ whenever $|z|>\M$.
\end{hypothesisM}

\newtheorem*{hypothesisR}{\sc Hypothesis (\HR)}
\begin{hypothesisR}
Let ${\mathcal J}=\{z:\E\pi_{0,z}>0\}$ be the set of admissible steps
under $\P$.  Then
${\mathcal J}\not\subset\R u$  for all $u\in\R^d$,  and 
\beq
\P\{\exists z:\pi_{0,0}+\pi_{0,z}=1\}<1.\label{Yell5}
\eeq
\end{hypothesisR}

The bound on $\momp$ in Hypothesis \mom is of course meaningless
and only indicates that our result is true if $\momp$ is large enough.
We have not sought to tighten the exponent because in any case the final
bound would not be small with our current arguments.  
After the theorem we return to discuss the hypotheses further. 
These assumptions are strong enough to imply
a law of large numbers: there exists a  velocity $v\ne 0$ such that
\beq
P_0\bigl\{\,\lim_{n\to\infty}n^{-1}X_n=v\bigr\}=1.
\label{lln1}\eeq
  Representations for $v$ are given in \eqref{defv}
and Lemma \ref{velocity} below.
Define the (approximately)
 centered and diffusively scaled process
\begin{align}
B_n(t)=\frac{X_{[nt]}-[nt]v}{\sqrt{n}}.
\label{Bndef}
\end{align}
As usual  $[x]=\max\{n\in\Z: n\le x\}$ is the integer part of a real $x$.
Let $D_{\R^d}[0,\infty)$
be the standard Skorohod space of $\R^d$-valued cadlag paths
(see \citep{ethi-kurt} for the basics).
Let $Q_n^\w=P^\w_0(B_n\in\cdot\,)$
denote the quenched distribution
of the process $B_n$ on $D_{\R^d}[0,\infty)$.

The result of this paper
 concerns the limit of the process $B_n$
as $n\to\infty$.
As expected, the limit process is a Brownian motion with
correlated coordinates.
 For a symmetric, nonnegative definite
$d\times d$ matrix $\mathfrak D$,
a {\sl Brownian motion with diffusion matrix $\mathfrak D$} is the $\R^d$-valued
process $\{{B}(t):t\geq0\}$ with
 continuous paths, independent increments,
and such that  for $s<t$ the $d$-vector ${B}(t)-{B}(s)$ has Gaussian distribution
with mean zero and covariance matrix
$(t-s)\mathfrak D$.
The matrix $\mathfrak D$  is {\sl degenerate} in direction $u\in\R^d$
if $u^t\mathfrak D u=0$.   Equivalently,
$u\cdot {B}(t)=0$ almost surely.

Here is the main result.

\begin{theorem}
\label{main}
Let $d\geq2$ and
consider a random walk in an i.i.d.\ product
random environment that satisfies transience {\rm\eqref{dir-trans}},
 moment assumption  {\rm(\HT)} on the regeneration time,
  bounded step-size hypothesis {\rm(\HM)}, and
the  regularity required by {\rm(\HR)}.
Then for $\P$-almost every
$\w$  distributions $Q_n^\w$  
converge weakly on $D_{\R^d}[0,\infty)$  to the distribution of a Brownian
motion with a diffusion matrix $\kD$ that is independent of
$\w$.  $u^t\kD u=0$ iff $u$ is orthogonal to the span of
$\{x-y: \E(\pi_{0x})\E(\pi_{0y})>0\}$.
\end{theorem}

Equation \eqref{defkD} gives the expression for the diffusion matrix
$\kD$, familiar for example from \citep{szni-00}.

\smallskip

We turn to a discussion of the hypotheses.  Obviously \step 
is only for technical convenience, while \mom and \reg
are the serious assumptions. 

Moment assumption \mom is difficult
to check.  Yet it is a sensible hypothesis because it is
known to follow from many concrete assumptions.   

A RWRE is called  {\sl non-nestling} if for some $\delta>0$ 
\begin{align}
\label{non-nestling}
\P\Bigl\{\sum_{z\in\Z^d} z\cdot\uhat\,\pi_{0,z}\geq\delta\Bigr\}=1.
\end{align}
 This terminology was introduced by Zerner \citep{zernerldp}. 
Together with \astep, non-nestling implies even uniform 
quenched exponential moment bounds on the regeneration times.
See Lemma 3.1 in \citep{rass-sepp-07-b-}.  

Most work on RWRE takes as standing assumptions
 that $\pi_{0,z}$ is supported by the $2d$ nearest neighbors of
the origin,  and  {\sl uniform
ellipticity}:  for some $\kappa>0$,  
\begin{align}
\label{unif-ellipticity}
\P\{\pi_{0,e}\ge\kappa\}=1 \qquad\text{for all unit
vectors $e$.}
\end{align}
Nearest-neighbor jumps with  uniform
ellipticity of course imply Hypotheses (\HM) and (\HR).
In the uniformly elliptic case,
the moment bound  (\HT) on $\tau_1$ follows from the easily
testable  condition  (see \citep{szni-02})  
\[\E\Bigl[\,\Bigl(\,\sum_{z\in\Z^d}z\cdot\uhat\,\pi_{0,z}\Bigr)^+\,\Bigr]
>\kappa^{-1}
\E\Bigl[\Bigl(\,\sum_{z\in\Z^d}z\cdot\uhat\,\pi_{0,z}\Bigr)^-\,\Bigl]. \]

 A more general condition that 
implies Hypothesis (\HT) is {\sl Sznitman's condition} (T'), 
 see Proposition 3.1 in \citep{szni-02}. Condition
(T') cannot be checked by examining the environment $\w_0$
at the origin. 
 But it is still an ``effective'' condition  
in the sense that it can be checked by
examining the environment in finite cubes. Moreover, in condition (T')
the direction vector  $\uhat$ can be replaced by a vector in a 
neighborhood. Consequently the vector can be taken rational, and
then also integral.   Thus our assumption
that $\uhat\in\Z^d$ entails no loss in generality.

 Hypothesis \mom is further justified by a 
currently accepted assumption about uniformly elliptic RWRE. 
Namely,   
it is believed that once a uniformly elliptic 
walk is {\sl ballistic} ($v\ne0$)
the regeneration time has all moments (see \citep{szni-02}). 
Thus conditional on  this supposition, 
 the present work settles the question of quenched
 CLT for  uniformly elliptic, 
 multidimensional ballistic RWRE with bounded steps.

 Hypotheses (\HT) and (\HM) are used throughout the paper. 
Hypothesis (\HR) on the other hand makes
 only one important appearance:  to guarantee the  nondegeneracy of 
a certain Markov chain (Lemma \ref{Yapplm1}  below).
Yet it is Hypothesis (\HR) that is actually necessary for the quenched CLT.

Hypothesis (\HR) can be  violated in two ways:
(a)  the walk
lies in a  one-dimensional linear subspace, or 
(b) assumption \eqref{Yell5} is false 
in which case the walk follows a  sequence of steps 
 completely determined by $\w$ and the only
quenched  randomness 
is in the time taken to leave a site (call this the ``restricted path'' 
case).   In case (b) the walk is bounded if  there is
a chance that the walk intersects itself. This is ruled out by
transience \eqref{dir-trans}.   

In the unbounded situation 
 in case (b) 
  the quenched CLT breaks down 
because the scaled variable   $n^{-1/2}(X_n-nv)$ is not even tight
under $P^\w_0$.  There is still a quenched CLT for the 
walk centered at its quenched mean, that is, for the process
$\Btil_n(t)=n^{-1/2}\{X_{[nt]}-E_0^\w(X_{[nt]})\}$.
Furthermore, the quenched mean itself 
satisfies  a CLT.  Process $B_n$ does satisfy an averaged 
CLT, which comes from the combination of the diffusive fluctuations of 
$\Btil_n$ and of the quenched mean. 
(See \citep{rass-sepp-06} for these results.)
The same situation should hold in one dimension also, and 
has been proved in some cases 
(\citep{gold-07}, \citep{rass-sepp-06}, \citep{zeit-stflour}). 

\medskip

Next a brief discussion of the current
situation in this area of probability and the place of the present 
work in this context.
Several themes appear in
recent work on quenched CLT's for
multidimensional RWRE.

 (i) Small perturbations of classical random walk
have been studied by many authors. The most significant
results include the early work of Bricmont and Kupiainen \citep{bric-kupi-91}
and more recently Sznitman and Zeitouni  \citep{szni-zeit-06}
for small perturbations of Brownian motion in
dimension $d\ge 3$.

(ii) An averaged CLT can be turned into a quenched CLT
by  bounding the variances of quenched expectations
of test functions on the path space.
  This idea was applied 
by Bolthausen and Sznitman \citep{bolt-szni-02} 
to  nearest-neighbor,  uniformly elliptic  non-nestling
 walks in dimension $d\ge 4$  under
a small noise assumption. 
 Berger and Zeitouni  \citep{berg-zeit-07-} developed the approach
further to cover more general ballistic walks without
 the small noise assumption, but still in dimension
$d\ge 4$.  

 After the appearance of the first version of the 
present paper,  Berger and Zeitouni combined some ideas from
our Section \ref{pathintersections} with their own approach to bounding
intersections.  This resulted in an alternative proof of 
Theorem \ref{main} in  the uniformly
elliptic nearest-neighbor case
 that appeared in a revised version of
article \citep{berg-zeit-07-}.  The proof in
\citep{berg-zeit-07-} has the 
virtue that it does not require the ergodic invariant distribution
that we utilize to reduce the proof to a bound on the variance 
of the quenched mean.


   (iii) Our approach is based on the subdiffusivity
of the quenched mean of the walk.  That is, we show that
the variance of $E^\w_0(X_n)$ is of order $n^{2\alpha}$ for some
$\alpha<1/2$.
 This is achieved through intersection bounds. 
We introduced this line of reasoning in \citep{rass-sepp-05},
subsequently applied it to
 walks with a forbidden direction in \citep{rass-sepp-07-a},
and recently to non-nestling walks in \citep{rass-sepp-07-b-}.
Theorem \ref{RS} below summarizes the general principle for
application in the present paper.

It is common in this field to look for an invariant
distribution $\P_\infty$
for the environment process that is mutually absolutely continuous
with the original $\P$,  at least on the
 part of the space $\Omega$ to which the drift points. 
Instead of 
absolute continuity, we use bounds on the variation
distance between $\P_\infty$ and $\P$.  This distance 
decays polynomially in direction $\uhat$, at a rate 
that depends on the strength of the moment assumption \amom.
From this we also get an ergodic theorem for functions 
of the environment that are local in direction $-\uhat$.
This in turn would give the absolute continuity if it were needed 
for the paper.  

\smallskip




The remainder of the paper is for   the proofs. The next section collects 
preliminary material and finishes with an outline of the rest
of the paper.

\medskip

{\sl Acknowledgements.} We thank anonymous referees for thorough readings
of the paper and numerous valuable suggestions.

\section{Preliminaries for the proof}
\label{prelim}
Recall that we assume
$\uhat\in\Z^d$.  This is convenient
because the lattice $\Z^d$ decomposes into {\sl levels}
identified by the integer value   $x\cdot\uhat$.
See Appendix \ref{vectorapp} for the step from a general
$\uhat$ to an integer vector $\uhat$.  

Let us summarize  notation for the reader's convenience.
Constants whose exact values are not important and can change
from line to line are often denoted by $C$.
The set of nonnegative integers is  $\N=\{0,1,2,\dotsc\}$.
  Vectors and
sequences are abbreviated $x_{m,n}=(x_m, x_{m+1},\dotsc,x_n)$
and $x_{m,\infty}=(x_m, x_{m+1}, x_{m+2}, \dotsc)$.  Similar
notation is used for finite and infinite random paths:
 $X_{m,n}$ $=$ $(X_m$, $X_{m+1},$ $\dotsc,$ $X_n)$
and $X_{m,\infty}$ $=$ $(X_m,$ $X_{m+1},$ $ X_{m+2}, \dotsc)$.
$X_{[0,n]}=\{X_k:0\leq k\leq n\}$ denotes the set of sites visited
by the walk.
 ${\mathfrak D}^t$ is the transpose of a vector or matrix
${\mathfrak D}$. An element of $\R^d$ is regarded as a $d\times 1$
 column vector.
 The left shift on the path
space  $(\Z^d)^\N$ is
 $(\theta^kx_{0,\infty})_n = x_{n+k}$.
$\lvert\,\cdot\,\rvert$ denotes Euclidean norm on $\R^d$. 

 $\E$, $E_0$, and $E_0^\w$ denote
expectations under, respectively, $\P$, $P_0$, and $P_0^\w$.
$\P_\infty$ will denote an invariant measure on $\Omega$, with
expectation $\E_\infty$.  Abbreviate
$P^\infty_0(\cdot)=\E_\infty P^\w_0(\cdot)$
and $E^\infty_0(\cdot)=\E_\infty E^\w_0(\cdot)$ to indicate
that the environment of a quenched expectation is averaged
under $\P_\infty$.
A family of  $\sigma$-algebras  on $\Omega$ that
in a sense look towards the future  is defined by
${\kS}_\ell=\sigma\{\w_x: x\cdot \uhat\geq \ell\}$.

Define the {\sl drift}
\[D(\w)=E_0^\w[X_1]=\sum_z z\pi_{0,z}(\w).\]
The {\sl environment process}  is the   Markov chain
on $\Omega$ with transition kernel
\[\Pi(\w,A)=P_0^\w\{T_{X_1}\w\in A\}.\]

The proof of the  quenched CLT  Theorem \ref{main}  utilizes
crucially the environment process and its invariant
distribution. A preliminary part of the proof  is summarized in
the next theorem quoted from  \citep{rass-sepp-05}.  This  Theorem \ref{RS}
was proved by  applying the arguments
 of   Maxwell and Woodroofe \citep{maxw-wood-00} and Derriennic and Lin
\citep{derr-lin-03} to the environment process.

\begin{theorem}  {\rm \citep{rass-sepp-05}}
Let $d\geq1$.  Suppose the  probability measure  $\P_\infty$  on
$(\Omega,{\kS})$ is invariant and ergodic for the Markov
transition  $\Pi$. Assume that
$\sum_z|z|^2\E_\infty[\pi_{0,z}]<\infty$ and  that there
exists an $\alpha<1/2$ such that as $n\to\infty$
\begin{align}
\E_\infty\bigl[\,\abs{E_0^\w(X_n)-n\E_\infty(D)}^2\,\bigr]=\Ord(n^{2\alpha}).
\label{cond}
\end{align}
Then  as $n\to\infty$ the following weak limit happens
 for $\P_\infty$-a.e.\ $\w$:  distributions  $Q_n^\w$
converge weakly  on the space
 $D_{\R^d}[0,\infty)$ to the distribution of a Brownian motion
with a symmetric, nonnegative definite diffusion matrix
$\mathfrak D$
that is independent of $\w$.
\label{RS}\end{theorem}

Proceeding with further definitions, we already defined
above the first 
Sznitman-Zerner
 regeneration time 
 $\tau_1$ as the first time at which 
 \begin{align*}
\label{tau}
\sup_{n<\tau_1} X_n\cdot\uhat<X_{\tau_1}\cdot\uhat\,=\inf_{n\ge\tau_1} X_n\cdot\uhat.
\end{align*}
The first backtracking time is defined by
\beq\beta=\inf\{n\ge0:X_n\cdot\uhat<X_0\cdot\uhat\}.\label{defbeta}\eeq
$P_0$-a.s.\ transience in direction $\uhat$ guarantees that 
\beq P_0(\beta=\infty)>0.\label{transbeta}\eeq
 Otherwise the walk would 
return below level 0 infinitely often
(see Proposition 1.2 in \citep{szni-zern-99}).
Furthermore, a walk transient in direction $\uhat$ will reach 
infinitely many
levels.  At each new level it has  a fresh chance to regenerate. This 
implies that $\tau_1$ is $P_0$-a.s.\   finite 
\citep[Proposition 1.2]{szni-zern-99}.
Consequently we can iterate to  define  $\tau_0=0$, and
for $k\ge 1$
\begin{align*}
\tau_k=\tau_{k-1}+\tau_1\circ\theta^{\tau_{k-1}}.
\label{tauk}
\end{align*}

For i.i.d.\ environments
Sznitman and Zerner \citep{szni-zern-99}  proved that the 
{\sl regeneration slabs}
\beq
\cS_k=
\bigl( \tau_{k+1}-\tau_k,\,
 (X_{\tau_k+n}-X_{\tau_k})_{0\le n\le \tau_{k+1}-\tau_k},\,
\{\w_{X_{\tau_k}+z}: 0\le z\cdot\uhat<(X_{\tau_{k+1}}-X_{\tau_k})\cdot\uhat\}
\bigr)  
\label{regenslab}\eeq
are   i.i.d.\ for $k\ge 1$, each distributed   as
the initial slab  $\bigl( \tau_{1},\,
 (X_{n})_{0\le n\le \tau_{1}},\,
\{\w_{z}: 0 \le z\cdot\uhat<X_{\tau_{1}}\cdot\uhat\}\bigr)$
under $P_0(\,\cdot\,\vert\,\beta=\infty)$.
Strictly speaking, uniform ellipticity and  nearest-neighbor jumps
 were standing assumptions in \citep{szni-zern-99}, but these assumptions
are not needed for the proof of the i.i.d.\ structure.
From this and assumptions \eqref{dir-trans} and \mom it then follows
for $k\ge 1$  that 
\beq
E_0\bigl[(\tau_{k+1}-\tau_k)^\momp\bigr]=
E_0[\tau_1^\momp\,\vert\, \beta=\infty]\le
\frac{E_0(\tau_1^\momp)}{P_0(\beta=\infty)}<\infty.
\label{taumom2}
\eeq

From the renewal structure and moment estimates
a law of large numbers \eqref{lln1} and an averaged functional central
limit theorem follow, along the lines of Theorem 2.3
in \citep{szni-zern-99} and
Theorem 4.1 in \citep{szni-00}.  These references treat walks that
satisfy Kalikow's condition,  less general than Hypothesis (\HT).
But the proofs only rely on the existence of moments of $\tau_1$,
now ensured by Hypothesis (\HT).
   The limiting velocity for the
law of large numbers  is
\beq
v=\frac{E_0[X_{\tau_1}\vert\beta=\infty]}{E_0[{\tau_1}\vert\beta=\infty]}.
\label{defv}\eeq
The averaged CLT states that the distributions  $P_0\{B_n\in\,\cdot\,\}$
converge to the distribution of a Brownian motion with
diffusion matrix
\beq
\kD=\frac{E_0\bigl[(X_{\tau_1}-\tau_1v)(X_{\tau_1}-\tau_1v)^t
\big\vert\beta=\infty\bigr]}
{E_0[{\tau_1}\vert\beta=\infty]}.
\label{defkD}
\eeq

Once we know
 that the $\P$-a.s.\ quenched CLT holds with a constant diffusion matrix,
 this diffusion matrix must be the same $\kD$  as for the
averaged CLT.
We prove here  the degeneracy statement of
Theorem \ref{main}.

\begin{lemma}  Define  $\kD$  by {\rm \eqref{defkD}} and let $u\in\R^d$.
Then
 $u^t\kD u=0$ iff $u$ is orthogonal to the span of
$\{x-y: \E[\pi_{0,x}]\E[\pi_{0,y}]>0\}$.
\label{degen-lm} 
\end{lemma}

\begin{proof} The argument is a minor embellishment of that given
for a similar degeneracy  statement
on p.~123--124  of \citep{rass-sepp-06} for the forbidden-direction
 case where $\pi_{0,z}$ is supported by $z\cdot\uhat\ge 0$.
We spell out enough  of the argument  to show
 how to adapt that proof to the present case.

Again, the intermediate step is to show that  $u^t\kD u=0$
iff $u$ is orthogonal to the span of
$\{x-v: \E[\pi_{0,x}]>0\}$.  The argument from orthogonality
to  $u^t\kD u=0$  goes as in \citep[p.~124]{rass-sepp-06}.

Suppose $u^t\kD u=0$ which is the same as
\beq  P_0\{X_{\tau_1}\cdot u
=\tau_1 v\cdot u \,\vert\,\beta=\infty\}=1. \label{degen-aux}\eeq
Take $x$ such that $\E\pi_{0,x}>0$. Several cases need to be
considered.  

If $x\cdot\uhat\ge0$ but $x\ne0$ a small modification of the
argument in \citep[p.\ 123]{rass-sepp-06} works to show that 
$x\cdot u=v\cdot u$.

Suppose $x\cdot\uhat<0$. Then take $y$ such 
that $y\cdot\uhat>0$ and $\E\pi_{0,y}>0$.  Such $y$ must exist by
the transcience assumption \eqref{dir-trans}. 

If $y$ is collinear with $x$ and there is no other noncollinear 
vector $y$ with $y\cdot\uhat>0$, then, since the one-dimensional
case is
excluded by Hypothesis (\HR), there must
exist another vector $z$ that is not collinear with $x$ or $y$ and such
that $z\cdot\uhat\le0$ and $\E\pi_{0,z}>0$.

Now for any $n\ge1$, let $m_n$ be the positive integer such that 
\[(m_ny+2z+nx)\cdot\uhat\ge0
\quad\text{but}\quad  ((m_n-1)y+2z+nx)\cdot\uhat<0. \] 
Let the walk first take $m_n$ $y$-steps, followed by one $z$-step,
then $n$ $x$-steps, followed by 
another $z$-step, 
again $m_n$ $y$-steps, and then regenerate (meaning that 
$\beta\circ\theta^{2m_n+n+2}=\infty$).
This path is non-self-intersecting and, by the minimality of $m_n$, 
backtracks enough
to ensure that the first regeneration time is $\tau_1=2m_n+n+2$. Hence 
\begin{align*}
&P_0\{X_{\tau_1}=2m_ny+nx+2z, \tau_1=2m_n+n+2\,\vert\,\beta=\infty\}
\ge (\E\pi_{0,y})^{2m_n}(\E\pi_{0,x})^n(\E\pi_{0,z})^2>0
\end{align*}
and then by \eqref{degen-aux} 
\beq (nx+2m_ny+2z)\cdot u=(n+2m_n+2)v\cdot u. \label{degenaux1}\eeq
Since $y\cdot\uhat>0$ we have already shown that $y\cdot u=v\cdot u$.
Taking $n\nearrow\infty$ implies $x\cdot u=v\cdot u$.

If $y$ is not collinear with $x$, repeat the above argument, but without using 
any $z$-steps and hence with simply $n=1$.

When $x=0$ making the walk take an extra step of size 0 along the path, an almost
identical argument to the above can be repeated.
Since we have shown that $y\cdot u=v\cdot u$ for any $y\ne0$
with $\E\pi_{0,y}>0$, this allows to also conclude that $0\cdot u=v\cdot u$.
%
%
%

Given  $u^t\kD u=0$, 
we have established  $x\cdot u=v\cdot u$  for any $x$ with $\E\pi_{0,x}>0$.
Now follow the
proof in \citep[p.~123--124]{rass-sepp-06} to its conclusion.
\end{proof}

Here is  an outline of the proof of
Theorem \ref{main}.  It all goes via Theorem \ref{RS}.

\smallskip

(i) After some basic estimates in Section \ref{ballistic},
we prove in Section \ref{Pinfty_exists}
 the existence of the ergodic invariant distribution $\P_\infty$ required
for Theorem \ref{RS}.  $\P_\infty$ is not convenient to work with
 so we still need  to do  computations with
$\P$. For this purpose
Section \ref{Pinfty_exists} proves that
 in the direction $\uhat$
the measures $\P_\infty$ and $\P$ come polynomially close in variation
distance and that the environment process satisfies
 a $P_0$-a.s.\ ergodic theorem.
 In Section \ref{substitution}  we show that $\P_\infty$
and $\P$  are interchangeable both in the hypotheses that
need to be checked
and in the conclusions obtained.
In particular, the $\P_\infty$-a.s.\ quenched CLT
coming from Theorem \ref{RS} holds also $\P$-a.s.  Then we know
that the diffusion matrix $\kD$ is the one in \eqref{defkD}.

\smallskip

 The bulk of the work goes towards verifying condition
 \eqref{cond}, but under $\P$ instead of $\P_\infty$.
There are two main stages to this argument.

\smallskip

(ii) By a decomposition into martingale increments
the proof of \eqref{cond}  reduces to bounding the number of common
points of two independent walks in a common environment
(Section \ref{pathintersections}).

\smallskip

(iii) The intersections are controlled by introducing levels
at which both walks regenerate. These
joint regeneration levels are reached fast enough and
the relative positions of the walks from one
joint regeneration level to the next are a Markov chain.
When this Markov chain drifts away from the origin it
 can be approximated well enough by a symmetric random
walk.  This approximation enables us to control the growth
of the Green function of the Markov chain, and thereby the number of
common points. This is in Section \ref{intersectbound} and
in Appendix \ref{greenapp} devoted to  the Green function bound.

Appendix \ref{vectorapp} shows that the assumption  that
 $\uhat$  has integer coordinates entails no
loss of generality if the moment  required is doubled.
The proof given in Appendix \ref{vectorapp} is from 
Berger and Zeitouni \cite{berg-zeit-07-}.
Appendix \ref{Yapp} contains a proof (Lemma \ref{Yapplm1}) that 
requires a systematic enumeration of a large number of cases.  

The end result  of the development is the bound 
\begin{align}
\E\bigl[\,\abs{E_0^\w(X_n)-E_0(X_n)}^2\,\bigr]=\Ord(n^{2\alpha})
\label{cond1aa}
\end{align}
on the variance of the quenched mean, for some $\alpha\in(1/4,1/2)$. 
The parameter $\alpha$ can be taken arbitrarily close to $1/4$ 
if the exponent $\momp$ in \mom can be taken arbitrarily large.  
The same is also true under the invariant measure $\P_\infty$, 
namely   \eqref{cond} is valid for  some $\alpha\in(1/4,1/2)$. 
Based on the behavior of the Green function of a symmetric random walk,
optimal orders in \eqref{cond1aa} should be $n^{1/2}$ in $d=2$, $\log n$
in $d=3$, and constant in  $d\ge4$.
Getting an optimal bound 
in each dimension is not a present goal,  
so in the end we bound all dimensions with the two-dimensional case. 
 
The requirement $\momp> 176d$ of Hypothesis
\mom  is derived from the
bounds established along the way. There is room in the estimates
for we take one simple and lax route to a sufficient bound. 
 Start from 
\eqref{appAfinal} with $p_1=p_2=\momp/6$ as dictated by 
Proposition \ref{qqbarprop4} and \eqref{defhfunction}.   
Taking $\momp=220$ gives the bound $Cn^{22/32}$. 
Feed this bound into Proposition \ref{intersections} 
where it sets $\bar\alpha=11/32$. 
Next in \eqref{momp-cond6} take $\alpha-\bar\alpha=1/8$ 
to get the requirement $\momp>176d$. Finally 
 in \eqref{momp-cond5} take $\alpha-\bar\alpha=1/32$ which
places the demand  $\momp>160d$. With $d\ge 2$ all are
satisfied with  $\momp>176d$.  (Actually $11/32+1/8+1/32=1/2$
but since the inequalities are strict there is room to keep
$\alpha$ strictly below $1/2$.) 

Sections \ref{ballistic}--\ref{pathintersections} are valid
for all  dimensions $d\ge 1$, but Section \ref{intersectbound}
requires $d\ge 2$.  

\section{Basic estimates for ballistic RWRE}
\label{ballistic}
In addition to the regeneration  times
already defined, let 
\[J_m = \inf\{ i\ge0: \tau_i \ge m\}.\]

\begin{lemma}
\label{exponential}
Let $\,\P$ be an i.i.d.\ product measure and satisfy
 Hypotheses {\rm(\HM)} and {\rm(\HT)}. 
We have these bounds: 
\begin{align}
&E_0[\tau_\ell^\momp]\le C\ell^\momp\quad\text{for all $\ell\ge 1$.} 
\label{tau-bound}\\
&\sup_{m\ge0} E_0[\,|\tau_{J_m}-m|^p\,]\le C
\quad\text{ for $1\le p\le \momp-1$.}
\label{renewal-tau-bound}\\
&\sup_{m\ge0}E_0[\,|\inf_{n\ge0}(X_{m+n}-X_m)\cdot\uhat|^p\,]\le C
\quad\text{ for $1\le p\le \momp-1$.}
\label{backtrack-bound}\\
&\sup_{m\ge0}P_0\{(X_{n+m}-X_m)\cdot\uhat\le \sqrt n\,\}\le C n^{-p}
\quad\text{ for $1\le p\le (\momp-1)/2$.}
\label{ldp}
\end{align}
\end{lemma}

\begin{proof}
\eqref{tau-bound} follows from \eqref{taumom2} and 
 Jensen's inequality. 

\def\frt{g}  

The proof of \eqref{renewal-tau-bound} comes by a renewal argument.
Let $Y_j=\tau_{j+1}-\tau_{j}$ for $j\ge 1$ and $V_0=0$, 
 $V_m=Y_1+\dotsm+Y_m$.  
The forward recurrence time  of this pure renewal process is 
$\frt_n=\min\{k\ge 0: n+k\in \{V_m\}\}$. A decomposition 
according to the value of  $\tau_1$ gives
\beq
\tau_{J_n}-n = (\tau_1-n)^+  +   \sum_{k=1}^{n-1} \one\{\tau_1=k\} 
\frt_{n-k}.
\label{renewaux1}\eeq

First we bound the moment of $\frt_n$.  For this 
 write a renewal equation  
\[
\frt_n = (Y_1-n)^+  +   \sum_{k=1}^{n-1} \one\{Y_1=k\} \frt_{n-k}\circ\theta
\]
where
 $\theta$ shifts the sequence $\{Y_k\}$ 
so that   $\frt_{n-k}\circ\theta$ is  independent of $Y_1$. 
Only one term on the right can be nonzero,  so for any $p\ge1$
\[
\frt_n^p = ((Y_1-n)^+)^p  +
          \sum_{k=1}^{n-1} \one\{Y_1=k\} (\frt_{n-k}\circ\theta)^p.
\]
Set $z(n) = E_0[((Y_1-n)^+)^p\,]$.  Assumption $p\le \momp-1$ and
\eqref{taumom2}  
give 
$E_0[Y_1^{p+1}]$ $<\infty$ which 
implies $\sum z(n) < \infty$.  Taking expectations and using independence
gives the equation 
\[
E_0\frt_n^p = z(n)  +   \sum_{k=1}^{n-1} P_0\{Y_1=k\} E_0\frt_{n-k}^p.
\]
Induction on $n$ shows that   
\[ E_0\frt_n^p \le \sum_{k=1}^n z(k) \le C
\quad\text{ for all $n$.} \]
  Raise \eqref{renewaux1} to the power $p$,   take expectations, 
use Hypothesis \amom,  and 
substitute this last bound in there to complete the proof of  
  \eqref{renewal-tau-bound}. 

\eqref{backtrack-bound} follows readily.  Since
 the walk does not backtrack after time $\tau_{J_m}$
  and steps are bounded
by Hypothesis (\HM), 
\[|\inf_{n\ge0}(X_{m+n}-X_m)\cdot\uhat|
=\Big|\inf_{n: m\le n\le\tau_{J_m}}(X_n-X_m)\cdot\uhat\Big|
\le \M|\uhat|(\tau_{J_m}-m).\]
Apply  \eqref{renewal-tau-bound} to this last quantity.

Lastly we show \eqref{ldp}.  For $a<b$  define 
\[V_{a,b}=\sum_{i\ge1}\one\{a<\tau_i<b\}.\]
Then $(X_{m+n}-X_m)\cdot\uhat\le \sqrt n$ implies
$V_{m,m+n}\le \sqrt n$.  Recall  the i.i.d.\ structure of slabs
$({\mathcal S}_k)_{k\ge1}$ defined in \eqref{regenslab}.  
For the first inequality note that  either there are
no regeneration times in $[m,m+n)$, or there is  one and we restart
at the first one. 
\begin{align*}
&P_0\{V_{m,m+n}\le\sqrt n\,\}\\
&
\le P_0\{\tau_{J_m}-m\ge n\} 
+P_0\{V_{0,n}\le\sqrt n\,\vert\,\beta=\infty\}
+\sum_{k=1}^{n-1} 
P_0\{\tau_{J_m}-m=k\}P_0\{V_{0,n-k}\le \sqrt n-1\,\vert\,\beta=\infty\}\\
&
\le P_0\{\tau_{J_m}-m\ge n\}
+P_0\{\tau_{[\sqrt n\,]+1}\ge n\,\vert\,\beta=\infty\}
+C\sum_{k=1}^{n-1} 
k^{-2p}P_0\{\tau_{[\sqrt n\,]}\ge n-k\,\vert\,\beta=\infty\}\\
&
\le \frac{C}{n^p}
+C n^{p}\sum_{k=1}^{n-1} \frac1{k^{2p}{(n-k)^{2p}}}
\le \frac C{n^p}.
\end{align*}
We used \eqref{renewal-tau-bound} in the second inequality and
then again in the third inequality, along with \eqref{tau-bound}.
For the last inequality split the sum according to
 $k\le n/2$ and $k>n/2$,  in the former case 
bound $1/(n-k)$ by $2/n$, and in the latter case bound
$1/k$ by $2/n$.
\end{proof}

\section{Invariant measure and ergodicity}
\label{Pinfty_exists}
For integers $\ell$ define the $\sigma$-algebras 
${\kS}_\ell=\sigma\{\w_x: x\cdot \uhat\geq \ell\}$ on $\Omega$. 
 Denote the restriction of the measure $\P$ to the
 $\sigma$-algebra ${\kS}_\ell$ by $\P_{\vert{\kS}_\ell}$. 
In this section we prove the next two theorems.  The 
variation distance of two probability measures is
$\vard(\mu,\nu)=\sup\{\mu(A)-\nu(A)\}$ with the supremum taken
over measurable sets $A$. 
 $\E_\infty$
denotes expectation under the invariant measure $\P_\infty$
whose existence is established below. The corresponding joint measure on 
environments and paths is denoted by 
 $P_0^\infty(d\w, dx_{0,\infty})=\P_\infty(d\w) P_0^\w(dx_{0,\infty})$
with expectation $E_0^\infty$.

\begin{theorem}
\label{Th_exist}
Assume $\P$ is product and  satisfies Hypotheses {\rm(\HM)}
and {\rm(\HT)}, with $\momp>4d+1$.
Then
there exists a probability measure $\P_\infty$ on $\Omega$ with
these properties.
\begin{enumerate}
\item[{\rm (a)}] Hypothesis {\rm (\HM)} holds $\P_\infty$-almost surely. 
\item[{\rm (b)}] $\P_\infty$ is invariant and ergodic
 for the Markov  transition kernel $\Pi$.
\item[{\rm (c)}] 
For all $\ell\ge 1$ 
\begin{align}
\label{vard-exp}
\vard({\P_\infty}_{\vert{\kS}_\ell},\P_{\vert{\kS}_\ell})\le C\ell^{1-\momp}.
\end{align}
\item[{\rm (d)}]   Under  $P_0^\infty$ the walk 
 has these properties:
\begin{enumerate}
\item[{\rm (e.i)}]  For $1\le p\le\momp-1$  
\begin{align}
E_0^\infty\Big[\,\bigl\lvert\inf_{n\ge0}X_n\cdot\uhat\bigr\rvert^p\,\Big]
\le C.
\label{backtrack-bound-invariant}
\end{align}
\item[{\rm (e.ii)}] 
For $1\le p\le(\momp-1)/2$ and
$n\ge1$,
\begin{align}
&P_0^\infty\{X_n\cdot\uhat\le n^{1/2}\}\le Cn^{-p}.
\label{ldp-invariant}
\end{align}
\end{enumerate}
\end{enumerate}
\end{theorem}

More could be said about $\P_\infty$. For example, following 
\citep{szni-zern-99}, one can show that $\P_\infty$ comes  as
a limit, and has a renewal-type representation that involves the
regeneration times.  But we cover only properties  needed in the
sequel.
Along the way we establish this ergodic theorem under the 
original environment measure. 

\begin{theorem} 
\label{erg-thm}  Assumptions as in the above Theorem {\rm\ref{Th_exist}}.
Let $\Psi$ be a bounded $\kS_{-a}$-measurable function 
on $\Omega$, for some $0<a<\infty$.  Then 
\beq
\lim_{n\to\infty} 
 n^{-1}\sum_{j=0}^{n-1} \Psi(T_{X_j}\w) 
 =  \E_\infty \Psi
\quad\text{$P_0$-almost surely.}
\label{erg-P}\eeq
\end{theorem} 

Theorem \ref{erg-thm} tells us that there is a unique invariant $\P_\infty$ 
in a natural relationship to $\P$, and also 
gives the absolute continuity 
${\P_\infty}_{|\kS_{-a}}\ll\P_{|\kS_{-a}}$. 
  Limit  \eqref{erg-P} cannot hold
for all bounded measurable $\Psi$  on $\Omega$ because this would 
imply the absolute continuity  $\P_\infty\ll\P$ on the entire space
$\Omega$.  A  counterexample that satisfies (\HT) and (\HM)
but where the quenched walk is degenerate
was given by Bolthausen and Sznitman \citep[Proposition 1.5]{bolt-szni-02}.
Whether regularity  assumption (\HR) or ellipticity
will make a difference here is not presently clear. 
For the simpler case of  space-time walks (see description of
model  in \citep{rass-sepp-05}) 
 with nondegenerate
$P^\w_0$ absolute continuity $\P_\infty\ll\P$ does hold on the 
entire space. Theorem 3.1 in \citep{bolt-szni-02} proves this 
for nearest-neighbor jumps with some weak ellipticity. 
 The general case is no harder. 

\begin{proof}[Proof of Theorems \ref{Th_exist} and \ref{erg-thm}] 
Let  $\P_n(A)=P_0\{T_{X_n}\w\in A\}$. A computation shows
that \[f_n(\w)=\frac{d\P_n}{d\P}(\w)=\sum_x P_x^\w\{X_n=0\}.\]

By Hypothesis (\HM) we can replace the state space 
$\Omega=\mathcal P^{\Z^d}$ with  the compact space 
$\Omega_0={\mathcal P}_{0}^{\Z^d}$ where 
\beq {\mathcal P}_{0}=
\{(p_z)\in\mathcal P:
\text{$p_z=0$ if $|z|>\M$} \}.
\label{defP0}
\eeq

Compactness gives  a subsequence $\{n_j\}$ along which
${n_j}^{-1}\sum_{m=1}^{n_j}\P_m$ converges weakly to a probability measure
$\P_\infty$ on $\Omega_0$.   Hypothesis (\HM) transfers to 
$\P_\infty$ by virtue of having been included in the 
state space $\Omega_0$. 
We have verified part (a) of Theorem \ref{Th_exist}. 

Due to Hypothesis (\HM) $\Pi$ is Feller-continuous.
Consequently the weak limit 
${n_j}^{-1}\sum_{m=1}^{n_j}\P_m\to\P_\infty$   together with 
$\P_{n+1}=\P_n\Pi$ implies the 
 $\Pi$-invariance of $\P_\infty$.

Next we derive the bound on the variation distance.
On metric spaces  total variation distance
can be  characterized  in terms of continuous functions: 
\[\vard(\mu,\nu)=\frac12
\sup\Big\{\int f d\mu-\int f d\nu:f\text{ continuous},\ \sup|f|\le1\Big\}.\]  This makes $\vard(\mu,\nu)$ lower semicontinuous
which we shall find convenient below. 

Fix $\ell>0$. Then 
\beq
\frac{d{\P_n}_{|{\kS}_\ell}}{d{\P}_{|{\kS}_\ell}}
=
\E\big[\sum_x P_x^\w\{X_n=0,\,\max_{j\leq n} X_j\cdot\uhat\leq \ell/2\}
\big\vert {\kS}_\ell\big]
+\sum_x\E[P_x^\w\{X_n=0,\,\max_{j\leq n} X_j\cdot\uhat>\ell/2\}|{\kS}_\ell].
\label{RNder}\eeq
The $L^1(\P)$-norm of the second term is 
\begin{align*}
&\sum_x P_x\{X_n=0,\,\max_{j\leq n} X_j\cdot\uhat>\ell/2\} 
=P_0\{\max_{j\leq n} X_j\cdot\uhat>X_n\cdot\uhat+\ell/2\}\equiv I_{n,\ell}.
\end{align*}
The integrand in the first term on the right-hand side of \eqref{RNder}
 is measurable with
respect to $\sigma(\w_x:x\cdot\uhat\le \ell/2)$ and therefore independent of 
$\kS_{\ell}$. So this term is equal to the  nonrandom constant
\begin{align*}
&\sum_x P_x\{X_n=0,\,\max_{j\leq n} X_j\cdot\uhat\le \ell/2\} \\
&\qquad = 1 - 
P_0\{\max_{j\leq n} X_j\cdot\uhat>X_n\cdot\uhat+\ell/2\}\\
&\qquad = 1 - I_{n,\ell}.
\end{align*}
Altogether, 
\begin{align*}
\vard({\P_n}_{|{\kS}_\ell},{\P}_{|{\kS}_\ell})\leq\tfrac12
\int\Bigl\lvert\frac{d{\P_n}_{|{\kS}_\ell}}{d{\P}_{|{\kS}_\ell}}-1\Bigr\rvert d\P
\leq I_{n,\ell}.
\end{align*}
Now write
\begin{align*}
&\frac1n\sum_{k=1}^n I_{k,\ell}
=\frac1n\sum_{k=1}^n 
P_0\{\max_{j\le k}X_j\cdot\uhat >  X_k\cdot\uhat  + \ell/2\}\\
&\le  \frac1n E_0\Big[\sum_{k=1}^{\tau_1\wedge n}
\one\big\{\max_{j\le k}X_j\cdot\uhat >  X_k\cdot\uhat  + \ell/2\big\}\Big]
+\frac1n\sum_{k=2}^n 
E_0[(\tau_k-\tau_{k-1})\one\{X_{\tau_k}\cdot\uhat-X_{\tau_{k-1}}\cdot\uhat> \ell/2\}]\\
&\le n^{-1} E_0[\tau_1\wedge n]+\frac{n-1}n 
E_0[\tau_1\one\{\tau_1>\ell/2\M\}|\beta=\infty]\\
&\le Cn^{-1}+ C\ell^{1-\momp}.
\end{align*}
The last inequality came from Hypothesis \mom and H\"older's 
inequality.  Let $n\to\infty$ along the relevant subsequence
and use 
lower semicontinuity and convexity of the variation distance. 
This proves part (c).

Concerning backtracking: notice first that due to  \eqref{backtrack-bound}
we have
\begin{align*}
\E_k[E_0^\w(|\inf_{n\ge0}X_n\cdot\uhat|^p)]
&=E_0[E_0^{T_{X_k}\w}(|\inf_{n\ge0}X_n\cdot\uhat|^p)]
=E_0[|\inf_{n\ge0}(X_{n+k}-X_k)\cdot\uhat|^p]
\le C_p.
\end{align*}
Since $E_0^\w(|\inf_{0\le n\le N}X_n\cdot\uhat|^p)$ is a continuous function
of $\w$, the definition of $\P_\infty$ along with the above estimate
and monotone convergence imply \eqref{backtrack-bound-invariant}.
(e.i) has been proved.

Write once again, using \eqref{ldp}
\begin{align*}
\E_k[P_0^\w\{X_n\cdot\uhat\le \sqrt n\}]
&=E_0[P_0^{T_{X_k}\w}\{X_n\cdot\uhat\le \sqrt n\}]
=P_0\{(X_{n+k}-X_k)\cdot\uhat\le \sqrt n\}
\le Cn^{-p}.
\end{align*}
Since $P_0^\w\{X_n\cdot\uhat\le\sqrt n\}$ is a continuous function
of $\w$, the definition of $\P_\infty$ along with the above estimate
imply \eqref{ldp-invariant} and proves (e.ii).

As the last point 
we prove the ergodicity. 
Let $\Psi$ be a bounded local function on $\Omega$. It suffices
to prove that for some constant $\conb$  
\beq
\lim_{n\to\infty} 
E^\infty_0 \Bigl\lvert n^{-1}\sum_{j=0}^{n-1} \Psi(T_{X_j}\w) 
-\conb\,\Bigr\rvert = 0.
\label{ergaux1}
\eeq
By an approximation it follows from this that for all $F\in L^1(\P_\infty)$
\beq
n^{-1}\sum_{j=0}^{n-1} \Pi^jF(\w)  \rightarrow \E_\infty F 
\quad\text{in $L^1(\P_\infty)$.}
\label{erqaux2}
\eeq
  By standard theory (Section IV.2 in \citep{rose})
this is equivalent to ergodicity of $\P_\infty$ for the transition $\Pi$. 

We combine the proof of Theorem \ref{erg-thm} with the proof of
 \eqref{ergaux1}.  For this purpose let
$a$ be a positive integer and  $\Psi$ a bounded $\kS_{-a+1}$-measurable
function.   Let 
\[
\varphi_i=\sum_{j=\tau_{ai}}^{\tau_{a(i+1)}-1} \Psi(T_{X_j}\w).
\]
From the  i.i.d.\ regeneration slabs and the moment
bound \eqref{tau-bound} follows the limit 
\beq
\lim_{m\to\infty} 
 m^{-1}\sum_{j=0}^{\tau_{am}-1}   \Psi(T_{X_j}\w) 
=
\lim_{m\to\infty} 
 m^{-1}\sum_{i=0}^{m-1} \varphi_i 
=\cona \qquad\text{$P_0$-almost surely,}
\label{ergaux3} \eeq
where the constant $\cona$ is defined by the limit.  

To justify limit \eqref{ergaux3}  more explicitly,  recall the 
 definition of regeneration slabs given in \eqref{regenslab}.
  Define  a function  $\Phi$
of the regeneration slabs  by
\[
\Phi(\cS_0,\cS_1,\cS_2,\dotsc)=\sum_{j=\tau_a}^{\tau_{2a}-1}
\Psi(T_{X_{j}}\w).
\]
 Since each regeneration slab has thickness in $\uhat$-direction
 at least 1, the $\Psi$-terms in the sum do not read the environments
below level zero and consequently the sum is a function of 
$(\cS_0,\cS_1,\cS_2,\dotsc)$.
Next one can check for $k\ge 1$ that 
\begin{align*}
&\Phi(\cS_{a(k-1)},\cS_{a(k-1)+1},\cS_{a(k-1)+2},\dotsc)\\
&\qquad=
\sum_{j=\tau_a(X_{\tau_{a(k-1)}+\,\centerdot\,}-X_{\tau_{a(k-1)}})}^{\tau_{2a}
(X_{\tau_{a(k-1)}+\,\centerdot\,}-X_{\tau_{a(k-1)}})-1} \!\!\!\!\!\!
\Psi\bigl(T_{X_{\tau_{a(k-1)}+j}-X_{\tau_{a(k-1)}}}(T_{X_{\tau_{a(k-1)}}}\w)\bigr)
=\varphi_k. 
\end{align*}
Now the sum
of $\varphi$-terms in \eqref{ergaux3} can be decomposed into
\[
\varphi_0+\varphi_1+ 
\sum_{k=1}^{m-2}\Phi(\cS_{ak},\cS_{ak+1},\cS_{ak+2},\dotsc). 
\]
The limit \eqref{ergaux3} follows because the slabs  $(\cS_k)_{k\ge 1}$ are 
i.i.d.\ and the finite initial terms 
 $\varphi_0+\varphi_1$ are eliminated
 by the $m^{-1}$ factor. 

Let $\alpha_n=\inf\{k: \tau_{ak}\ge n\}$. 
 Bound \eqref{tau-bound} 
 implies that $n^{-1}(\tau_{a(\alpha_n-1)}-\tau_{a\alpha_n})
\to 0$ $P_0$-almost surely.
Consequently \eqref{ergaux3} 
yields the next limit,   for another
constant $\conb$: 
\beq
\lim_{n\to\infty} 
 n^{-1}\sum_{j=0}^{n-1}   \Psi(T_{X_j}\w) 
= \conb \qquad\text{$P_0$-almost surely.}
\label{ergaux5}
\eeq
By boundedness this limit 
is valid also in  $L^1(P_0)$ and the initial
point of the walk is immaterial by shift-invariance of $\P$.
Let $\ell>0$ and  abbreviate 
\[
G_{n,x}(\w)=E^\w_x\Bigl[\;
\Bigl\lvert n^{-1}\sum_{j=0}^{n-1} \Psi(T_{X_j}\w) 
-\conb\,\Bigr\rvert \one\bigl\{\,\inf_{j\ge 0}X_j\cdot\uhat\ge 
X_0\cdot\uhat-\ell^{1/2}/2\bigr\}\,\Bigr].
\]
  Let \[\cI=\{x\in\Z^d: x\cdot\uhat\ge \ell^{1/2},
\abs{x}\le \M\ell\}.\]
If $\ell$ is large enough
relative to $a$, then  for $x\in\cI$ 
the function $G_{n,x}$ is $\kS_{\ell^{1/2}/3}$-measurable.
Use the bound \eqref{vard-exp} on the variation
distance and the fact that  the functions $G_{n,x}(\w)$ are 
uniformly bounded over all $x,n,\w$.  
\begin{align*}
&\P_\infty\Bigl\{ \;\sum_{x\in \cI}P^\w_0[X_\ell=x] G_{n,x}(\w)\ge \e_1
\Bigr\}
\le \sum_{x\in \cI} \P_\infty\{ G_{n,x}(\w)\ge \e_1/(C\ell^d)\} \\
&\le  C\ell^d \e_1^{-1}  \sum_{x\in \cI} \E_\infty G_{n,x} 
\le  C\ell^d \e_1^{-1}  \sum_{x\in \cI} \E G_{n,x} 
+  C\ell^{2d} \e_1^{-1} \ell^{(1-\momp)/2}.
\end{align*}
By \eqref{ergaux5} $\E G_{n,x}\to 0$ for any fixed $x$.
Thus from above we get for any fixed $\ell$, 
\begin{align*}
\varlimsup_{n\to\infty} E^\infty_0\bigl[ \,\one\{X_\ell\in\cI\}
G_{n,X_\ell} \bigr] \le \e_1 + C\ell^{2d} \e_1^{-1}\ell^{(1-\momp)/2}.
\end{align*}
The reader should bear in mind that the constant $C$ is 
changing from line to line.  Finally, take $p\le(\momp-1)/2$ and
use \eqref{backtrack-bound-invariant} and \eqref{ldp-invariant} to write 
\begin{align*}
&\varlimsup_{n\to\infty} 
E^\infty_0 \Bigl\lvert n^{-1}\sum_{j=0}^{n-1} \Psi(T_{X_j}\w) 
-\conb\,\Bigr\rvert\\
&\le \varlimsup_{n\to\infty} 
E^\infty_0 \Bigl[ \one\{X_\ell\in\cI\}
\Bigl\lvert n^{-1}\!\!\!\sum_{j=\ell}^{n+\ell-1}\!\! \Psi(T_{X_j}\w) 
-\conb\Bigr\rvert\,
\one\bigl\{\inf_{j\ge \ell}X_j\cdot\uhat\ge 
X_\ell\cdot\uhat-\ell^{1/2}/2\bigr\} \Bigr]\\
&\quad\qquad  +\; C P^\infty_0 \{X_\ell\notin\cI\}
\;+\; C P^\infty_0 \bigl\{\inf_{j\ge \ell}X_j\cdot\uhat<
X_\ell\cdot\uhat-\ell^{1/2}/2\bigr\}\\
&\le \varlimsup_{n\to\infty} E^\infty_0\bigl[ \,\one\{X_\ell\in\cI\}
G_{n,X_\ell} \bigr] \\
&\quad\qquad
\;+\;
C P^\infty_0 \{X_\ell\cdot\uhat<\ell^{1/2}\}
\;+\;
C P_0^\infty \bigl\{\inf_{j\ge0}X_j\cdot\uhat< -\ell^{1/2}/2\bigr\}\\
&\le 
\e_1 + C\ell^{2d} \e_1^{-1}\ell^{(1-\momp)/2}+C\ell^{-p}+C\ell^{-p/2}.
\end{align*} 
Consequently, if we first pick $\e_1$ 
small enough  then $\ell$ large, 
we will have shown \eqref{ergaux1}. For the second term on the
last line we need $\momp>4d+1$.
  Ergodicity of $\P_\infty$
has been shown.  
This concludes the proof of Theorem \ref{Th_exist}. 

Theorem \ref{erg-thm} has also been established. It follows from
the combination of \eqref{ergaux1} and \eqref{ergaux5}.
\end{proof}

\section{Change of measure}
\label{substitution}
There are several stages in the proof where we need to
check that a desired conclusion is not affected by
choice between  $\P$ and  $\P_\infty$.  We collect
all instances of such transfers in this section.
The standing assumptions of this section are that
 $\P$ is an i.i.d.\ product measure that   satisfies
 Hypotheses {\rm(\HT)} and {\rm(\HM)}, and that
 $\P_\infty$ is the measure given by Theorem {\rm \ref{Th_exist}}.
We  show first that  $\P_\infty$ can be
 replaced with $\P$  in the key condition \eqref{cond} of
Theorem \ref{RS}.

\begin{lemma}
\label{velocity}
The velocity $v$ defined by {\rm\eqref{defv}} satisfies
$v=\E_\infty(D)$.
There exists a constant C such that
\begin{align}
\label{velocity-bound}
|E_0(X_n)-n\E_\infty(D)|\le C \qquad\text{for all $n\ge 1$.}
\end{align}
\end{lemma}

\begin{proof} We start by showing $v=\E_\infty(D)$.
The finite step-size condition in the definition of
\eqref{defP0} of $\cP_0$ makes the function $D(\w)$ bounded and continuous on
$\Omega_0$.  
By the Ces\`aro definition of $\P_\infty$,
\[
\E_\infty(D) = \lim_{j\to\infty} \frac1{n_j}\sum_{k=0}^{n_j-1} \E_k(D)
 = \lim_{j\to\infty} \frac1{n_j}\sum_{k=0}^{n_j-1} E_0[D(T_{X_k}\w)].
\]
Hypothesis (\HM) implies 
that the law of large numbers $n^{-1}X_n\to v$ holds also in $L^1(P_0)$.
From this and the Markov property
\[
v=\lim_{n\to\infty} \frac1n \sum_{k=0}^{n-1} E_0[X_{k+1}-X_{k}]
=\lim_{n\to\infty} \frac1n \sum_{k=0}^{n-1} E_0[D(T_{X_k}\w)].
\]
We have proved  $v=\E_\infty(D)$.

 The variables
$(X_{\tau_{j+1}}-X_{\tau_{j}}, \tau_{j+1}-\tau_{j})_{j\ge 1}$ are
i.i.d.\  with sufficient moments by Hypotheses (\HT) and \astep.
With $\alpha_n= \inf\{j\ge 1: \tau_{j}-\tau_1\ge n\}$ Wald's identity gives
\begin{align*}
E_0 [X_{\tau_{\alpha_n}}-X_{\tau_1}]&=E_0[\alpha_n]
E_0[X_{\tau_1}\vert\beta=\infty]\\
E_0 [ \tau_{\alpha_n}-\tau_1]&=E_0[\alpha_n] E_0[{\tau_1}\vert\beta=\infty].
\end{align*}
Consequently, by the definition \eqref{defv} of $v$,
\begin{align*}
E_0[X_n] - nv =
v E_0[\tau_{\alpha_n}-\tau_1-n]-
E_0 [X_{\tau_{\alpha_n}}-X_{\tau_1}-X_{n}].
\end{align*}
The right-hand-side is bounded by a constant
again  by Hypotheses (\HT) and \step and  by \eqref{renewal-tau-bound}.
\end{proof}

\begin{proposition}
\label{exchange}
Assume that there exists an $\bar\alpha<1/2$ such that
\begin{align}
\E\left[\abs{E_0^\w(X_n)-E_0(X_n)}^2\right]=\Ord(n^{2\bar\alpha}).
\label{cond1}
\end{align}
Let $\alpha\in(\bar\alpha,1)$ and assume that 
\beq
\momp>\frac{5d}{\alpha-\bar\alpha}.
\label{momp-cond5}
\eeq
Then condition \eqref{cond} is satisfied with  $\alpha$.
\end{proposition}

\begin{proof} Assumption \eqref{momp-cond5} permits us to 
choose $p$ such that 
\[
2d\frac{1-\bar\alpha}{\alpha-\bar\alpha}<p\le (\momp-1)/2.
\]
Due to the strict inequality above there is room to choose
$0<\e<d^{-1}(\alpha-\bar\alpha)$ such that 
$
p>2d+2\e^{-1}(1-\alpha).
$
Let $\ell=n^{\e}$ and $j=\ell^2$.  

By \eqref{velocity-bound} assumption \eqref{cond1}
turns into
\begin{align}
\label{cond1-modified}
\E\left[\abs{E_0^\w(X_n)-nv}^2\right]=\Ord(n^{2\bar\alpha}).
\end{align}
Define $A_\ell=\{\inf_{n\ge0}X_n\cdot\uhat\ge\ell\}$. The next
calculation starts with  $\Pi$-invariance of $\P_\infty$. 
\begin{align*}
&\E_\infty[|E_0^\w(X_n)-nv|^2]\\
&=E_0^\infty\Big[\big|E_0^{T_{X_j}\w}(X_n-nv)\big|^2\Big]\\
&\le E_0^\infty\Big[\big|E_0^{T_{X_j}\w}(X_n-nv)\big|^2,
X_j\cdot\uhat>\ell\Big]
+4\M^2n^2 P_0^\infty\{X_j\cdot\uhat\le\ell\}\\
&\le 2E_0^\infty\Big[\big|E_0^{T_{X_j}\w}(X_n-nv,A_{-\ell/2})\big|^2
,X_j\cdot\uhat>\ell\Big]\\
&\qquad\qquad
+ 8\M^2n^2 E_0^\infty\Big[P_0^{T_{X_j}\w}(A_{-\ell/2}^c),X_j\cdot\uhat>\ell\Big]
+ 4\M^2n^2 P_0^\infty\{X_j\cdot\uhat\le\ell\}\\
&\le 2\sum_{\substack{x:|x|\le \M j\\\text{and } x\cdot\uhat>\ell}} 
\E_\infty\big[\big|E_0^{T_x\w}(X_n-nv,A_{-\ell/2})\big|^2\big]\\
&\qquad\qquad
+ 8\M^2n^2 \sum_{\substack{x:\, |x|\le \M j\\\text{and }\, x\cdot\uhat>\ell}}
\E_\infty\big[P_0^{T_x\w}(A_{-\ell/2}^c)\big]
+ 4\M^2n^2 P_0^\infty\{X_j\cdot\uhat\le\ell\}\\
\intertext{[switch from $\E_\infty$ back to $\E$ by \eqref{vard-exp}]}
&\le 2\sum_{\substack{x: |x|\le \M j\\\text{and } x\cdot\uhat>\ell}} 
\E\big[\big|E_0^{T_x\w}(X_n-nv,A_{-\ell/2})\big|^2\big]
+ 8\M^2n^2 \sum_{\substack{x: |x|\le \M j\\\text{and } x\cdot\uhat>\ell}}
\E\big[P_0^{T_x\w}(A_{-\ell/2}^c)\big]\\
&\qquad\qquad+ C(\M j)^d\M^2n^2\ell^{-p}
+ 4\M^2n^2 P_0^\infty\{X_j\cdot\uhat\le\ell\}\\
&\le 2\sum_{\substack{x: |x|\le \M j\\\text{and } x\cdot\uhat>\ell}} 
\E\big[\,\lvert X_n-nv\rvert^2\;\big]
+ 16\M^2n^2 \sum_{\substack{x: |x|\le \M j\\\text{and } x\cdot\uhat>\ell}}
P_0(A_{-\ell/2}^c)\\
&\qquad\qquad+ C(\M j)^d\M^2n^2\ell^{-p}
+ 4\M^2n^2 P_0^\infty\{X_j\cdot\uhat\le\ell\}\\
\intertext{[use form 
\eqref{cond1-modified} of the assumption; 
apply \eqref{backtrack-bound} to $P_0(A_{-\ell/2}^c)$
and \eqref{ldp-invariant} to $P_0^\infty\{X_j\cdot\uhat\le\ell\}$;
recall that $j=\ell^2=n^{2\e}$]}
&\le C j^d n^{2\bar\alpha}+ C j^dn^2\ell^{-p}
+ Cn^2 j^{-p} \le C\bigl(n^{2\bar\alpha+2d\e} + n^{2d\e+2-p\e}
+ n^{2-2p\e}\bigr). 
\end{align*}
The first two exponents are $<2\alpha$ by the choice of $p$ and $\e$,
and the last one is less than the second one. 
\end{proof}

Once  we have verified the assumptions of Theorem \ref{RS} we have
the CLT under  $\P_\infty$-almost every $\w$.  But the goal is the CLT
under  $\P$-almost every $\w$.
As the final point of this section we prove
 the transfer of the central limit theorem from $\P_\infty$ to $\P$.
  This  is where we use the ergodic theorem, Theorem \ref{erg-thm}.
Let $\wiener$ be the probability distribution
 of the Brownian motion with diffusion
matrix $\mathfrak D$.

\begin{lemma}
 Suppose the weak convergence $Q^\w_n\Rightarrow W$
holds for $\P_\infty$-almost every $\w$.  Then the same is true
 for $\P$-almost every $\w$.
\label{clt-transfer-lm}
\end{lemma}

\begin{proof}
It suffices to show that for any $\delta>0$
and any bounded uniformly continuous
$F$ on $D_{\R^d}[0,\infty)$ 
\[
\varlimsup_{n\to\infty} E^{\w}_0[F(B_n)]\le \int F\,d\wiener+\delta \quad
\text{$\P$-a.s.}
\]
By considering also $-F$
this gives $E^{\w}_0[F(B_n)]\to \int F\,d\wiener$ {$\P$-a.s.}
for each such function.
A countable collection of them determines weak convergence.

Fix such an $F$ and  assume $|F|\le 1$.
 Let $c=\int F\,d\wiener$ and
\[
\overline h(\w)=\lsup_{n\to\infty} E_0^\w[F(B_n)].
\]
 For $\ell>0$ recall the events
\[
A_{-\ell}=\{\inf_{n\ge0} X_n\cdot\uhat\ge-\ell\}\]
and define
\[
\overline h_\ell(\w)=\lsup_{n\to\infty} E_0^\w[F(B_n), A_{-\ell}]\]
and
\[\Psi_\ell(\w)=\one\{\w\,:\, \bar{h}_\ell(\w)\le c+\tfrac12\delta,
P_0^\w(A_{-\ell}^c)\le\tfrac12\delta\}.
\]
The assumed quenched CLT under $\P_\infty$ gives $\P_\infty\{\bar h=c\}=1$.
Therefore, $\P_\infty$-a.s.
\[\Psi_\ell(\w)=\one\{\w:P_0^\w(A_{-\ell}^c)\le\tfrac12\delta\}.\]
From \eqref{backtrack-bound-invariant} we know that if
$\ell$ is fixed large enough, then $\E_\infty\Psi_\ell>0$. 
Since $\Psi_\ell$ is $\kS_{-\ell}$-measurable
Theorem \ref{erg-thm} implies that
\[n^{-1}\sum_{j=1}^n\Psi_\ell(T_{X_j}\w)\to\E_\infty\Psi_\ell>0\quad
P_0\text{-a.s.}\]
But $\{\bar h_\ell\le c+\tfrac12\delta,
P_0^\w(A_{-\ell}^c)\le\tfrac12\delta\}\subset\{\bar h\le c+\delta\}$.
We conclude that the stopping time
\[
\zeta=\inf\{n\ge 0: \bar{h}(T_{X_n}\w)\le c+\delta\}
\]
is $P_0$-a.s.\ finite. From the definitions we now have
\[
\varlimsup_{n\to\infty} E^{T_{X_\zeta}\w}_0[F(B_n)]\le
\int F\, dW+\delta \quad
\text{$P_0$-a.s.}
\]
Then by bounded convergence
\[\varlimsup_{n\to\infty}E_0^\w E_0^{T_{X_{\zeta}}\w}[F(B_n)]
\le\int F\, dW+\delta \quad \P\text{-a.s.}\]
Since $\zeta$ is a finite stopping time, the strong Markov property,
 the uniform continuity of $F$ and bounded step size
Hypothesis (\HM) imply
\[\varlimsup_{n\to\infty}E_0^\w [F(B_n)]
\le\int F\, dW+\delta\quad \P\text{-a.s.}\]
This concludes the proof.
\end{proof}

\section{Reduction to path intersections}
\label{pathintersections}
The preceding sections have reduced the proof of the
main result Theorem \ref{main}   to proving the estimate
\begin{align}
\E\bigl[\,\abs{E_0^\w(X_n)-E_0(X_n)}^2\,\bigr]=\Ord(n^{2\alpha})
\quad\text{for some $\alpha<1/2$.}
\label{cond1a}
\end{align}
The next reduction takes us to
 the expected number of intersections of the paths of two independent
walks $X$ and $\Xtil$
 in the same environment.
The argument uses a decomposition into martingale differences
through an  ordering of lattice sites.
This idea for bounding a variance is natural and has
been used  in RWRE earlier by Bolthausen and Sznitman \citep{bolt-szni-02}.

Let $P_{0,0}^\w$ be the quenched
law of the walks $(X,\Xtil)$ started at $(X_0,\Xtil_0)=(0,0)$
 and $P_{0,0}=\int P_{0,0}^\w\,\P(d\w)$ the averaged
law with expectation operator $E_{0,0}$. The set of sites visited
by a walk is denoted by $X_{[0,n)}=\{X_k: 0\le k<n\}$ and
 $\abs{A}$ is the number of elements in a discrete set $A$.

\begin{proposition}
\label{intersections}
Let $\P$ be an i.i.d.\ product measure and satisfy Hypotheses {\rm(\HT)} and 
{\rm(\HM)}.
Assume that there exists an $\bar\alpha<1/2$ such that
\begin{align}
E_{0,0}[\,|X_{[0,n)}\cap\Xtil_{[0,n)}|\,]=\Ord(n^{2\bar\alpha}).
\label{cond2}
\end{align}
Let $\alpha\in(\bar\alpha, 1/2)$.  Assume 
\beq
\momp> \frac{22d}{\alpha-\bar\alpha}.
\label{momp-cond6}\eeq
Then condition \eqref{cond1a} is satisfied for $\alpha$.
\end{proposition}

\begin{proof}
For $L\geq0$, define
${\mathcal B}(L)=\{x\in\Z^d:|x|\leq L\}$.
Fix $n\geq1$ and let $(x_j)_{j\geq1}$ be some fixed ordering of ${\mathcal B}(\M n)$ 
satisfying
\[\forall i\geq j:x_i\cdot \uhat\geq x_j\cdot \uhat.\]
For $B\subset\Z^d$ let $\kS_B=\sigma\{\w_x: x\in B\}$.
Let $A_j=\{x_1,\dotsc,x_{j}\}$, $\zeta_0= E_0(X_n)$,  and for $j\geq 1$
\[
\zeta_j=\E[E^\w_0(X_n)|\kS_{A_j}].
\]
 $(\zeta_j-\zeta_{j-1})_{j\geq1}$ is a sequence of $L^2(\P)$-martingale
differences. By Hypothesis \step $X_n\in\mathcal B(\M n)$ and so 
\begin{align}
\E[\,|E^\w_0( X_n) -E_0(X_n)|\,^2]=
\sum_{j=1}^{|{\mathcal B}(\M n)|}\E[\,|\zeta_j-\zeta_{j-1}|^2\,].
\label{mgale-decomp}
\end{align}
For $z\in\Z^d$ define half-spaces
\[
{\mathcal H}(z)=\{x\in\Z^d: x\cdot\uhat > z\cdot\uhat\}.
\]
Since $A_{j-1}\subset A_j\subset {\mathcal H}({x_j})^c$,
\begin{align}
&\E[\,|\zeta_j-\zeta_{j-1}|^2\,]  \nn\\
&\quad =
\int \P(d\w_{A_j})
\Bigl\lvert \iint \P(d\w_{A_j^c})\P(d\tilde{\w}_{x_j})
 \bigl\{ E^\w_0(X_n)-E^{\pp{\w,\tilde{\w}_{x_j}}}_0(X_n)\bigr\}
\Bigr\rvert^2\nn\\
&\quad \leq
\iint \P(d\w_{{\mathcal H}(x_j)^c})\P(d\tilde{\w}_{x_j})
\Bigl\lvert \int \P(d\w_{{\mathcal H}(x_j)}) 
\bigl\{ E^\w_0(X_n)-E^{\pp{\w,\tilde{\w}_{x_j}}}_0(X_n)\bigr\}
\Bigr\rvert^2.
\label{line-2}
 \end{align}
Above $\pp{\w,\tilde{\w}_{x_j}}$ denotes an environment obtained from $\w$
by replacing  $\w_{x_j}$ with  $\tilde{\w}_{x_j}$.

We fix a point $z=x_j$ to develop a bound for the expression above,
and then  return to collect the estimates.
Abbreviate $\wwtil=\pp{\w,\tilde{\w}_{x_j}}$.
Consider two walks that both start at $0$, one obeys
environment $\w$ and the other obeys $\wwtil$.
Couple them  so that they stay together until
the first time they visit $z$. Until a visit to $z$ happens, the
walks are identical. Let
\[H_z=\min\{n\ge1:X_n=z\}\]  be the first hitting
time of site $z$ and write
\begin{align}
&\int \P(d\w_{{\mathcal H}(z)})\bigl( E^\w_0(X_n)-E^\wwtil_0(X_n)\bigr)
\label{line-7}\\
&=
\int \P(d\w_{{\mathcal H}(z)}) \sum_{m=0}^{n-1} P^\w_0\{ H_z=m\}
\bigl( E^\w_z[X_{n-m}-z]-E^\wwtil_z[X_{n-m}-z]  \bigr)\nn\\
\begin{split}
&=\int \P(d\w_{{\mathcal H}(z)}) \sum_{m=0}^{n-1} \sum_{\ell>0}
 P^\w_0\{ H_z=m, \ell-1\leq\max_{0\leq j\leq m}X_j\cdot\uhat-z\cdot\uhat<\ell  \}
\\
&\qquad\qquad \times
\bigl( E^\w_z[X_{n-m}-z]-E^\wwtil_z[X_{n-m}-z]  \bigr).
\end{split} \label{line-19}
\end{align}
 Decompose
${\mathcal H}(z)={\mathcal H}_\ell(z)\cup {\mathcal H}_\ell'(z)$ where
\[
{\mathcal H}_\ell(z)=\{x\in\Z^d:\! z\cdot\uhat < x\cdot\uhat< z\cdot\uhat+\ell\}
\text{ and }
{\mathcal H}_\ell'(z)=\{x\in\Z^d:\!  x\cdot\uhat \geq z\cdot\uhat+\ell\}.
\]
Take a single $(\ell,m)$ term from the sum in
\eqref{line-19} and only
 the expectation $E^\w_z[X_{n-m}-z]$, and split it further into
two terms:
\begin{align}
&\int \P(d\w_{{\mathcal H}(z)})
P^\w_0\{ H_z=m, \ell-1\leq\max_{0\leq j\leq m}X_j\cdot\uhat-z\cdot\uhat<\ell\}
E^\w_z[X_{n-m}-z]\nn\\
&=\int \P(d\w_{{\mathcal H}(z)})
P^\w_0\{ H_z=m, \ell-1\leq\max_{0\leq j\leq m}X_j\cdot\uhat-z\cdot\uhat<\ell\}
E^\w_z[X_{\tau_\ell+n-m}-X_{\tau_\ell}]
\label{line-22}  \\
&+\int \P(d\w_{{\mathcal H}(z)})
P^\w_0\{ H_z=m, \ell-1\leq\max_{0\leq j\leq m}X_j\cdot\uhat-z\cdot\uhat<\ell\}
E^\w_z[ X_{n-m}-X_{\tau_\ell+n-m}+X_{\tau_\ell}-z]
\label{line-23}
\end{align}
 Regeneration time $\tau_\ell$ with index
$\ell$ is used simply to guarantee
 that the post-regeneration walk $X_{\tau_\ell+\,\centerdot}$ stays in
${\mathcal H}_\ell'(z)$.  
  Below we make use of this to get independence
from the environments
in ${\mathcal H}_\ell'(z)^c$.

Integral \eqref{line-22} is developed further as follows.
\begin{align}
&\int \P(d\w_{{\mathcal H}(z)})
 P^\w_0\{ H_z=m, \ell-1\leq\max_{0\leq j\leq m}X_j\cdot\uhat-z\cdot\uhat<\ell\}
 E^\w_z[X_{\tau_\ell+n-m}-X_{\tau_\ell}]\nn\\
&=\int \P(d\w_{{\mathcal H}_\ell(z)})
 P^\w_0\{ H_z=m, \ell-1\leq\max_{0\leq j\leq m}X_j\cdot\uhat-z\cdot\uhat<\ell\}
\int \P(d\w_{{\mathcal H}_\ell'(z)})
 E^\w_z[ X_{\tau_\ell+n-m}-X_{\tau_\ell}]\nn\\
&=\int \P(d\w_{{\mathcal H}_\ell(z)})
 P^\w_0\{ H_z=m, \ell-1\leq\max_{0\leq j\leq m}X_j\cdot\uhat-z\cdot\uhat<\ell\}
 E_z[ X_{\tau_\ell+n-m}-X_{\tau_\ell}\vert \kS_{{\mathcal H}_\ell'(z)^c}]\nn\\
&=\int \P(d\w_{{\mathcal H}_\ell(z)})
 P^\w_0\{ H_z=m, \ell-1\leq\max_{0\leq j\leq m}X_j\cdot\uhat-z\cdot\uhat<\ell\}
 E_0[ X_{n-m}\vert \beta=\infty].
 \label{line-31}
\end{align}
The last equality above comes from the regeneration structure, see
Theorem 1.4 in Sznitman-Zerner \citep{szni-zern-99}.  The $\sigma$-algebra
$\kS_{{\mathcal H}_\ell'(z)^c}$ is contained in the  $\sigma$-algebra
$\cG_\ell$ defined by (1.29) of \citep{szni-zern-99} for the walk starting at $z$.

The last quantity
  \eqref{line-31}  above
reads the environment only until the first visit
to $z$, hence does not see the distinction between $\w$ and $\wwtil$.
Consequently when integral
 \eqref{line-19} is developed separately for
$\w$ and $\wwtil$ into the sum of integrals \eqref{line-22}
and \eqref{line-23},
integrals \eqref{line-22} first develop into 
\eqref{line-31} separately  for $\w$ and $\wwtil$
and then cancel each other.

We are left with two instances
of integral \eqref{line-23}, one for both  $\w$ and $\wwtil$.
Put these
back into the $(\ell, m)$ sum in \eqref{line-19}.  Include
also the square around this expression from line \eqref{line-2}. 
These expressions for  $\w$ and $\wwtil$ are 
bounded separately with identical steps
 and added together in the end.  Thus we first separate
the two by an application of  $(a+b)^2\le2(a^2+b^2)$.
We continue the argument for the expression for $\w$ with this
bound on the square of \eqref{line-19}:  
\begin{align*}
&2\Bigl\{\,\sum_{\ell>0}\sum_{m=0}^{n-1}\int\P(d\w_{\cH(z)})
P_0^\w\{H_z=m,\ell-1\le\max_{0\le k\le m} X_k\cdot\uhat-z\cdot\uhat<\ell\}\\
&\quad\qquad\times
\bigl|E_z^\w(X_{n-m}-X_{\tau_\ell+n-m}
+X_{\tau_\ell}-z)\bigr|\,\Bigr\}^2\\
\intertext{[apply the step bound \astep]}
&\le
8\M^2 \int\P(d\w_{\cH(z)})
\Bigl\{\,\sum_{\ell>0}
P_0^\w\{H_z<n,\ell-1\le\max_{0\le k\le H_z} X_k\cdot\uhat-z\cdot\uhat<\ell\}
E_z^\w(\tau_\ell)\Bigr\}^2\\
\intertext{[introduce $\e=(\alpha-\bar\alpha)/4>0$]}
&\le
16\M^2 n^\e\sum_{\ell\le n^\e} 
\int\P(d\w_{\cH(z)}) P_0^\w\{H_z<n\}^2 
E_z^\w(\tau_\ell^2)\\
&\quad+
16\M^2\sum_{\ell>n^\e}\int\P(d\w_{\cH(z)})
P_0^\w\{H_z<n,\ell-1\le\max_{0\le k\le H_z} X_k\cdot\uhat-z\cdot\uhat<\ell\}
E_z^\w(\tau_\ell^2)\\
\intertext{[pick conjugate exponents $p>1$ and $q>1$]}
&\le 16\M^2 n^{\e}\sum_{\ell\le n^\e}
\Big(\int\P(d\w_{\cH(z)}) P_0^\w\{H_z<n\}^{2q}\Big)^{1/q} 
\Big(\int\P(d\w_{\cH(z)}) E_z^\w[\tau_\ell^{2p}]\Big)^{1/p}\\
&+
16\M^2\sum_{\ell>n^\e}
\Big(\int\P(d\w_{\cH(z)}) E_z^\w[\tau_\ell^{2p}]\Big)^{1/p}\\
&\quad\qquad\times
\Big(\int\P(d\w_{\cH(z)})
P_0^\w\{H_z<n,
\ell-1\le\max_{0\le k\le H_z}\! X_k\cdot\uhat-z\cdot\uhat<\ell\}^q
\Big)^{1/q}.
\end{align*}
The step above requires $\momp\ge 2p$. This and what is needed below
can be achieved by choosing
\[
p=\frac{d}{\alpha-\bar\alpha} \quad\text{and}\quad 
q=\frac{d}{d-(\alpha-\bar\alpha)}.
\]
Now put the above bound and its counterpart for 
 $\tilde\w$  back into \eqref{line-2},  and  continue with another
application  of H\"older's inequality:
\begin{align*}
&\E[\,|\zeta_j-\zeta_{j-1}|^2\,]\\
&\le 32\M^2n^\e\sum_{\ell\le n^\e} 
\E[ P_0^\w\{H_{x_j}<n\}^{2q}]^{1/q}
E_0[\tau_\ell^{2p}]^{1/p}\\
&\quad\qquad+
32\M^2\sum_{\ell>n^\e}
E_0[\tau_\ell^{2p}]^{1/p}
\E\Big[
P_0^\w\{H_{x_j}<n,
\ell-1\le\max_{0\le k\le H_{x_j}} X_k\cdot\uhat-{x_j}\cdot\uhat<\ell\}^{q}
\Big]^{1/q}\\
\intertext{[apply \eqref{tau-bound}]} 
&\le Cn^{4\e}\,
\E[ P_0^\w\{H_{x_j}<n\}^{2q}]^{1/q}\\
&\quad\qquad+C \sum_{\ell>n^\e}\ell^2\,
\E\Big[P_0^\w\{H_{x_j}<n,
\ell-1\le\max_{0\le k\le H_{x_j}} X_k\cdot\uhat-{x_j}\cdot\uhat<\ell\}^{q}
\Big]^{1/q}\\
\intertext{[utilize $q>1$]}
&\le Cn^{4\e}\,
\E[ P_0^\w\{H_{x_j}<n\}^2]^{1/q}\\
&\quad\qquad+C \sum_{\ell>n^\e}\ell^2
\sum_{k=0}^{n-1} \sum_{|x|\le \M n}
E_0\Big[P_0^\w\{X_k=x\}
P_x^\w\{|\inf_{m\ge0}X_m\cdot\uhat-x\cdot\uhat|\ge\ell-1\}\Big]^{1/q}\\
&\le Cn^{4\e}\,
P_{0,0}\{x_j\in X_{[0,n)}\cap\Xtil_{[0,n)}\}^{1/q}
+Cn^{d+1}
\sum_{\ell>n^\e}\ell^2
P_0\{|\inf_{m\ge0}X_m\cdot\uhat|\ge\ell-1\}^{1/q}\\
&\le Cn^{4\e}\,
P_{0,0}\{x_j\in X_{[0,n)}\cap\Xtil_{[0,n)}\}^{1/q}
+Cn^{2\alpha-d}.
\end{align*}
In the last step we used  \eqref{backtrack-bound} with an exponent
 $\tilde p= 3q+ q\e^{-1}(2d+1-2\alpha)$.   This requires $\tilde p\le \momp-1$
which  follows from \eqref{momp-cond6}. 
Finally put these bounds in the sum in \eqref{mgale-decomp}
and develop the last bound:
\begin{align*}
&\E[\,|E_0^\w(X_n)-E_0(X_n)|^2\,]
=\sum_{j=1}^{|{\mathcal B}(\M n)|}\E[\,|\zeta_j-\zeta_{j-1}|^2\,]\\
&\qquad\le Cn^{4\e}
\sum_{j=1}^{|{\mathcal B}(\M n)|}
P_{0,0}\{x_j\in X_{[0,n)}\cap\Xtil_{[0,n)}\}^{1/q}+Cn^{2\alpha}\\
&\qquad\le Cn^{4\e}(n^d)^{1-1/q}
\Big(\,\sum_{j=1}^{|{\mathcal B}(\M n)|}
P_{0,0}\{x_j\in X_{[0,n)}\cap\Xtil_{[0,n)}\}\Big)^{1/q}+Cn^{2\alpha}\\
&\qquad\le Cn^{4\e+d-d/q+ 2\bar\alpha/q}+Cn^{2\alpha}.
\end{align*}
where we used the assumption \eqref{cond2} in the last inequality.
With $q=d(d-(\alpha-\bar\alpha))^{-1}$ and 
 $\e=(\alpha-\bar\alpha)/4$ as chosen above, the last line
is $\Ord(n^{2\alpha})$.  
\eqref{cond1a} has been verified.
\end{proof}

\section{Bound on intersections}
\label{intersectbound}
The remaining piece of the proof of Theorem \ref{main}
is this estimate:
\begin{align}
E_{0,0}[\,\lvert {X_{[0,n)}\cap\Xtil_{[0,n)}}\rvert\,]=\Ord(n^{2\alpha})
\quad\text{for some $\alpha<1/2$,}
\label{cond3}
\end{align}
where $X$ and $\Xtil$ are two independent walks driven by a common
environment with quenched distribution
$P^\w_{x,y}[X_{0,\infty}\in A, \Xtil_{0,\infty}\in B]
= P^\w_{x}(A)P^\w_{y}(B)$ and averaged distribution
$E_{x,y}(\cdot)=\E P^\w_{x,y}(\cdot)$. 

To deduce the  sublinear bound
 we introduce joint regeneration times at which both
 walks regenerate on the same level in space (but not
necessarily at the same time).
Intersections happen only within  the
joint regeneration slabs, and the
expected number of intersections
decays at a polynomial rate in the distance between the points of
entry into the slab.
From  joint regeneration to regeneration the difference of
the two walks is  a Markov chain. This Markov chain can be approximated
by a symmetric random walk.  Via this preliminary work the
required bound boils down to deriving a Green function
 estimate for a Markov chain that can be suitably approximated
by a symmetric random walk.  This part is relegated to 
Appendix \ref{greenapp}.  Except for the appendices, we complete the proof
of the functional central limit theorem
in this section.

To aid  our discussion of a pair of walks $(X,\Xtil)$ we introduce
some new notation.
We write  $\theta^{m,n}$ for the shift on pairs of paths:
$\theta^{m,n}(x_{0,\infty},y_{0,\infty})=(\theta^mx_{0,\infty},
\theta^ny_{0,\infty})$.
If we write separate expectations for $X$ and $\Xtil$
under $P^\w_{x,y}$, these are denoted by $E^\w_x$ and $\Etil^\w_y$.

By a {\sl joint stopping time} we mean a pair $(\alpha, \altil)$
 that  satisfies $\{\alpha=m,\altil=n\}\in\sigma\{X_{0,m},\Xtil_{0,n}\}$.
Under the distribution $P^\w_{x,y}$ the walks $X$ and $\Xtil$ are
independent. Consequently if $\alpha\vee\altil<\infty$
$P^\w_{x,y}$-almost surely then
for any events $A$ and $B$,
\begin{align*}
&P^\w_{x,y}\{ (X_{0,\alpha},\Xtil_{0,\altil})\in A,\,
 (X_{\alpha, \infty},\Xtil_{\altil,\infty})\in B\}\\
&\qquad = E^\w_{x,y} \bigl[ \one\{(X_{0,\alpha},\Xtil_{0,\altil})\in A\}
P^\w_{X_{\alpha},\Xtil_{\altil}}\{
 (X_{0,\infty},\Xtil_{0,\infty})\in B\} \bigr].
\end{align*}
This type of joint restarting will be used without comment in the sequel.

The backtracking time
$\beta$ is as before in \eqref{defbeta}  and 
for the $\Xtil$ walk it is  $\tilde\beta=\inf\{n\ge 1: \Xtil_n\cdot\uhat
<\Xtil_0\cdot\uhat\}$.
When the walks are on a common level their difference
lies in the hyperplane
\beq
\mathbb{V}_d=\{z\in\Z^d: z\cdot\uhat=0\}.
\label{defVd}\eeq 
From 
a common level there is a uniform positive chance for simultaneously 
never backtracking.

\begin{lemma}
\label{common-beta-lemma}
Assume $\uhat$-transience {\rm\eqref{dir-trans}}
and the bounded step hypothesis {\rm(\HM)}. 
Then 
\begin{align}
\label{common-beta}
\eta\equiv\inf_{x-y\in\V_d} P_{x,y}\{\beta\wedge\tilde\beta=\infty\}>0.
\end{align}
\end{lemma}

\begin{proof}
By shift-invariance it is enough
to consider the case $P_{0,x}$ for $x\in\V_d$.
By the independence of environments and the bound $\M$
on the step size, 
\begin{align*}
P_{0,x}\{\beta=\betil=\infty\}
\ge P_0\{\beta>|x|/4\M\}^2 - 2P_0\{|x|/4\M<\beta<\infty\}.
\end{align*}
As $|x|\to\infty$ the right-hand side
 above converges to $2\eta_1=P_0\{\beta=\infty\}^2>0$.
Then we can find  $L>0$ such that
\beq
|x|>L \;\Longrightarrow\; P_{0,x}\{\beta\wedge\tilde\beta=\infty\}>\eta_1>0.
\label{backtrackaux}\eeq

It remains to check that $P_{0,x}\{\beta\wedge\tilde\beta=\infty\}>0$
for any fixed $x\le|L|$. The case $x=0$ is immediate because 
$P_{0,0}\{\beta=\tilde\beta=\infty\}=0$ 
 implies $P_0^\w\{\beta=\infty\}^2=0$ $\P$-a.s.\ and therefore
 contradicts transience \eqref{transbeta}. 

Let us assume that $x\ne0$.

If ${\mathcal J}=\{z:\E\pi_{0,z}>0\}\subset\R u$, 
transience  implies $u\cdot\uhat>0$.  Then  
$x+\R u$ and $\R u$ do not intersect and independence gives
$P_{0,x}\{\beta=\tilde\beta=\infty\}=P_0\{\beta=\infty\}^2>0$. 
(We did not invoke Hypothesis \reg to rule out this case
to avoid appealing to \reg 
unnecessarily.) 

Let us now assume
that ${\mathcal J}\not\subset\R u$ for any $u$.

The proof is completed by constructing 
two finite walks that start at $0$ and $x$ with these properties:
the walks  do not backtrack below level $0$, 
they reach a common fresh level $\ell$ at entry points 
that are as far apart as desired, and this pair
of walks has positive probability. 
  Then if additionally
the walks regenerate at level $\ell$ (an event independent
of the one just described) the event 
$\beta\wedge\tilde\beta=\infty$  has been realized. 
We also make these walks  reach level $\ell$ 
 in such a manner that no lower level can 
serve as a level for joint regeneration.
 This construction will be helpful later on in
 the proof of \label{common-beta-proof}
Lemma \ref{Yapplm1}.

To construct the paths let $z$ and $w$ be two
nonzero  noncollinear vectors such that 
$z\cdot\uhat>0$, $\E\pi_{0z}>0$, and $\E\pi_{0w}>0$.  
Such exist: the assumption that $\mathcal J$ not be one-dimensional
implies the existence of some pair of  noncollinear vectors
$w, \tilde w\in\mathcal J$.  Then transience \eqref{dir-trans} 
implies the existence of $z\in\mathcal J$ with $z\cdot\uhat>0$. 
Either  $w$ or $\tilde w$ must be noncollinear with $z$.
  
The case $w\cdot\uhat>0$ is easy:  let one walk 
repeat $z$-steps and the other one 
repeat $w$-steps suitably many time.
We provide more detail for the case $w\cdot\uhat\le0$. 

Let $n>0$ and $m\ge0$ be the minimal integers such that
$-nw\cdot\uhat=mz\cdot\uhat$. 
Since $mz+nw\ne 0$ by noncollinearity but  $(mz+nw)\cdot\uhat= 0$
 there must exist a 
vector $\tilde u$ such that
$\tilde u\cdot\uhat=0$ and $mz\cdot\tilde u+nw\cdot\tilde u>0$. 
 Replacing $x$ by $-x$ if necessary we can then assume that
\begin{align}
nw\cdot\tilde u+mz\cdot\tilde u>0\ge x\cdot\tilde u.
\label{pointaway}
\end{align}
Interchangeability of $x$ and $-x$ comes  from symmetry and 
 shift-invariance: 
\[P_{0,x}\{\beta\wedge\tilde\beta=\infty\}=
P_{0,-x}\{\beta\wedge\tilde\beta=\infty\}.\]
The point of \eqref{pointaway} is that the path
$\{(iz)_{i=0}^m\,,\,(mz+jw)_{j=0}^n\}$ points away from $x$ 
in direction $\tilde u$.

Pick $k$ large enough to have $|x-kmz-knw|>L$. 
Let the $X$ walk
start at 0 and take $km$ $z$-steps followed by $kn$ $w$-steps (returning
back to level 0) and then
$km+1$ $z$-steps (ending at a fresh level). 
Let the $\Xtil$ walk start at $x$ and take
$km+1$ $z$-steps. These two paths do not self-intersect
or  intersect
each other, as can be checked routinely though somewhat 
tediously. 

The endpoints of the paths are $2kmz+z+knw$ and $x+kmz+z$
which are on a common level, but further than  $L$ apart.
After these paths let the two walks regenerate, with
probability controlled by \eqref{backtrackaux}. 
This joint evolution implies $\beta\wedge\tilde\beta=\infty$ 
so by independence of environments 
\[P_{0,x}\{\beta\wedge\tilde\beta=\infty\}\ge
(\E\pi_{0z})^{3km+2}(\E\pi_{0w})^{kn}\eta_1>0.\qedhere
\]
\end{proof}

We now begin the development towards joint regeneration times for the walks $X$ and
$\Xtil$.
Define the stopping time 
\[
\gamma_\ell=\inf\{n\ge 0: X_n\cdot\uhat\ge \ell\}
\]
and  the running maximum 
\[M_n=\sup\{X_i\cdot\uhat:i\le n\}.\]  
We write $\gamma(\ell)$ when subscripts or superscripts 
become complicated.  
 ${\Mtil}_n$ and   $\tilde\gamma_\ell$ 
are the corresponding
quantities for the $\Xtil$ walk.  

Let $\h$ be the greatest common divisor of 
\beq
\cL=\{\ell\ge0:P_0(\exists n:X_n\cdot\uhat=\ell)>0\}.\label{defcL}
\eeq 
First we observe that all high enough multiples of $\h$ are accessible 
levels from 0.

\begin{lemma}
\label{level-aux-lemma-1}
There exists a finite $\ell_0$ such that for all $\ell\ge\ell_0$
\[P_0\{\exists n:X_n\cdot\uhat=\h\ell\}>0.\]
\end{lemma}

\begin{proof}
The point is that $\cL$ is closed under
addition. Indeed, if $\ell_1$ and $\ell_2$ are in $\cL$, then 
let $x_{0,n_i}^{(i)}$, $i\in\{1,2\}$, 
be two paths
such that $x^{(i)}_0=0$, $x_{n_i}^{(i)}\cdot\uhat=\ell_i$, and 
$P_0\{X_{0,n_i}=x_{0,n_i}^{(i)}\}>0$. Let $k_1$ be the smallest index such that 
$x_{k_1}^{(1)}=x^{(1)}_{n_1}+x^{(2)}_{k_2}$ for some $k_2\in[0,n_2]$.
The set of such $k_1$ is not empty because 
$k_1=n_1$ and $k_2=0$ satisfy this equality. Now the path
$(x^{(1)}_{0,k_1},\,x^{(1)}_{n_1}+x^{(2)}_{k_2+1,n_2})$ starts at 0, ends on
level $\ell_1+\ell_2$ and has positive $P_0$-probability.

The familiar argument \citep[Lemma 5.4, Ch.\ 5]{durr-probability} shows that all large enough multiples of 
$\h$ lie in $\cL$.
\end{proof}

Next we show that all high enough multiples of $\h$ can be reached as fresh levels
without backtracking.

\begin{lemma}
\label{level-aux-lemma-2}
There exists a finite  $\ell_1$ such that  for all $\ell\ge\ell_1$
\begin{align}
P_0\{X_{\gamma_{\h\ell}}\cdot\uhat=\h\ell,\,\beta>\gamma_{\h\ell}\}>0.
\label{level-aux-1}
\end{align}
\end{lemma}

\begin{proof}
Pick and fix a step $x$ such that $\E\pi_{0,x}>0$ and $x\cdot\uhat>0$. Then 
$x\cdot\uhat=k\h$ for some $k>0$. For any $0\le j\le k-1$,
by appeal to Lemma \ref{level-aux-lemma-1},
we find a path $\sigma^{(j)}$, with positive $P_0$-probability, going from
0 to a level $\ell\h$ with $\ell=j\mod k$.
By deleting initial and final segments  
if necessary and by shifting the reduced path,
we can assume that $\sigma^{(j)}$
visits a level in $k\h\Z$ only at the beginning and a level in 
$j\h+k\h\Z$ only at 
the end. In particular,  $\sigma^{(0)}$ is the single point 0.

Let $y^{(j)}$ be the endpoint of $\sigma^{(j)}$. Pick $m=m^{(j)}$ 
large enough so that the path 
$\sitil^{(j)}=((ix)_{0\le i<m},\,mx+\sigma^{(j)},mx+\,y^{(j)}+(ix)_{1\le i\le m})$ 
stays at or above
level 0 and ends at a fresh level. It has positive $P_0$-probability
because its constituent pieces all do. Note that
 the only self-intersections are those that possibly
exist within the piece $mx+\sigma^{(j)}$, and even these can be
removed by erasing loops from $\sigma^{(j)}$ as part of its
construction if so desired.    
Let $\ell_1$ be the maximal level attained by $\sitil^{(0)},\dots,\sitil^{(k-1)}$.

Given $\ell\ge\ell_1$ let $j=\ell\mod k$. Path $\sitil^{(j)}$ followed by 
appropriately many $x$-steps realizes the event in \eqref{level-aux-1}
and has positive $P_0$-probability.
\end{proof}

Next we extend the estimation to joint fresh levels of two walks
reached without backtracking.

\begin{lemma}
\label{level-aux-lemma-3}
Let $\ell_2\h$ be the next multiple of $\h$ after $\M \lvert\uhat\rvert+\ell_1\h$ with 
$\ell_1$ as in Lemma \ref{level-aux-lemma-2}.
There exists $\eta>0$ with this property: uniformly over all $x$ and $y$ such that
$x\cdot\uhat,\,y\cdot\uhat\in[0,\M \lvert\uhat\rvert\,]\cap\h\Z$,
\begin{align}
\begin{split}
&P_{x,y}\{\exists i:\,i\h\in[0,\ell_2\h],\\ 
&\qquad\qquad X_{\gamma_{i\h}}\cdot\uhat=\Xtil_{\gatil_{i\h}}\cdot\uhat=i\h,\,
\beta>\gamma_{i\h},\,\betil>\gatil_{i\h}\}\ge\eta.
\end{split}
\label{common-level}
\end{align}
\end{lemma}

\begin{proof}
Let $x\cdot\uhat=\ell\h$ and $y\cdot\uhat=\elltil\h$.  
Lemma \ref{level-aux-lemma-2} gives a positive $P_0$-probability path $\sigma=z_{0,n}$ 
that connects 0 to level $\ell_2\h-\ell\h$ and stays above level 0.
Choose $\sitil=\ztil_{0,\ntil}$ 
similarly for $\elltil$. If the paths $x+\si$ and $y+\sitil$ intersect, 
redefine $x+\sigma$ to follow
$y+\sitil$ from the first time it intersects $y+\sitil$. The probability in \eqref{common-level}
is bounded below by
\[P_{x,y}\{X_{0,n}=x+\si,\,\Xtil_{0,\ntil}=y+\sitil\}>0.\]

Uniformity over $x,y$ comes from observing that there are 
finitely many possible 
such positive lower bounds because we have finitely many 
admissible initial levels $\ell$ and $\elltil$
and finitely many ways to to intersect the shifts of the corresponding paths.
\end{proof}

Define the first common fresh level to be 
\[L=\inf\{\ell:X_{\gamma_\ell}\cdot\uhat={\Xtil_{\gatil_\ell}}\cdot\uhat=\ell\}.\]
If the walks start on a common level then this initial level 
is $L$.   Iteration of Lemma 
\ref{level-aux-lemma-3} shows that 
 $L$ is always a.s.\ finite
provided the walks start on levels  in $\h\Z$.  (This and more
is proved in Lemma \ref{level-aux-lemma-4} below.)   

Next we define, in stages, the first joint regeneration level of two 
walks $(X,\Xtil)$ that start at initial points  $X_0,\Xtil_0$ on a common level
$\lambda_0\in\h\Z$. First define
\[
J=\begin{cases}
M_{\beta\wedge\tilde\beta}\vee\tilde M_{\beta\wedge\tilde\beta}+\h 
&\text{if }\ \beta\wedge\tilde\beta<\infty,\\
\infty &\text{if }\ \beta\wedge\tilde\beta=\infty\end{cases}
\] 
and then
\[
\lambda=\begin{cases}
L\circ\theta^{\gamma_J,\gatil_J}
=\inf\{\ell\ge J:X_{\gamma_\ell}\cdot\uhat={\Xtil_{\gatil_\ell}}\cdot\uhat=\ell\}
&\text{if }\ J<\infty,\\
\infty &\text{if }\ J=\infty.\end{cases}
\] 
If $\lambda<\infty$, then $\lambda$ is the first common fresh level after at least one
walk backtracked. Also, $\lambda=\infty$ iff neither walk backtracked.
Let
\[\lambda_1=L\circ\theta^{\gamma(\lambda_0+\h),\gatil(\lambda_0+\h)}\] 
which is
the first common fresh level strictly above the initial level $\lambda_0$. For $n\ge2$
as long as $\lambda_{n-1}<\infty$ define successive common fresh levels
\[\lambda_n=\lambda\circ\theta^{\gamma_{\lambda_{n-1}},
\tilde\gamma_{\lambda_{n-1}}}.\]
Joint regeneration at level $\lambda_n$ is signaled by $\lambda_{n+1}=\infty$.
Consequently the first joint regeneration level is
\[\Lambda=\sup\{\lambda_n:\lambda_n<\infty\}.\]
$\Lambda<\infty$ a.s.\ because by Lemma \ref{common-beta-lemma} at each 
common fresh level $\lambda_n$ the walks have at least chance
 $\eta>0$ to
simultaneously not backtrack.
The first joint regeneration times are
\begin{align}
\label{defmu1mutil1}
(\mu_1,\tilde\mu_1)=(\gamma_\Lambda,\tilde\gamma_\Lambda).
\end{align}

The present goal is  to get moment  
bounds on $\mu_1$ and $\tilde\mu_1$. 
To be able to shift levels back to level 0 we fix representatives from all 
non-empty levels. 
For all $j\in\cL_0=\{z\cdot\uhat:z\in\Z^d\}$ pick and fix $\vhat(j)\in\Z^d$
such that $\vhat(j)\cdot\uhat=j$. By the definition of $\h$ 
as the greatest common divisor of $\cL$ in \eqref{defcL}
and the group structure
of $\cL_0$, $\vhat(j)$ is defined for all $j\in\h\Z$.

\begin{lemma}
\label{level-aux-lemma-4}
For  $m\ge1$ and $p\le\momp$
\begin{align}
\label{lambda-moment}
\sup_{x,y\in\V_d} P_{x,y}\{\Lambda>m\}\le C_p m^{-p}.
\end{align}
\end{lemma}

\begin{proof}
Recall $\ell_2$ from Lemma \ref{level-aux-lemma-3}. 
Consider $m>2\ell_2\h$ and let
$n_0=[m/(2\ell_2\h)]$.
 
Iterations of \eqref{common-beta} utilized below proceed as
follows:  for $k\ge2$ and any event $B$ that depends on
the paths $(X_{0\,,\,\gamma(\lambda_{k-1})},
\Xtil_{0\,,\,\gatil(\lambda_{k-1})})$,
\begin{align*}
&P_{x,y}\{\lambda_k<\infty,\,\lambda_{k-1}<\infty,\, B\}\\
&=
P_{x,y} \{(\beta\wedge\betil)\circ
\theta^{\gamma_{\lambda_{k-1}},\gatil_{\lambda_{k-1}}}<\infty,\,
\lambda_{k-1}<\infty,\,B\}\\
&=\sum_{z,w} P_{x,y} \{ 
X_{\gamma(\lambda_{k-1})}=z,\,\Xtil_{\gatil(\lambda_{k-1})}=w,\,
\lambda_{k-1}<\infty,\,B\} P_{z,w}\{\beta\wedge\betil<\infty\}\\
&\le P_{x,y} \{ \lambda_{k-1}<\infty,\,B\} (1-\eta).
\end{align*}
The product comes from dependence on disjoint environments:
 the event 
$\{\beta\wedge\betil<\infty\}$ does not need environments below
the starting level $z\cdot\uhat=w\cdot\uhat$, while the event
$\{X_{\gamma(\lambda_{k-1})}=z,\,\Xtil_{\gatil(\lambda_{k-1})}=w, \,B\}$
only reads environments strictly below this level. 

After the sum decomposition below 
 iterate \eqref{common-beta}
 to bound  $P_{x,y}\{\lambda_{n_0}<\infty\}$ 
and to go from $\lambda_n<\infty$ down
to $\lambda_{k+1}<\infty$ inside the sum.  
Then weaken $\lambda_{k+1}<\infty$ to
$\lambda_{k}<\infty$. 
Note that  $\lambda_1<\infty$ a.s.\ so this event 
 does not contribute a $1-\eta$  factor and hence there is only
a power $(1-\eta)^{n_0-1}$ for the middle term. 
\begin{align}
&P_{x,y}\{\Lambda>2m\}\nn\\
&\le P_{x,y}\{\lambda_1>m\}+P_{x,y}\{\lambda_{n_0}<\infty\}
+ \sum_{n=2}^{n_0-1}\sum_{k=1}^{n-1} P_{x,y}\{\lambda_n<\infty,\,
\tfrac{m}{n}<\lambda\circ\theta^{\gamma_{\lambda_k},\tilde\gamma_{\lambda_k}}-\lambda_k<\infty\}\nn\\
&\le P_{x,y}\{\lambda_1>m\}+(1-\eta)^{n_0-1} 
\label{level-line-1}\\ 
&\qquad+\sum_{n=2}^{n_0-1}\sum_{k=1}^{n-1} (1-\eta)^{n-k-1}
P_{x,y}\{\lambda_k<\infty,\,
\tfrac{m}{n}<\lambda\circ\theta^{\gamma_{\lambda_k},\tilde\gamma_{\lambda_k}}
-\lambda_k<\infty\}.
\label{level-line0}
\end{align}
Separate probability \eqref{level-line0} into two parts: 
\begin{align}
&P_{x,y}\{\lambda_k<\infty,\,
\tfrac{m}{n}<\lambda\circ\theta^{\gamma_{\lambda_k},\tilde\gamma_{\lambda_k}}
-\lambda_k<\infty\}\nn\\
&\le2 P_{x,y}\{\lambda_k<\infty,\,
\tfrac{m}{2n}<M_{\beta\wedge\tilde\beta}\circ\theta^{\gamma_{\lambda_k},\tilde\gamma_{\lambda_k}}+\h-\lambda_k <\infty\}\label{level-line1}\\
&\qquad+
P_{x,y}\{\lambda_k<\infty,\,J\circ
\theta^{\gamma_{\lambda_k},\tilde\gamma_{\lambda_k}}<\infty,
\tfrac{m}{2n}<(L\circ\theta^{\gamma_J,\gatil_J}-J)\circ
\theta^{\gamma_{\lambda_k},\tilde\gamma_{\lambda_k}}\}.
\label{level-line2}
\end{align}
For probability \eqref{level-line1}
\begin{align}
&P_{x,y}(\lambda_k<\infty, \,\tfrac{m}{2n}< 
M_{\beta\wedge\tilde\beta}\circ\theta^{\gamma_{\lambda_k},\tilde\gamma_{\lambda_k}}
+\h-\lambda_k <\infty )\nn\\
&=\sum_{z\cdot\uhat=\ztil\cdot\uhat=0}
P_{x,y}\{\lambda_{k}<\infty,\,
X_{\gamma_{\lambda_{k}}}=z+\vhat(\lambda_k),
\Xtil_{\tilde\gamma_{\lambda_{k}}}=\ztil+\vhat(\lambda_k)\}
P_{z,\ztil}\{\tfrac{m}{2n}<  
M_{\beta\wedge\tilde\beta}+\h   <\infty\}\nn\\
&\le CP_{x,y}\{\lambda_{k}<\infty\} (n/m)^p\le\cdots\le 
C(1-\eta)^{k-1}(n/m)^p.
\label{level-line3}
\end{align}
The independence above came from the fact that
the variable  $M_{\beta\wedge\tilde\beta}$
needs environments only on levels at or above the initial level. 
Starting at level 0, on the event $\beta\wedge\tilde\beta<\infty$ we have 
\[M_{\beta\wedge\tilde\beta}+\h \le \M|\uhat|\,\beta\wedge\tilde\beta+\h \le 
C(\tau_1+\tilde\tau_1).\] 
Then  we invoked   Hypothesis (\HT) for the moments of $\tau_1$ and $\tautil_1$. 
Finally iterate \eqref{common-beta} again
as prior to \eqref{level-line0}.

Probability \eqref{level-line2}  does not develop as conveniently because $L$ needs
environments below the starting level. To remove this dependence we 
use
the event $\cE$ defined below. Start by rewriting  \eqref{level-line2}  as follows.
\begin{align}
&P_{x,y}\{\lambda_k<\infty,\,J\circ
\theta^{\gamma_{\lambda_k},\tilde\gamma_{\lambda_k}}<\infty,\,
\tfrac{m}{2n}<(L\circ\theta^{\gamma_J,\gatil_J}-J)\circ
\theta^{\gamma_{\lambda_k},\tilde\gamma_{\lambda_k}}\}\nn\\
&=\sum_{j\in\h\Z} \sum_{z,\ztil} E_{x,y}\bigl[\lambda_k<\infty,\,
J\circ\theta^{\gamma(\lambda_k),\gatil(\lambda_k)}=j,\,
X_{\gamma_j}=z,\,
\Xtil_{\gatil_j}=\ztil,
P^\w_{z,\ztil}\{\tfrac{m}{2n}<L-j\}\bigr].
\label{level-line4}
\end{align}
Fix $j$ for the moment. We bound the probability in \eqref{level-line4}.
Let 
$s_0$ and $s_1$ be the integers defined by
\[(s_0-1)\ell_2\h<j\le s_0\ell_2\h<\cdots<s_1\ell_2\h\le j+\tfrac{m}{2n}<(s_1+1)\ell_2\h.\]
In the beginning of the proof
 we assured that $\tfrac{m}{2n}>\ell_2\h$ so $s_0$ and $s_1$ are well defined.
Define
\[\cE=\{\exists i:\,i\h\in[0,\ell_2\h],\, 
X_{\gamma_{i\h}}\cdot\uhat=\Xtil_{\gatil_{i\h}}\cdot\uhat=i\h,\,
\beta>\gamma_{i\h},\,\betil>\gatil_{i\h}\},\]
an event that guarantees a common fresh level in a zone 
of height $\ell_2\h$ without backtracking.
We use $\cE$ in situations where
the levels of the initial points are in 
$[0,r_0\lvert\uhat\rvert]\cap\h\Z$ and
then $\cE$ only needs environments $\{\w_a:a\cdot\uhat\in[0,\ell_2\h)\}$.
For any integer $s\in[s_0,s_1-1]$ we do the following decomposition.
\begin{align*}
&P^\w_{z,\ztil}\{L>(s+1)\ell_2\h\}\\
&\le P^\w_{z,\ztil}\{L>s\ell_2\h,\, 
(X_{\gamma(s\ell_2\h)+\,\centerdot}-\vhat(s\ell_2\h),\Xtil_{\gatil(s\ell_2\h)+\,\centerdot}-
\vhat(s\ell_2\h))
\in\cE^c\}\\
&\le \sum_{w,\wtil} P_{z,\ztil}^\w\{L>s\ell_2\h,\,
X_{\gamma(s\ell_2\h)}=w,\,\Xtil_{\gatil(s\ell_2\h)}=\wtil\}
P^{T_{\vhat(s\ell_2\h)}\w}_{w-\vhat(s\ell_2\h),\wtil-\vhat(s\ell_2\h)}\{\cE^c\}.
\end{align*}
To begin the iterative factoring write 
$P^\w_{z,\ztil}\{\tfrac{m}{2n}<L-j\}\le P^\w_{z,\ztil}\{L>s_1\ell_2\h\}$ and 
substitute the above decomposition with $s=s_1-1$ into \eqref{level-line4}. 
Notice that for each $(w,\wtil)$, the quenched probability
\[P^{T_{\vhat((s_1-1)\ell_2\h)}\w}_{w-\vhat((s_1-1)\ell_2\h),
\wtil-\vhat((s_1-1)\ell_2\h)}\{\cE^c\}\] 
is a function of environments $\{\w_a:a\cdot\uhat\in[(s_1-1)\ell_2\h,\,s_1\ell_2\h)\}$ 
and thereby independent of everything else inside the 
expectation $E_{x,y}$ in \eqref{level-line4},
as long as $s_0\le s_1-1$.
By Lemma \ref{level-aux-lemma-3}
\[P_{w-\vhat((s_1-1)\ell_2\h),\wtil-\vhat((s_1-1)\ell_2\h)}\{\cE^c\}\le 1-\eta.\]
After this first round probability \eqref{level-line2} is bounded,
via \eqref{level-line4},  by
\begin{align*}
&\sum_{j\in\h\Z} \sum_{z,\ztil} E_{x,y}\bigl[\lambda_k<\infty,\,
J\circ\theta^{\gamma(\lambda_k),\gatil(\lambda_k)}=j,\,
X_{\gamma_j}=z,\,\Xtil_{\gatil_j}=\ztil,
P^\w_{z,\ztil}\{L>(s_1-1)\ell_2\h\}\bigr](1-\eta).
\end{align*}
This procedure is repeated $s_1-s_0-1$ times to arrive at the upper bound
\begin{align*}
&P_{x,y}\{\lambda_k<\infty,\,J\circ
\theta^{\gamma_{\lambda_k},\tilde\gamma_{\lambda_k}}<\infty,\,
\tfrac{m}{2n}<(L\circ\theta^{\gamma_J,\gatil_J}-J)\circ
\theta^{\gamma_{\lambda_k},\tilde\gamma_{\lambda_k}}\}\\
&\qquad\le P_{x,y}\{\lambda_k<\infty\}(1-\eta)^{s_1-s_0-1}\\
&\qquad\le C P_{x,y}\{\lambda_k<\infty\}(1-\eta)^{m/(2\ell_2\h n)}\\
&\qquad\le C P_{x,y}\{\lambda_k<\infty\} (n/m)^p\le C(1-\eta)^{k-1} (n/m)^p.
\end{align*}
In the last step we iterated \eqref{common-beta} as earlier.

Substitute this upper bound and \eqref{level-line3} back to lines
\eqref{level-line1}--\eqref{level-line2}. 
These in turn go back into the sum on
line \eqref{level-line0}. The remaining
 probability  $P_{x,y}\{\lambda_1>m\}$  
on line \eqref{level-line-1} is 
bounded by $Ce^{-cm}$,  by another iteration
of Lemma \ref{level-aux-lemma-3} with the help of event $\cE$. 

To summarize, we have shown
\[P_{x,y}\{\Lambda>2m\}\le  Ce^{-cm}+C\sum_{n\ge1} n (1-\eta)^{n-2}  (n/m)^p
\le C m^{-p}.
\qedhere
\]
\end{proof}

Next we extend the tail bound to the regeneration times. 

\begin{lemma}  Suppose $\momp>3$. 
Then
\begin{align}
\label{mu-bound}
\sup_{x,y\in\V_d} P_{x,y}[\,\mu_1\vee\tilde\mu_1\ge m\,]\le Cm^{-\momp/3}.
\end{align}
In particular,  for any $p< \momp/3$, 
\begin{align}
\label{mu-bound7}
\sup_{x,y\in\V_d} E_{x,y}[\,|\mu_1\vee\tilde\mu_1|^p\,]\le C.
\end{align}
\label{mu-lemma1}\end{lemma}

\begin{proof}
By \eqref{tau-bound}, since $x\cdot\uhat=0$ for $x\in\V_d$,
 we can bound 
\[P_{x,y}\{\gamma_\ell\ge m\}=P_0\{\gamma_\ell\ge m\}\le P_0\{\tau_\ell\ge m\}
\le C(\ell/m)^\momp.\]

Pick conjugate exponents  $s=3$ and $t=3/2$. 
\begin{align*}
P_{x,y}\{\mu_1\ge m\}
&\le\sum_{\ell\ge1}P_{x,y}\{\gamma_\ell\ge m,\Lambda=\ell\}\\
&\le C\sum_{\ell\ge1}P_0\{\gamma_\ell\ge m\}^{1/s} P_{x,y}\{\Lambda=\ell\}^{1/t}\\
&\le C\sum_{\ell\ge1} \frac{\ell^{\momp/3}}{m^{\momp/3}}
\frac1{\ell^{2\momp/3}}\le Cm^{-\momp/3}.
\end{align*}
The same holds for $\tilde\mu_1$. 
\end{proof}

After these preliminaries define the sequence of
joint regeneration times by $\mu_0=\mutil_0=0$ and 
\beq
(\mu_{i+1},\tilde\mu_{i+1})=(\mu_i,\tilde\mu_i)
+(\mu_1,\tilde\mu_1)\circ\theta^{\mu_i, \tilde\mu_i}.
\label{defmukmutilk}
\eeq
The previous estimates, Lemmas \ref{level-aux-lemma-4}
and \ref{mu-lemma1}, show that common regeneration levels
come fast enough. 
The next tasks are to identify suitable Markovian 
structures and to develop a coupling.  Recall again the 
definition \eqref{defVd} of $\V_d$. 

\begin{proposition}  Under the averaged measure $P_{x,y}$
with $x,y\in\V_d$,  
the process  $(\Xtil_{\mutil_i}-X_{\mu_i})_{i\ge 1}$ is a Markov chain
on $\mathbb{V}_d$ with transition probability
\beq
q(x,y)=P_{0,x}\{  \Xtil_{\mutil_1}-X_{\mu_1}= y\,\vert\, 
\beta=\tilde\beta=\infty\}.
\label{defqxy}
\eeq
\label{Ymcprop}\end{proposition}

Note that the time-homogeneous 
Markov chain does not start from $\Xtil_{0}-X_{0}$ because the transition
  to $\Xtil_{\mutil_1}-X_{\mu_1}$ does not include the condition
$\beta=\tilde\beta=\infty$.

\begin{proof}
Let $n\ge2$ and $z_1,\dotsc,z_n\in\V_d$. The proof comes from iterating the following
steps.
\begin{align*}
&P_{0,z}\{\Xtil_{\mutil_i}-X_{\mu_i}=z_i\text{ for }1\le i\le n\}\\
&=\sum_{\wtil-w=z_{n-1}}\!\!\!\!
P_{0,z}\{\Xtil_{\mutil_i}-X_{\mu_i}=z_i\text{ for }1\le i\le n-2,\,
X_{\mu_{n-1}}=w,\,\Xtil_{\mutil_{n-1}}=\wtil\}\\
&\qquad\qquad\times
P_{w,\wtil}\{\Xtil_{\mutil_1}-X_{\mu_1}=z_n\,|\,\beta=\betil=\infty\}\\
&=\sum_{\wtil-w=z_{n-1}}\!\!\!\! 
P_{0,z}\{\Xtil_{\mutil_i}-X_{\mu_i}=z_i\text{ for }1\le i\le n-2,\,
X_{\mu_{n-1}}=w,\,\Xtil_{\mutil_{n-1}}=\wtil\}\\
&\qquad\qquad\times
P_{0,z_{n-1}}\{\Xtil_{\mutil_1}-X_{\mu_1}=z_n\,|\,\beta=\betil=\infty\}\\
&=P_{0,z}\{\Xtil_{\mutil_i}-X_{\mu_i}=z_i\text{ for }1\le i\le n-1\}q(z_{n-1},z_n).
\end{align*}

The factoring in the first equality above  
 is justified by the fact that 
\begin{align*}
&P_{0,z}^\w\{\Xtil_{\mutil_i}-X_{\mu_i}=z_i\text{ for }1\le i\le n-2,\,
X_{\mu_{n-1}}=w,\,\Xtil_{\mutil_{n-1}}=\wtil,\,
\Xtil_{\mutil_n}-X_{\mu_n}=z_n\}\\
&\qquad\qquad
=P_{0,z}^\w(A)P_{w,\wtil}^\w(B),
\end{align*}
where $A$ is a collection of paths staying below level $w\cdot\uhat=\wtil\cdot\uhat$,
while \[B=\{\Xtil_{\mutil_1}-X_{\mu_1}=z_n,\,\beta=\betil=\infty\}\]  
is a collection of paths that stay at or above their initial level.
\end{proof}

The Markov chain $Y_k=\Xtil_{\mutil_k}-X_{\mu_k}$ will be
compared to a random walk obtained by performing the same
construction of joint regeneration times to two
independent walks in independent environments.
To indicate the difference in construction we change
notation. Let the
pair of walks $(X,\Xbar)$ obey $P_0\otimes P_z$ with $z\in\mathbb{V}_d$,
and denote the first backtracking time of the $\Xbar$ walk
by $\betabar=\inf\{n\ge 1: \Xbar_n\cdot\uhat<\Xbar_0\cdot\uhat\}$.
Construct the
joint regeneration times $(\rho_k,\rhobar_k)_{k\ge 1}$
for $(X,\Xbar)$  by the same recipe
[\eqref{defmu1mutil1}, \eqref{defmukmutilk}, and the equations leading to them]
 as was used to construct
$(\mu_k,\mutil_k)_{k\ge 1}$ for $(X,\Xtil)$.
Define $\Ybar_k=\Xbar_{\rhobar_k}-X_{\rho_k}$.  An analog
of the previous
proposition, which we will not spell out,
shows that  $(\Ybar_k)_{k\ge 1}$ is a Markov chain with transition
\beq
\qbar(x,y)=
P_0\otimes P_x[  \Xbar_{\rhobar_1}-X_{\rho_1}= y\,\vert\,
\beta=\betabar=\infty].
\label{defqbar}\eeq

 In the next two proofs we make use of the
following decomposition.
Suppose  $x\cdot\uhat=y\cdot\uhat=0$, and let
 $(x_1,y_1)$ be another pair of points on a common, higher
 level: $x_1\cdot\uhat=y_1\cdot\uhat=\ell>0$.  Then we can write
\beq\begin{split}
&\{ (X_0,\Xtil_0)=(x,y),\,\beta=
\tilde{\beta}=\infty, \,
(X_{\mu_1},\Xtil_{\mutil_1})=(x_1,y_1)\}\\
&\quad =
\bigcup_{(\gamma,\gatil)}
\{ X_{0,n(\gamma)}=\gamma,\, \Xtil_{0,n(\gatil)}=\gatil, \,
\beta\circ\theta^{n(\gamma)}=
\tilde\beta\circ\theta^{n(\gatil)}=\infty\}.
\end{split}\label{pathdecomp1}\eeq
Here  $(\gamma,\gatil)$ range over all pairs of paths that connect
$(x,y)$ to $(x_1,y_1)$, that stay between levels $0$ and $\ell-1$
before the final points, and
for which a joint regeneration fails at all levels before $\ell$.
 $n(\gamma)$ is the index of the final point
 along the path, so for example
$\gamma=(x=z_0,z_1,\dotsc, z_{n(\gamma)-1},  z_{n(\gamma)}=x_1)$.

\begin{proposition} The process
$(\Ybar_k)_{k\ge 1}$ is a symmetric random walk on $\mathbb{V}_d$
and its   transition probability satisfies
\begin{align*}
\qbar(x,y)&=\qbar(0,y-x)=\qbar(0,x-y)\\
&=
P_0\otimes P_0\{  \Xbar_{\rhobar_1}-X_{\rho_1}= y-x\,\vert\,
\beta=\betabar=\infty\}.
\end{align*}
\label{Ybarprop2}
\end{proposition}

\begin{proof}   It remains  to show that for independent $(X,\Xbar)$
 the transition \eqref{defqbar}
 reduces to  a symmetric random walk. This becomes
obvious once probabilities are decomposed into sums over paths
because the events of interest are insensitive to shifts by
$z\in\mathbb{V}_d$.
\beq\begin{split}
&P_0\otimes P_x\{\beta=\betabar=\infty\,,\,
  \Xbar_{\rhobar_1}-X_{\rho_1}= y\}\\
&=\sum_w P_0\otimes P_x\{\beta=\betabar=\infty\,,\,
  X_{\rho_1}= w\,,\,\Xbar_{\rhobar_1}=y+w \}\\
&=\sum_w  \sum_{(\gamma, \gabar)}
P_0\{ X_{0,n(\gamma)}=\gamma,\, \beta\circ\theta^{n(\gamma)}=\infty\}
P_x\{ X_{0,n(\gabar)}=\gabar, \,
\beta\circ\theta^{n(\gabar)}=\infty\}\\
&=\sum_w  \sum_{(\gamma, \gabar)}
P_0\{ X_{0,n(\gamma)}=\gamma\} P_x\{ X_{0,n(\gabar)}=\gabar\}
\bigl(P_0\{\beta=\infty\}\bigr)^2.
\end{split}
\label{temp-gam-7}
\eeq

Above we used the decomposition idea from \eqref{pathdecomp1}.
Here  $(\gamma, \gabar)$ range  over the appropriate
class of pairs of
paths in $\Z^d$  such that $\gamma$ goes from $0$ to $w$ and
$\gabar$ goes from $x$ to $y+w$.
The independence for the last equality above comes from
noticing that the quenched probabilities
$P^\w_0\{X_{0,n(\gamma)}=\gamma\}$ and $P^\w_w\{\beta=\infty\}$
depend on independent collections of environments.

The probabilities on the last line of
\eqref{temp-gam-7}   are not changed if each pair $(\gamma,\gabar)$
is replaced
by $(\gamma,\gamma')=(\gamma,\gabar-x)$.  These pairs connect
 $(0,0)$ to $(w,y-x+w)$.
  Because $x\in\mathbb{V}_d$ satisfies
$x\cdot\uhat=0$,
 the shift has not changed regeneration levels.
This shift turns $P_x\{ X_{0,n(\gabar)}=\gabar\}$ on   the last line
of \eqref{temp-gam-7} into $P_0\{ X_{0,n(\gamma')}=\gamma'\}$.
We can reverse  the steps in
\eqref{temp-gam-7}   to arrive at the probability
\[
P_0\otimes P_0\{\beta=\betabar=\infty\,,\,
  \Xbar_{\rhobar_1}-X_{\rho_1}= y-x\}.
 \]
This proves $\qbar(x,y)=\qbar(0,y-x)$.

Once both walks start at $0$
 it is immaterial which is labeled
$X$ and which $\Xbar$, hence symmetry holds.
\end{proof}

It will be useful to know that $\qbar$ inherits all possible
transitions from $q$.

\begin{lemma} If $q(z,w)>0$ then also $\qbar(z,w)>0$.
\label{qqbarlm3}
\end{lemma}

\begin{proof}
By the decomposition from \eqref{pathdecomp1}
 we can express
\begin{align*}
&P_{x,y}\{  (X_{\mu_1},\Xtil_{\mutil_1})=(x_1,y_1) \vert \beta=
\tilde{\beta}=\infty \}
=
\sum_{(\gamma,\gatil)}
\frac{\E P^\w(\gamma)P^\w(\gatil)P^\w_{x_1}\{\beta=\infty\}
P^\w_{y_1}\{\beta=\infty \}}{P_{x,y}\{\beta=
\tilde{\beta}=\infty \}}.
\end{align*}
If this probability is positive, then at least one pair
$(\gamma,\gatil)$ must satisfy $\E P^\w(\gamma)P^\w(\gatil)>0$.
This implies that $ P(\gamma)P(\gatil)>0$ so that also
\[
P_{x}\otimes P_{y}\{  (X_{\mu_1},\Xtil_{\mutil_1})=(x_1,y_1) \vert \beta=
\tilde{\beta}=\infty \} >0.
\qedhere
\]
\end{proof}

In the sequel we detach  the notations $Y=(Y_k)$ and $\Ybar=(\Ybar_k)$
from their original definitions  in terms of the walks
$X$, $\Xtil$ and $\Xbar$,
 and use  $(Y_k)$ and $(\Ybar_k)$  to denote canonical Markov chains with
transitions $q$ and $\qbar$.
Now we construct  a coupling.

\begin{proposition} The single-step transitions $q(x,y)$ for $Y$ and
$\qbar(x,y)$ for $\Ybar$
can be coupled in such a way that, when the processes start
from a common state $x\ne0$,
\[
P_{x,x}\{Y_1\ne\Ybar_1\} \le C|x|^{-\momp/6} 
\]
for all $x\in\mathbb{V}_d$.  Here $C$ is a finite
positive constant independent of $x$.
\label{qqbarprop4}\end{proposition}

\begin{proof} We start by constructing a coupling of three walks $(X,\Xtil,\Xbar)$
such that the pair $(X,\Xtil)$ has distribution $P_{x,y}$ and
 the pair $(X,\Xbar)$ has distribution $P_{x}\otimes P_y$.

First let $(X,\Xtil)$ be two independent walks in a common environment
$\w$ as before. Let $\ombar$ be an environment independent
of $\w$.   Define the walk $\Xbar$ as follows.
Initially $\Xbar_0=\Xtil_0$.
On the sites $\{X_k:0\le k<\infty\}$ $\Xbar$ obeys environment
$\ombar$, and on all other sites $\Xbar$ obeys $\omega$.
$\Xbar$ is coupled to agree with $\Xtil$ until the time
\[
T=\inf\{ n\ge 0: \Xbar_n\in \{X_k:0\le k<\infty\}\,\}
\]
it hits the path of $X$.

The coupling between $\Xbar$ and $\Xtil$  can be achieved simply as
follows. Given $\w$ and $\ombar$,
  for each $x$ create two independent  i.i.d.~sequences
$(z^x_k)_{k\ge 1}$  and $(\zbar^x_k)_{k\ge 1}$ with distributions
\[
Q^{\w,\ombar}\{z^x_k=y\}=\pi_{x,x+y}(\w)
\quad\text{and}\quad
Q^{\w,\ombar}\{\zbar^x_k=y\}=\pi_{x,x+y}(\ombar).
\]
Do this independently at each $x$.
Each time the $\Xtil$-walk visits state $x$,
it uses a new $z^x_k$ variable as its next step, and never reuses the same
$z^x_k$ again.  The $\Xbar$ walk operates the same way except that
it uses the variables $\zbar^x_k$ when $x\in\{X_k\}$ and
the $z^x_k$ variables when  $x\notin\{X_k\}$.  Now $\Xbar$ and $\Xtil$
follow the same steps $z^x_k$ until $\Xbar$ hits the set $\{X_k\}$.

It is intuitively obvious that the walks $X$ and $\Xbar$ are
independent because they never use the same environment.
The following calculation verifies this.
Let $X_0=x_0=x$ and $\Xtil=\Xbar=y_0=y$ be the initial states, and
$\bfP_{x,y}$  the joint measure created by the coupling.
Fix finite vectors $x_{0,n}=(x_0,\dotsc,x_n)$ and
$y_{0,n}=(y_0,\dotsc,y_n)$ and recall also the notation
 $X_{0,n}=(X_0,\dotsc,X_n)$.
The description of the coupling tells us to start as follows.
\begin{align*}
&\bfP_{x,y}\{X_{0,n}=x_{0,n}, \Xbar_{0,n}=y_{0,n}\}\\
&=\int\P(d\w)\int\P(d\ombar) \int P_x^\w(d{z_{0,\infty}})
\one\{z_{0,n}=x_{0,n}\} \\
&\qquad \times
\!\!\prod_{i: y_i\notin\{z_k:\,0\le k<\infty\}}\!\!\!\pi_{y_i,y_{i+1}}(\w)
\prod_{i: y_i\in\{z_k:\,0\le k<\infty\}}\!\!\!\pi_{y_i,y_{i+1}}(\ombar)\\
\intertext{[by dominated convergence]}
&=\lim_{N\to\infty}\int\P(d\w)\int\P(d\ombar) \int P_x^\w(dz_{0,N})
\,\one\{z_{0,n}=x_{0,n}\} \\
&\qquad \times
\!\prod_{i: y_i\notin\{z_k:\,0\le k\le N\}}\!\!\!\pi_{y_i,y_{i+1}}(\w)
\prod_{i: y_i\in\{z_k:\,0\le k\le N\}}\!\!\!\pi_{y_i,y_{i+1}}(\ombar)\\
&=\lim_{N\to\infty} \sum_{z_{0,N}: z_{0,n}=x_{0,n}}
\int\P(d\w) \, P_x^\w[X_{0,N}=z_{0,N}]
\!\!\prod_{i: y_i\notin\{z_k:\,0\le k\le N\}}\!\!\!\pi_{y_i,y_{i+1}}(\w)\\
&\qquad \times
\int\P(d\ombar)\!\!
\prod_{i: y_i\in\{z_k:\,0\le k\le N\}}\!\!\!\pi_{y_i,y_{i+1}}(\ombar)\\
\intertext{[by independence of the two functions of $\w$]}
&=\lim_{N\to\infty} \sum_{z_{0,N}: z_{0,n}=x_{0,n}}
\int\P(d\w) \, P_x^\w\{X_{0,N}=z_{0,N}\}\\
&\qquad\times
\int \P(d\w)
\prod_{i: y_i\notin\{z_k:\,0\le k\le N\}}\pi_{y_i,y_{i+1}}(\w)
\int\P(d\ombar)
\!\!\prod_{i: y_i\in\{z_k:\,0\le k\le N\}}\!\!\!\pi_{y_i,y_{i+1}}(\ombar)\\
&=P_x\{X_{0,n}=x_{0,n}\}\, P_y\{X_{0,n}=y_{0,n}\}.
\end{align*}

Thus at this point the coupled pairs $(X,\Xtil)$ and $(X,\Xbar)$
have the desired marginals $P_{x,y}$ and $P_x\otimes P_y$.

Construct the joint regeneration
times $(\mu_1,\mutil_1)$ for  $(X,\Xtil)$ and
$(\rho_1,\rhobar_1)$ for  $(X,\Xbar)$ by the earlier
recipes.
Define  two pairs of walks stopped at their
joint regeneration times:
\beq
(\Gamma,\Gammabar)\equiv\bigl( (X_{0,\,\mu_1},\Xtil_{0,\,\mutil_1}),
(X_{0,\,\rho_1},\Xbar_{0,\,\rhobar_1})\bigr).
\label{defGaGa}
\eeq

Suppose the sets $X_{[0,\,\mu_1\vee\rho_1)}$ and
$\Xtil_{[0,\,\mutil_1\vee\rhobar_1)}$ do not intersect. Then the
construction implies that the path $\Xbar_{0,\,\mutil_1\vee\rhobar_1}$
 agrees with
$\Xtil_{0,\,\mutil_1\vee\rhobar_1}$, and this forces the equalities
$(\mu_1,\mutil_1)=(\rho_1,\rhobar_1)$
 and $(X_{\mu_1},\Xtil_{\mutil_1})=(X_{\rho_1},\Xbar_{\rhobar_1})$.
We insert an estimate on this event.

\begin{lemma}
For $x\ne y$ in $\V_d$,
\beq
P_{x,y}\{X_{[0,\,\mu_1\vee\rho_1)}\cap
\Xtil_{[0,\,\mutil_1\vee\rhobar_1)}\ne\emptyset\}\le
C|x-y|^{-\momp/3}.
\label{capbound}
\eeq
\end{lemma}

\begin{proof}
Write
\begin{align*}
P_{x,y}\{X_{[0,\,\mu_1\vee\rho_1)}\cap\Xtil_{[0,\,\mutil_1\vee\rhobar_1)}\ne\emptyset\}
&\le P_{x,y}\{\mu_1\vee\mutil_1\vee\rho_1\vee\rhobar_1>|x-y|/2\M\}.
\end{align*}
The conclusion follows from \eqref{mu-bound}, extended to cover also 
$(\rho_1,\rhobar_1)$.
\end{proof}

From \eqref{capbound}
 we obtain
\beq
\bfP_{x,y}\bigl\{\,(X_{\mu_1},\Xtil_{\mutil_1})\ne
(X_{\rho_1},\Xbar_{\rhobar_1})\,\bigr\}
\le  \bfP_{x,y}\bigl\{\,\Gamma\ne\Gammabar\bigr\}
\le C\abs{x-y}^{-p_0/3}.
\label{XXtil-goal5}
\eeq

  But we are not finished yet. To represent 
the transitions $q$ and $\qbar$ we must also 
 include the conditioning on no backtracking.
For this generate an i.i.d.~sequence
$(X^{(m)},\Xtil^{(m)},\Xbar^{(m)})_{m\ge 1}$, each triple constructed
as $(X,\Xtil,\Xbar)$  above.  Continue to write $\bfP_{x,y}$ for  the
probability measure of the entire sequence.    Let also again
\[\Gamma^{(m)}=
(X^{(m)}_{0\,,\,\mu^{(m)}_1},\Xtil^{(m)}_{0\,,\,\mutil^{(m)}_1})
 \quad\text{and}\quad
\Gammabar^{(m)}=
(X^{(m)}_{0\,,\,\rho^{(m)}_1},\Xbar^{(m)}_{0\,,\,\rhobar^{(m)}_1})\]
  be the
pairs of paths
run up to  their joint regeneration times.

 Let $M$ be the first $m$ such that
the paths  $(X^{(m)},\Xtil^{(m)})$ do not backtrack,
which means that
\[
\text{$X^{(m)}_k\cdot\uhat\ge X^{(m)}_0\cdot\uhat$
and
$\Xtil^{(m)}_k\cdot\uhat\ge \Xtil^{(m)}_0\cdot\uhat$ for all $k\ge 1$.}
\]
Similarly define  $\Mbar$ for $(X^{(m)},\Xbar^{(m)})_{m\ge 1}$.  Both $M$ and
$\Mbar$ are stochastically bounded by geometric random variables
 by \eqref{common-beta}.

The pair of walks
$(X^{(M)},\Xtil^{(M)})$ is now
distributed as a pair of walks under the measure
$P_{x,y}\{\,\cdot\,\vert \beta=\tilde\beta=\infty\}$,
while
$(X^{(\Mbar)},\Xbar^{(\Mbar)})$
 is distributed as a pair of walks under
$P_{x}\otimes P_y\{\,\cdot\,\vert \beta=\betabar=\infty\}$.
Consider  the two pairs
of paths $(\Gamma^{(M)}, \Gammabar^{(\Mbar)})$
chosen by the random indices $(M,\Mbar)$.
We insert one more lemma.

\begin{lemma}
For $x\ne y$ in $\V_d$,
\beq
\bfP_{x,y}\bigl\{\,\Gamma^{(M)}\ne\Gammabar^{(\Mbar)}\bigr\}
\le C\abs{x-y}^{-\momp/6}.
\label{XXtil-goal7}
\eeq
\end{lemma}
\begin{proof}
Let $\cA_m$ be the event that
the walks $\Xtil^{(m)}$ and $\Xbar^{(m)}$ agree up to the maximum
$\mutil^{(m)}_1\vee\rhobar^{(m)}_1$ of their regeneration times.
The equalities $M=\Mbar$ and
$\Gamma^{(M)}=\Gammabar^{(\Mbar)}$ are a consequence
of the event 
\[\{\cA_1\cap\dotsm\cap \cA_M\}=\bigcup_{m\ge 1}\{M=m\}
\cap\cA_1\cap\dotsm\cap \cA_m ,\]
 for the following reason.
As pointed out earlier, on the event $\cA_m$  we have the equality
of the regeneration times  $\mutil^{(m)}_1=\rhobar^{(m)}_1$
and of the  stopped paths
$\Xtil^{(m)}_{0\,,\,\mutil^{(m)}_1}=
\Xbar^{(m)}_{0\,,\,\rhobar^{(m)}_1}$.  By definition, these walks
do not backtrack after the regeneration time.
  Since the walks  $\Xtil^{(m)}$ and $\Xbar^{(m)}$ agree
up to this time, they must backtrack or fail to backtrack
together.  If this is true for
each $m=1,\dotsc,M$, it forces  $\Mbar=M$, since the other factor
in deciding  $M$ and $\Mbar$ are the paths $X^{(m)}$ that are common
to both.   And since the paths agree up to the regeneration times,
we have  $\Gamma^{(M)}=\Gammabar^{(\Mbar)}$.

Estimate \eqref{XXtil-goal7}  follows:
\begin{align*}
&\bfP_{x,y}\bigl\{\,\Gamma^{(M)}\ne\Gammabar^{(\Mbar)}\,\bigr\}
\le \bfP_{x,y}\bigl\{\,\cA_1^c\cup\dotsm\cup \cA_M^c\,\bigr\}\\
&\le \sum_{m=1}^\infty \bfP_{x,y}\{M\ge m,\, \cA_m^c\,\}
\le \sum_{m=1}^\infty \bigl(\bfP_{x,y}\{M\ge m\}\bigr)^{1/2}
\bigl(\bfP_{x,y}( \cA_m^c)\bigr)^{1/2}\\
&\le C|x-y|^{-\momp/6}.
\end{align*}
The last step comes from the estimate in \eqref{capbound}
for each $\cA_m^c$ and the geometric bound on $M$.
\end{proof}

We are ready to finish the proof of Proposition
\ref{qqbarprop4}.
To create initial conditions
 $Y_0=\Ybar_0=x$ let the walks start at 
 $(X^{(m)}_0,\Xtil^{(m)}_0)=(X^{(m)}_0,\Xbar^{(m)}_0)=(0,x)$.
 Let the final outcome of
the coupling be the pair
\[
(Y_1,\Ybar_1)=\bigl( \Xtil^{(M)}_{\mutil^{(M)}_1} \;-\;
X^{(M)}_{\mu^{(M)}_1}\,,\,
\Xbar^{(\Mbar)}_{\rhobar^{(\Mbar)}_1} \;-\;
X^{(\Mbar)}_{\rho^{(\Mbar)}_1}\bigr)
\]
under the measure $\bfP_{0,x}$. The marginal distributions
of $Y_1$ and $\Ybar_1$ are correct
[namely, given by the transitions
\eqref{defqxy} and  \eqref{defqbar}]  because, as argued above,
the pairs of walks themselves have the right marginal distributions.
The event $\Gamma^{(M)}=\Gammabar^{(\Mbar)}$ implies
$Y_1=\Ybar_1$, so
estimate \eqref{XXtil-goal7} gives the bound claimed in
Proposition \ref{qqbarprop4}.
 \end{proof}

The construction of the Markov chain is complete, and we return to
the main development of the proof.  It remains to prove a sublinear
bound on the expected number
$E_{0,0}\lvert X_{[0,n)}\cap \Xtil_{[0,n)}\rvert $
of common points of two independent walks in a common environment.
Utilizing the joint regeneration times,  write
\beq
E_{0,0}\lvert X_{[0,n)}\cap \Xtil_{[0,n)}\rvert
\le\sum_{i=0}^{n-1} E_{0,0}\lvert X_{[\mu_i,\mu_{i+1})}
\cap \Xtil_{[\mutil_i,\mutil_{i+1})}\rvert.
\label{capbd7}\eeq

The term $i=0$ is a finite constant by bound \eqref{mu-bound}
 because the number of common points is
bounded by the number $\mu_1$ of steps.  
For each  $0<i<n$ 
apply a  decomposition
into pairs of paths from $(0,0)$ 
to given points $(x_1,y_1)$   in the style of \eqref{pathdecomp1}: 
$(\gamma,\gatil)$ are the pairs of paths with the property that 
\begin{align*}
&\bigcup_{(\gamma,\gatil)}
\{X_{0,n(\gamma)}=\gamma,\, \Xtil_{0,n(\gatil)}=\gatil,\,
\beta\circ\theta^{n(\gamma)}=
\tilde\beta\circ\theta^{n(\gatil)}=\infty\}
=\{ X_0=\Xtil_0=0,\,  X_{\mu_i}=x_1,\, \Xtil_{\mutil_i}=y_1\}.
\end{align*}
Each term $i>0$ in \eqref{capbd7} we rearrange as follows. 
\begin{align*}
&E_{0,0}\lvert X_{[\mu_i,\mu_{i+1})}
\cap \Xtil_{[\mutil_i,\mutil_{i+1})}\rvert\\
&=\sum_{x_1,y_1}\sum_{(\gamma,\gatil)} 
P_{0,0}\{ X_{0,n(\gamma)}=\gamma,\, \Xtil_{0,n(\gatil)}=\gatil\} 
E_{x_1,y_1}[\one\{\beta=\tilde{\beta}=\infty\}
 \lvert X_{[0\,,\, \mu_{1})}
\cap \Xtil_{[0\,,\,\mutil_{1})}\rvert\,]\\
&=\sum_{x_1,y_1}\sum_{(\gamma,\gatil)}
P_{0,0}\{ X_{0,n(\gamma)}=\gamma,\, \Xtil_{0,n(\gatil)}=\gatil\}
P_{x_1,y_1}\{\beta=\tilde{\beta}=\infty\} 
E_{x_1,y_1}[\, 
 \lvert X_{[0\,,\,\mu_1)}
\cap \Xtil_{[0\,,\,\mutil_1)}\rvert  \,\vert\,
\beta=\tilde{\beta}=\infty\,]\\
&=\sum_{x_1,y_1}
P_{0,0}\{ X_{\mu_i}=x_1,\, \Xtil_{\mutil_i}=y_1\}
 E_{x_1,y_1}[\, 
 \lvert X_{[0\,,\,\mu_1)}
\cap \Xtil_{[0\,,\,\mutil_1)}\rvert  \,\vert\,
\beta=\tilde{\beta}=\infty\,].
\end{align*}
We have used the product structure of $\P$ in the first and and last equalities.
The last conditional expectation above is handled by estimates
\eqref{common-beta}, \eqref{mu-bound}, \eqref{capbound}
  and Schwarz inequality:
\begin{align*}
&E_{x_1,y_1}[\, 
 \lvert X_{[0\,,\,\mu_1)}
\cap \Xtil_{[0\,,\,\mutil_1)}\rvert  \,\vert\,
\beta=\tilde{\beta}=\infty\,]\le \eta^{-1}  E_{x_1,y_1}[\, 
 \lvert X_{[0\,,\,\mu_1)}
\cap \Xtil_{[0\,,\,\mutil_1)}\rvert\,]\\
&\qquad \le \eta^{-1}  E_{x_1,y_1}[\mu_1\cdot \one\{
X_{[0\,,\,\mu_1)}\cap \Xtil_{[0\,,\,\mutil_1)}\ne\emptyset\}\,]\\
&\qquad\le \eta^{-1}  \bigl(E_{x_1,y_1}[\mu_1^2]\bigr)^{1/2} 
\bigl( P_{x_1,y_1}\{
X_{[0\,,\,\mu_1)}\cap \Xtil_{[0\,,\,\mutil_1)}\ne\emptyset\}\,\bigr)^{1/2}\\
&\qquad\le C\bigl(1\vee\abs{x_1-y_1}\bigr)^{-\momp/6} \le h(x_1-y_1). 
\end{align*}
On the last line we defined  
\beq h(x)=C(\abs{x}\vee1)^{-\momp/6}. \label{defhfunction}\eeq
Insert the last bound back up, and appeal to the Markov property
established in  Proposition \ref{Ymcprop}:
\begin{align*}
E_{0,0}\lvert X_{[\mu_i,\mu_{i+1})}
\cap \Xtil_{[\mutil_i,\mutil_{i+1})}\rvert
&\le  E_{0,0} \bigl[h( \Xtil_{\mutil_i}-X_{\mu_i})\bigr]\\
&=\sum_{x} P_{0,0}\{\Xtil_{\mutil_1}-X_{\mu_1}=x\}
\sum_yq^{i-1}(x,y)h(y).
\end{align*}

In order to apply Theorem \ref{greenthm1} from Appendix \ref{greenapp},
 we check
its hypotheses in the next lemma.  Part \eqref{Yell5} of Hypothesis (\HR) 
enters
here crucially to guarantee that the transition $q$ has enough
irreducibility. 

\begin{lemma} The Markov chain $(Y_k)_{k\ge 0}$ with transition $q(x,y)$ 
and the symmetric random walk
$(\Ybar_k)_{k\ge 0}$ with transition $\qbar(x,y)$ 
satisfy assumptions {\rm (A.i)}, {\rm (A.ii)},  {\rm (A.iii)}
 and  {\rm (A.iv)}
stated in the beginning of Appendix \ref{greenapp}. To
ensure that $p_1>15$ as required by  {\rm (A.iv)}, we assume
$\momp> 90$.
\label{Yapplm1}\end{lemma}

\begin{proof} 
From \eqref{mu-bound7} and Hypothesis (\HM)
 we get moment bounds
\[E_{0,x}\lvert \Xbar_{\rhobar_k}\rvert^p
\;+\; E_{0,x}\lvert X_{\rho_k}\rvert^p <\infty 
\]
for $p<\momp/3$. With $\momp>9$ this gives assumption (A.i), namely 
that $E_0\lvert \Ybar_1\rvert^3<\infty$. [Lemma
\ref{mu-lemma1} is applied here to $(X,\Xbar)$ even though we
wrote the proof only for $(X,\Xtil)$.]  
  Assumption (A.iii) 
comes from Lemma \ref{qqbarlm3}.  Assumption (A.iv) comes
from Proposition \ref{qqbarprop4}.   

The only part that needs work is assumption (A.ii).
The required exponential exit time bound is achieved through
a combination of the following three steps, for constants $\delta>0$,
 $L>0$
and a fixed vector $\bhat\ne 0$:
\begin{align}
&P_0[Y_1\ne 0]\ge\delta,\label{Yline-1}\\[5pt]
\inf_{0<\abs{x}\le L}
&P_x[\,\lvert Y_1\rvert > L\,]\ge\delta,\label{Yline-2}\\ 
\text{and }\ \inf_{\abs{x}> L}
&\Bigl\{ P_x[\, Y_1 =Y_0+\bhat\,]\,\wedge\,P_x[\, Y_1 =Y_0-\bhat\,]\Bigr\}
\ge\delta.\label{Yline-3}
\end{align}
Given any initial state $x$ contained in a cube $[-r,r]^d$, 
there is a sequence of at most $2r$ steps of the types 
covered by the above estimates  that takes the 
chain $Y$ outside the cube, and this sequence of steps is taken
with probability at least $\delta^{2r}$. Thus the exit time from
the cube is dominated by $2r$ times a geometric random variable with mean 
$\delta^{-2r}$.

To prove \eqref{Yline-1}--\eqref{Yline-3} we make use of 
\beq
P_x[Y_1=z]\ge P_{0,x}\{\beta=\tilde\beta=\infty,\,
 \Xtil_{\mutil_1}=y+z,\, X_{\mu_1}=y\}
\label{YXaux1}\eeq
which is a consequence of the definition of the transition 
\eqref{defqxy} and valid for all $x,y,z$.  To this end we 
construct suitable paths for the $X$ and $\Xtil$ walks with
positive probabilities. 
We carry out the rest of the  proof in 
Appendix \ref{Yapp} because this requires a fairly tedious 
cataloguing of cases. 
\end{proof}  

Appendix \ref{greenapp} also requires $0\le h(x)\le C(1\vee\abs{x})^{-p_2}$
for $p_2>0$.  This we have without further requirements on $\momp$.  
Now that  the assumptions have been checked, Theorem  \ref{greenthm1}
gives constants  $0<C<\infty$ and 
$0<\eta<1/2$ 
such that 
\[
\sum_{i=1}^{n-1} \sum_yq^{i-1}(x,y)h(y) \le Cn^{1-\eta}
\quad\text{ for all $x\in\mathbb{V}_d$ and $n\ge 1$.}
\]
Going back to \eqref{capbd7} and collecting the bounds along
the way gives the final estimate 
\[
E_{0,0}\lvert X_{[0,n)}\cap \Xtil_{[0,n)}\rvert \le C_p n^{1-\eta}
\]
for all $n\ge 1$. Taking $p$ large enough, $1-\eta$ can be made as close
as desired to $1/2$. 
This is \eqref{cond3} which was earlier shown to imply 
condition \eqref{cond} required by Theorem \ref{RS}. 
Previous work in Sections \ref{prelim} and \ref{substitution}
convert the CLT from Theorem \ref{RS} into the main result
Theorem \ref{main}.  The entire proof is complete, except for
the Green function estimate furnished by  Appendix \ref{greenapp}
and the remainder of the proof of Lemma \ref{Yapplm1} in Appendix
\ref{Yapp}.  

\appendix


\section{A Green function estimate}
\label{greenapp} 
This appendix can be read independently of the rest of the paper. 
Let us write a $d$-vector
in terms of coordinates  as $x=(x^1,\dotsc,x^d)$,
and similarly for random vectors  $X=(X^1,\dotsc,X^d)$. 

Let $\group$ be some subgroup of $\Z^d$. 
Let $Y=(Y_k)_{k\ge 0}$ be a Markov chain  on $\group$  with 
transition probability $q(x,y)$, and let  $\Ybar=(\Ybar_k)_{k\ge 0}$ be a
 symmetric  random walk  on $\group$ with 
transition probability $\qbar(x,y)=\qbar(y,x)=\qbar(0,y-x)$. 
Make the following assumptions.

\smallskip

(A.i) A finite third moment for the random walk: 
$E_0\lvert \Ybar_1\rvert^3<\infty$. 

\smallskip

(A.ii) Let $U_r=\inf\{n\ge 0: Y_n\notin [-r,r]^d\}$ be the 
exit time from a centered  cube of side length $2r+1$ for the 
Markov chain $Y$.
Then there is a constant $0<K<\infty$ such that 
\beq
\sup_{x\in [-r,r]^d}  E_x(U_r)\le K^r \quad\text{for all $r\ge 1$}.
\label{Y-ell-ass}
\eeq

\smallskip

(A.iii) For every $i\in\{1,\dotsc,d\}$, 
if  the one-dimensional random walk 
$\Ybar^i$ is degenerate in the sense that $\qbar(0,y)=0$ for $y^i\ne 0$,
then so is the process  $Y^i$
in the sense that $q(x,y)=0$ whenever $x^i\ne y^i$.  In other words, 
any coordinate that can move in the $Y$ chain somewhere in space
can also move in the $\Ybar$ walk.

\smallskip

(A.iv) For any initial state $x\ne0$  the transitions 
$q$ and $\qbar$ can be coupled  so that 
\beq
P_{x,x}\{Y_1\ne\Ybar_1\}\le C{\lvert x\rvert}^{-p_1}
\label{couplass}\eeq
where $0<C,p_1<\infty$ are constants independent of $x$ and
 $p_1 > 15$. 

\smallskip

Let $h$ be a function on $\group$
 such that $0\le h(x)\le C{(\lvert x\rvert\vee1)}^{-p_2}$ for constants 
$0<C, p_2<\infty$. 
This section is devoted
to proving the following Green function bound on the Markov chain. 

\begin{theorem} There are constants $0<C,\eta<\infty$ such 
that  
\[
\sum_{k=0}^{n-1} E_zh(Y_k) = \sum_y h(y) \sum_{k=0}^{n-1} P_z\{Y_k=y\}   
\le Cn^{1-\eta}
\]
for all $n\ge 1$ and $z\in\group$. 
If $p_1$ and $p_2$ can be taken arbitrarily large, 
then  $1-\eta$ can be taken arbitrarily close
to  (but still strictly
above) $1/2$.  
\label{greenthm1}
\end{theorem}

Precisely speaking, the bound that emerges is 
\beq
\sum_{k=0}^{n-1} E_zh(Y_k) 
\le Cn^{\{1-p_2/(2p_1-4)\}\vee\{(1/2)+ 13/(2p_1-4)\}}.  
\label{appAfinal}\eeq 
The remainder of the section proves the theorem. 
Throughout $C$ will change value but $p_1, p_2$ remain
the constants in the assumptions above.  

For the proof we can assume that each coordinate walk $\Ybar^i$
$(1\le i\le d)$  is 
nondegenerate.  For if the random walk has a degenerate coordinate  $\Ybar^j$
then Assumption (A.iii) implies that also for the 
Markov chain $Y^j_n=Y^j_0$ for all times $n\ge 0$.  Then we can 
project everything onto the remaining $d-1$ coordinates.  
Given the starting point $z$ of Theorem   \ref{greenthm1} write the 
Markov chain as  $Y_n=(z^j, Y'_n)$ where  $Y'_n$ is 
the  $\Z^{d-1}$-valued  Markov chain with 
transition    $q'(x',y')=q((z_j,x'),(z_j,y'))$.  
Take the $(d-1)$-dimensional 
 random walk $\Ybar'_n=(Y^1_n,\dotsc,Y^{j-1}_n,Y^{j+1}_n, \dotsc, Y^d_n)$.
Replace $h$ with $h'(x')= h(z_j,x')$. 
All the assumptions continue to hold with the same constants
because $\lvert x'\rvert\le \lvert (z_j,x')\rvert$ and the 
exit time from a cube only concerns the nondegenerate coordinates. 
The  constants from the assumptions determine the constants of the
theorem. Consequently the estimate of the theorem follows with
constants that do not depend on the frozen coordinate $z^j$. 

We begin by
 discarding terms outside a cube of side $r=n^{\e_1}$ for a small $\e_1>0$
that will be specified at the end of the proof. 
For convenience, use below the $\ell^1$ norm $\lvert\,\cdot\,\rvert_1$
 on $\Z^d$ because its
values are integers.  
\begin{align*}
&\sum_{|y|_1> n^{\e_1}} h(y) \sum_{k=0}^{n-1} P_z\{Y_k=y\} 
\le\sum_{k=0}^{n-1} \,\sum_{j\ge  [n^{\e_1}]+1} Cj^{-p_2}  
P_z\{\lvert Y_k\rvert_1= j\}\\
&\le\sum_{k=0}^{n-1}\,  Cn^{-p_2\e_1}   \sum_{j\ge  [n^{\e_1}]+1}
P_z\{\lvert Y_k\rvert_1= j\} 
\le  Cn^{1-p_2\e_1}. 
\end{align*}

Let \[B=[-n^{\e_1}, n^{\e_1}]^d.\]
Since $h$ is bounded,   it now remains to show that 
\beq
\sum_{k=0}^{n-1} P_z\{Y_k\in B\} \le Cn^{1-\eta}. 
\label{goal-Y-1}
\eeq
For this we can assume $z\in B$ since accounting for the time 
to enter $B$ can only improve the estimate.

Bound \eqref{goal-Y-1}  will be achieved in two stages. 
First we improve the assumed 
exponential exit time bound  \eqref{Y-ell-ass}  to a polynomial
bound.    Second, we show that often
enough $Y$ 
 follows the random walk $\Ybar$ during its excursions outside $B$.
The random walk excursions are long and thereby we obtain \eqref{goal-Y-1}.
Thus our first task is to construct a suitable coupling
of $Y$ and $\Ybar$. 

\begin{lemma}  Let $\zeta=\inf\{n\ge 1: \Ybar_n\in A\}$ be the 
first entrance time of the random walk $\Ybar$ 
into some set $A\subseteq\group$.
Then we can couple the Markov chain $Y$ and the random walk
 $\Ybar$ so that 
\[
P_{x,x}\{\text{ $Y_k\ne\Ybar_k$ for some $1\le k\le \zeta$ }\}
\le C E_x \Bigl[\;\sum_{k=0}^{\zeta-1} {\lvert\Ybar_k\rvert}^{-p_1}\Bigr].
\]
\label{YYbarlm1}
\end{lemma} 

The proof shows that the statement works also
if  $\zeta=\infty$ is possible, but we will not need this case.

\begin{proof} For each state $x$ create an i.i.d.~sequence
$(Z^x_k, \Zbar^x_k)_{k\ge 1}$ such that $Z^x_k$ has 
distribution $q(x,x+\,\cdot\,)$,  $\Zbar^x_k$ has 
distribution $\qbar(x,x+\,\cdot\,)=\qbar(0,\,\cdot\,)$, and 
each pair $(Z^x_k, \Zbar^x_k)$ is coupled so that 
$P(Z^x_k\ne \Zbar^x_k) \le C\lvert{x}\rvert^{-p_1}$.  
For distinct $x$ these sequences are independent. 

Construct the process $(Y_n,\Ybar_n)$ as follows: with 
counting measures 
\[L_n(x)=\sum_{k=0}^n \one\{Y_k=x\}
\quad\text{and}\quad 
 \Lbar_n(x)=\sum_{k=0}^n \one\{\Ybar_k=x\} \quad(n\ge 0) 
\]
and with initial point $(Y_0,\Ybar_0)$ given, define for $n\ge 1$ 
\[
Y_n=Y_{n-1}+ Z^{Y_{n-1}}_{L_{n-1}(Y_{n-1})}
\quad\text{and}\quad
\Ybar_n=\Ybar_{n-1}+ \Zbar^{\Ybar_{n-1}}_{\Lbar_{n-1}(\Ybar_{n-1})}.
\]

In words, every time the chain $Y$ visits a state $x$, it 
reads its next jump from a new variable $Z^x_k$ which is then
discarded and never used again.  And similarly for $\Ybar$.   
This construction has the property that, if $Y_k=\Ybar_k$ for 
$0\le k\le n$ with $Y_{n}=\Ybar_{n}=x$, 
then the next joint step  is 
$(Z^x_k,\Zbar^x_k)$ for $k=L_{n}(x)=\Lbar_{n}(x)$.  In other
words, given that the processes agree up to the present 
and reside together at $x$, the probability that they separate in
the next step is bounded by $C\abs{x}^{-p_1}$.

Now follow self-evident steps. 
\begin{align*}
&P_{x,x}\{\text{ $Y_k\ne\Ybar_k$ for some $1\le k\le \zeta$ }\}\\
&\le \sum_{k=1}^\infty
 P_{x,x}\{\text{ $Y_j=\Ybar_j\in A^c$ for $1\le j<k$,   $Y_k\ne\Ybar_k$ } \}\\
&\le \sum_{k=1}^\infty
 E_{x,x}\bigl[\one\{\text{ $Y_j=\Ybar_j\in A^c$ for $1\le j<k$ }\}
P_{Y_{k-1},\Ybar_{k-1}}\{ Y_1\ne\Ybar_1\} \,  \bigr]\\
&\le C\sum_{k=1}^\infty
 E_{x,x}\bigl[\one\{\text{ $Y_j=\Ybar_j\in A^c$ for $1\le j<k$ }\}
{\lvert \Ybar_{k-1}\rvert}^{-p_1} \,   \bigr]\\
&\le C E_x \sum_{m=0}^{\zeta-1} {\lvert\Ybar_m\rvert}^{-p_1}.
\qedhere
\end{align*}
\end{proof}

For the remainder of this section
 $Y$ and $\Ybar$ are always coupled
in the manner that satisfies Lemma \ref{YYbarlm1}.

\begin{lemma}   Fix a coordinate index  $j\in\{1,\dotsc,d\}$.  
Let $r_0$ be a positive integer and  $\wbar=\inf\{n\ge 1: \Ybar_n^j\le r_0\}$ 
 the first time the random walk $\Ybar$
 enters the half-space 
$\cH=\{x: x^j\le r_0\}$. Couple $Y$ and $\Ybar$ starting from a
common initial point $x\notin\cH$.
Then there is a constant $C$ independent of $r_0$ such that
\[
\sup_{x\notin\cH } P_{x,x}\{\text{ $Y_k\ne\Ybar_k$ for some
 $k\in\{1,\dotsc,\wbar\}$ }\} \le  Cr_0^{2-p_1}
\quad \text{ for all $r_0\ge 1$.} 
\]
The same result holds for $\cH=\{x: x^j\ge -r_0\}$.
\label{YY-aux-lm-1}
\end{lemma}

\begin{proof} By Lemma \ref{YYbarlm1} 
\begin{align*}
&P_{x,x}\{\text{ $Y_k\ne\Ybar_k$ for some
 $k\in\{1,\dotsc,\wbar\}$ }\}
\le  CE_x\biggl[\; \sum_{k=0}^{\wbar-1} 
\lvert \Ybar_k\rvert^{-p_1}\,\biggr]\\
&\qquad\qquad\qquad 
\le CE_{x^j}\biggl[\; \sum_{k=0}^{\wbar-1} 
 \lvert\Ybar_k^j\rvert^{-p_1}\,\biggr]
=C\sum_{t=r_0+1}^\infty t^{-p_1} g(x^j,t) 
\end{align*}
where  for $s,t\in [r_0+1,\infty)$
\[
g(s,t)= \sum_{n=0}^\infty P_s\{\Ybar_n^j=t\,,\,\wbar>n\} 
\]
is the Green function of the half-line $(-\infty, r_0]$ for 
the one-dimensional random walk $\Ybar^j$.
This is the expected
number of visits to $t$ before entering $(-\infty,r_0]$,
defined on p.~209 in Spitzer \citep{spitzer}.  The development
in Sections 18 and 19 in \citep{spitzer} gives the bound 
\beq
g(s,t)\le C(1+(s-r_0-1)\wedge(t-r_0-1))\le C(t-r_0),\quad 
s,t\in [r_0+1,\infty). 
\label{spitz-1}\eeq

Here is some more detail. Shift $r_0+1$ to the origin to 
match the setting in \citep{spitzer}.  Then    P19.3 on p.~209
gives
\[
g(x,y)= \sum_{n=0}^{x\wedge y} u(x-n)v(y-n)\qquad \text{for $x,y\ge 0$}
\]
where the functions $u$ and $v$ are defined on p.~201. 
For a symmetric random walk $u=v$ (E19.3 on p.~204). 
P18.7 on p.~202 implies that
\[
v(m)=\frac1{\sqrt{c}} \sum_{k=0}^\infty 
\mathbf{P}\{\mathbf{Z}_1+\dotsm+\mathbf{Z}_k=m\}
\]
where $c$ is a certain constant
and  $\{\mathbf{Z}_i\}$ are i.i.d.~strictly positive, integer-valued
ladder variables for the underlying random walk. (For $k=0$
the sum $\mathbf{Z}_1+\dotsm+\mathbf{Z}_k$ is identically
zero.)  
Now $v(m)\le v(0)$ for each $m$ because the $\mathbf{Z}_i$'s are
strictly positive.  (Either do induction on $m$, or note that
for a particular realization of the sequence $\{\mathbf{Z}_i\}$
a given $m$ can be attained for at most one value of $k$.)   
So the quantities $u(m)=v(m)$ are bounded. This 
justifies \eqref{spitz-1}. 

Continuing from further above we get the estimate claimed in the 
statement of the lemma: 
\[
E_x\biggl[\; \sum_{k=0}^{\wbar-1} 
\lvert \Ybar_k\rvert^{-p_1}\,\biggr]
\le C\sum_{t>r_0} (t-r_0) t^{-p_1}
\le C r_0^{2-p_1}. 
\qedhere
\]
\end{proof} 

For the next lemmas abbreviate $B_r=[-r,r]^d$ for $d$-dimensional
centered cubes. 

\begin{lemma} 
 There exist 
constants 
$0<\alpha_1, A_1 <\infty$ such that
\beq
\inf_{x\in B_r\smallsetminus B_{r_0}} 
P_x\{ \text{without entering $B_{r_0}$ chain $Y$ exits $B_r$ by time
 $A_1r^{3}$}\}
\ge \frac{\alpha_1}r
\label{YY-aux-1.5}
\eeq
 for large enough  positive integers $r_0$ and $r$ that  satisfy 
 \[r^{2/(p_1-2)}\le  r_0< r.\] 
\label{YY-aux-lm-2}\end{lemma}

\begin{proof} 
A point 
 $x\in B_r\smallsetminus B_{r_0}$ has a coordinate
 $x^j\in[-r, -r_0-1]\cup[r_0+1,r]$. 
The same argument works for
  both alternatives, and 
we  treat the case
$x^j\in[r_0+1,r]$.  

One way to realize the event in \eqref{YY-aux-1.5}  is this:  starting at 
$x^j$,
the $\Ybar^j$ walk exits  $[r_0+1,r]$ by time $A_1r^3$ through the right 
boundary into $[r+1,\infty)$, 
and $Y$ and $\Ybar$ stay coupled together throughout this time. 
Let $\zetabar$ be the time $\Ybar^j$ exits $[r_0+1,r]$ and $\wbar$ the 
time $\Ybar^j$ enters $(-\infty, r_0]$.  Then $\wbar\ge\zetabar$. 
 Thus the 
complementary probability of \eqref{YY-aux-1.5} is bounded above by 
\beq
\begin{split}
&P_{x^j}\{\text{ $\Ybar^j$  exits  $[r_0+1,r]$ into $(-\infty, r_0]$ }\}\\
&\qquad 
+ \; P_{x^j}\{ \zetabar >A_1r^3\} \; +  \;  
P_{x,x}\{\text{ $Y_k\ne\Ybar_k$ for some
 $k\in\{1,\dotsc,\wbar\}$ }\}.
\end{split} 
\label{YY-aux-3}\eeq

We treat the terms one at a time. 
From the development on p.~253-255  in Spitzer \citep{spitzer} we get the 
bound
\beq
P_{x^j}\{\text{ $\Ybar^j$  exits  $[r_0+1,r]$ into $(-\infty, r_0]$ }\}
 \le 1-\frac{\alpha_2}r
\label{YY-aux-4}
\eeq
for a constant $\alpha_2>0$, uniformly over
$0<r_0< x^j\le r$. In some more detail:
P22.7 on p.~253, 
the inequality in the third display  of p.~255, and the third moment
assumption on the steps of $\Ybar$  give a lower bound 
\beq
P_{x^j}\{\text{ $\Ybar^j$  exits  $[r_0+1,r]$ into $[r+1,\infty)$ }\}
\ge \frac{x^j-r_0-1-c_1}{r-r_0-1}
\label{spitz-8}
\eeq
for the probability of exiting to the right.
Here $c_1$ is a constant that comes from the term denoted
in \citep{spitzer}  by
$  M\sum_{s=0}^N(1+s)a(s) $
whose finiteness follows from the third moment assumption.  
The text  on p.~254-255 suggests that these
steps need the aperiodicity assumption. This need for
aperiodicity  can be traced back via P22.5 to 
P22.4 which is used to assert the boundedness of 
$u(x)$ and $v(x)$.  But as we observed above in the derivation
of \eqref{spitz-1} boundedness of $u(x)$ and $v(x)$   is true
without any additional assumptions.  
 
 To go forward from \eqref{spitz-8} 
fix any $m>c_1$ so that the numerator
above is positive for $x^j=r_0+1+m$. 
The probability in \eqref{spitz-8}
 is minimized at $x^j=r_0+1$, and from $x^j=r_0+1$
there is a fixed positive
probability $\theta$ to take $m$ steps to the right to get
past the point $x^j=r_0+1+m$.  Thus for all $x^j\in[r_0+1,r]$ we get the lower bound 
\[
P_{x^j}\{\text{ $\Ybar^j$  exits  $[r_0+1,r]$ into $[r+1,\infty)$ }\}
\ge \frac{\theta(m-c_1)}{r-r_0-1} \ge \frac{\alpha_2}{r}
\] 
where $\alpha_2>0$ is a constant, 
and \eqref{YY-aux-4} is verified. 

As in  \eqref{spitz-1} let $g(s,t)$ be the Green function
of the random walk $\Ybar^j$  for the 
half-line $(-\infty, r_0]$,  and 
let $\tilde{g}(s,t)$ be the Green function for the complement
of the interval 
$[r_0+1,r]$. Then $\tilde{g}(s,t)\le g(s,t)$, and by \eqref{spitz-1}
we get this moment bound:
\begin{align*}
 E_{x^j}[\,\zetabar\,]=\sum_{t=r_0+1}^r \tilde{g}(x^j,t) 
\le \sum_{t=r_0+1}^r g(x^j,t) 
\le Cr^2.
\end{align*}
Consequently, 
 uniformly over $x^j\in [r_0+1,r]$, 
\beq
 P_{x^j}[ \zetabar >A_1r^3] \le \frac{C}{A_1r}.
\label{YY-aux-5} \eeq

From Lemma \ref{YY-aux-lm-1}  
\beq
P_x\{\text{ $Y_k\ne\Ybar_k$ for some
 $k\in\{1,\dotsc,\wbar\}$ }\} \le  Cr_0^{2-p_1}.
\label{YY-aux-6}\eeq

Putting bounds \eqref{YY-aux-4}, \eqref{YY-aux-5}
and \eqref{YY-aux-6} together gives an upper bound of 
\[
1\;-\;\frac{\alpha_2}r \;+\; \frac{C}{A_1r} \;+\;  Cr_0^{2-p_1}  
\]
for the sum in \eqref{YY-aux-3} which bounds the complement of the 
probability in \eqref{YY-aux-1.5}.  By assumption 
$r_0^{2-p_1}\le r^{-2}$.   So if $A_1$ is fixed large enough,
then the sum above is not more than $1-\alpha_1/r$ for a 
constant $\alpha_1>0$,  for all large enough $r$  .  
\end{proof}

We iterate the last estimate to get down to an iterated logarithmic cube. 

\begin{corollary} 
Fix a constant $c_1>1$ 
and consider positive integers $r_0$ and $r$ that  satisfy 
 \[\log\log r\le  r_0\le c_1\log\log r< r.\]  Then  for large enough $r$   
\beq
\inf_{x\in B_r\smallsetminus B_{r_0}} 
P_x\{ \text{ without entering $B_{r_0}$ chain $Y$ exits $B_r$ by time
 $r^{4}$ }\}
\ge r^{-3}.
\label{YY-aux-1.7}
\eeq
\label{YY-aux-cor-2.1}\end{corollary}

\begin{proof}
Consider $r$ large enough so that $r_0$ is also large enough
to play the role of $r$ in Lemma \ref{YY-aux-lm-2}. 
Pick an integer $\gamma$ such that 
 $3\le \gamma\le (p_1-2)/2$. 
Put  $r_k=r_0^{\gamma^k}$
for $k\ge 0$ ($r_0$ is still $r_0$) and $t_n=A_1\sum_{k=1}^n r_0^{3\gamma^{k}}$
where $A_1$ is the constant from Lemma \ref{YY-aux-lm-2}.  

We claim that  for $n\ge 1$
\begin{align}
&\inf_{x\in B_{r_n}\smallsetminus B_{r_{0}}} 
P_x\{ \text{without entering $B_{r_0}$ chain $Y$ exits $B_{r_n}$ by time
 $t_n$}\}
\ge \prod_{k=1}^{n}\Bigl(\frac{\alpha_1}{r_k}\Bigr).
\label{YY-aux-1.71}
\end{align}
Here $\alpha_1$ is the constant coming from \eqref{YY-aux-1.5}
and we can assume  $\alpha_1\le 1$. 

We prove \eqref{YY-aux-1.71} by induction.  The case $n=1$ is Lemma 
\ref{YY-aux-lm-2} applied to  $r_1=r_0^\gamma$ and $r_0$. 
The inductive step comes
from the Markov property. Assume \eqref{YY-aux-1.71} is true for $n$
and consider exiting $B_{r_{n+1}}$ without entering $B_{r_0}$. 

\smallskip

(i) If the initial state $x$ lies in $B_{r_n}\smallsetminus B_{r_{0}}$
then by induction 
the chain first takes time $t_n$ to exit $B_{r_n}$ without entering
$B_{r_0}$ with probability bounded below by  
$\prod_{k=1}^{n}({\alpha_1}/{r_k})$.  If the walk landed
in  $B_{r_{n+1}}\smallsetminus B_{r_n}$ take another time 
$A_1r_{n+1}^3=A_1r_0^{3\gamma^{n+1}}$ to exit $B_{r_{n+1}}$ without 
entering  $B_{r_n}$ with 
probability at least  $\alpha_1/r_{n+1}$ (Lemma \ref{YY-aux-lm-2} again). 
The times taken add up to $t_{n+1}$ and the probabilities 
multiply to $\prod_{k=1}^{n+1}({\alpha_1}/{r_k})$.

\smallskip

(ii)  If the initial state $x$ lies in $B_{r_{n+1}}\smallsetminus B_{r_n}$
then apply Lemma \ref{YY-aux-lm-2} to exit $B_{r_{n+1}}$ without 
entering  $B_{r_n}$ in time $A_1r_{n+1}^3=A_1r_0^{3\gamma^{n+1}}$   with 
probability at least  $\alpha_1/r_{n+1}$. 

\smallskip

This completes the inductive proof of \eqref{YY-aux-1.71}.

Let $N=\min\{k\ge 1: r_k\ge r\}$. Then   $r_0^{\gamma^{N-1}}<r$.
If $r$ is large enough, and in particular  $r_0$ is large enough
to make $\log\log r_0 >0$, then also   
$N<1+(\log\log r)/(\log \gamma)<2\log\log r$. 

To prove the corollary take first  $n=N-1$ in \eqref{YY-aux-1.71}.
This gets the chain Y out of $B_{r_{N-1}}$ without entering $B_{r_0}$.
If $Y$ landed in $B_r\smallsetminus B_{r_{N-1}}$, 
 apply Lemma \ref{YY-aux-lm-2} once more to take $Y$ out of
$B_r$ without entering  $B_{r_{N-1}}$.  The probabilitry of
achieving this is bounded below by
\begin{align*}
\prod_{k=1}^{N-1}\Bigl(\frac{\alpha_1}{r_k}\Bigr)\cdot\frac{\alpha_1}{r}  
\ge \alpha_1^N r_0^{-\frac{\gamma^{N}}{\gamma-1}}r^{-1} 
\ge (\log r)^{2\log \alpha_1} r^{-\frac{\gamma}{\gamma-1}-1}
\ge r^{-3}
\end{align*}
where again we  required large enough $r$. 
For the time elapsed we get the bound 
\[
t_{N-1}+A_1r^3 \le A_1(N-1)r_0^{3\gamma^{N-1}} 
 + A_1r^3 
\le r^4 
\]
  for large enough $r$. 
\end{proof}

The reader can see that the exponents in the previous lemmas can 
be tightened.  But in the end the exponents still get rather large 
so we prefer to keep the statements and proofs simple for readability.  
We come to one of the main auxiliary lemmas of this development. 

\begin{lemma}  Let $U=\inf\{n\ge 0: Y_n\notin B_r\}$  be 
the first exit time from $B_r=[-r,r]^d$ for the Markov chain $Y$. 
 Then there exists a finite positive constant $C_1$ such that 
\[\sup_{x\in B_r}E_x[U]\le C_1r^{13}\quad\text{ for all $1\le r<\infty$.}\] 
\label{YY-aux-lm-6}
\end{lemma} 

\begin{proof} First observe that $\sup_{x\in B_r} E_x[U]<\infty$
by assumption \eqref{Y-ell-ass}. 
Throughout, let positive integers $r_0<r$ satisfy
 $\log\log r\le r_0\le 2\log\log r$ 
so that in particular  the assumptions of Corollary
\ref{YY-aux-cor-2.1} are satisfied.  
Once the statement is proved for large enough
$r$, we obtain it for all $r\ge 1$ by increasing $C_1$. 

Let $0=T_0=S_0\le T_1\le S_1\le T_2\le\dotsm$ be the successive
exit and  entrance
times into $B_{r_0}$. Precisely,
 for $i\ge 1$ as long as $S_{i-1}<\infty$ 
\[
T_i=\inf\{ n\ge S_{i-1}: Y_n\notin B_{r_0}\}
\quad\text{and}\quad
S_i=\inf\{ n\ge T_{i}: Y_n\in B_{r_0}\}
\]
Once $S_i=\infty$ then we set $T_j=S_j=\infty$ for all $j>i$. 
If $Y_0\in B_r\smallsetminus B_{r_0}$ then also $T_1=0$. 
From  assumption \eqref{Y-ell-ass} 
\beq 
\sup_{x\in B_{r_0}}E_x[T_1]\le K^{r_0}\le (\log r)^{2\log K}.  
\label{ell-escape}\eeq
So a priori $T_1$ is finite but $S_1=\infty$ is possible. 
Since $T_1\le U<\infty$ we can decompose as follows,
for $x\in B_r$: 
\beq
\begin{split}
E_x[U]&= \sum_{j=1}^\infty E_x[U,\, T_j\le U<S_j]\\
 &= \sum_{j=1}^\infty E_x[T_j\,,\, T_j\le U<S_j] + 
\sum_{j=1}^\infty E_x[U-T_j\,,\, T_j\le U<S_j].
\end{split}
\label{YY-aux-11}\eeq

\def\mm{4}    
\def\bb{3}    
\def\abb{7}  
\def\cbb{10} 

We first treat the last sum in \eqref{YY-aux-11}. By an inductive
application of  Corollary \ref{YY-aux-cor-2.1}, 
for any $z\in B_r\smallsetminus B_{r_0}$,
\beq\begin{split}
&P_z\{U>jr^{\mm},\, U<S_1\} \le P_z\{\text{ $Y_k\in B_r\smallsetminus B_{r_0}$
for $k\le jr^{\mm}$ }\}\\
&\qquad=E_z\Big[\one\{\text{ $Y_k\in B_r\smallsetminus B_{r_0}$          
for $k\le (j-1)r^{\mm}$ }\}
P_{Y_{(j-1)r^{\mm}}}\{ \text{ 
$Y_k\in B_r\smallsetminus B_{r_0}$  for $k\le r^{\mm}$ }\}\,\Big]\\
&\qquad\le \dotsm \le (1-r^{-\bb})^j.
\end{split}
\label{YY-aux-11.5}\eeq 
Utilizing this, still for $z\in B_r\smallsetminus B_{r_0}$,
\beq\begin{split}
E_z[ U,\,U<S_1]&=\sum_{m=0}^\infty P_z\{U>m\,,\,U<S_1\}\\
&\le r^{\mm} \sum_{j=0}^\infty P_z\{ U>jr^{\mm}\,,\,U<S_1\} \le 
r^{\abb}.
\end{split}\label{YY-aux-12}\eeq 
Next we take into consideration the failure to exit $B_r$ 
during the earlier excursions in $B_r\smallsetminus B_{r_0}$.
Let 
\[ H_i=\{ \text{$Y_n\in B_r$ for $T_i\le n<S_i$} \}  \]
be the event that in between the $i$th exit from $B_{r_0}$ and 
entrance back into $B_{r_0}$ the chain $Y$ does not exit $B_r$. 
We shall repeatedly use this  consequence of Corollary \ref{YY-aux-cor-2.1}:
 \beq \text{  for $i\ge 1$, on the event $\{T_i<\infty\}$,   
 $P_x\{ H_i\,\vert\,\cF_{T_i}\}\le 1-r^{-\bb}$.  
}\label{YY-aux-14}\eeq

  Here is the 
first instance. 
\begin{align*}
&E_x[U-T_j\,,\, T_j\le U<S_j]
=
E_x\Bigl[\; \prod_{k=1}^{j-1}\one_{H_k} \cdot\one\{T_j<\infty\}
\cdot E_{Y_{T_j}}(U,\,U<S_1)\Bigr]\\
&\le  r^{\abb} E_x\Bigl[ \;\prod_{k=1}^{j-1}\one_{H_k}
\cdot\one\{T_{j-1}<\infty\} \Bigr]
\le  r^{\abb} (1-r^{-\bb})^{j-1}.
\end{align*}
Note that if $Y_{T_j}$ above lies outside $B_r$ then 
$E_{Y_{T_j}}(U)=0$.  In the other case $Y_{T_j}\in 
 B_r\smallsetminus B_{r_0}$ 
and \eqref{YY-aux-12} applies. 
So for  the last sum in \eqref{YY-aux-11}:
\beq
\sum_{j=1}^\infty E_x[U-T_j\,,\, T_j\le U<S_j]
\le \sum_{j=1}^\infty  r^{\abb} 
(1-r^{-\bb})^{j-1}
\le r^{\cbb}.
\label{YY-aux-17}
\eeq

We turn to the second-last sum in \eqref{YY-aux-11}. 
Separate the $i=0$ term from the sum below and use
 \eqref{ell-escape} and \eqref{YY-aux-14}: 
\beq
\begin{split}
&E_x[T_j\,,\, T_j\le U<S_j] \ \le\  
\sum_{i=0}^{j-1} E_x\Bigl[\; \prod_{k=1}^{j-1}\one_{H_k} 
\cdot\one\{T_j<\infty\}\cdot
 (T_{i+1}-T_i)\Bigr] \\
&\le \  (\log r)^{2\log K} (1-r^{-\bb})^{j-1}  \\
&\quad   +\; 
 \sum_{i=1}^{j-1} E_x\Bigl[\; \prod_{k=1}^{i-1}\one_{H_k}\cdot
(T_{i+1}-T_i)\one_{H_i} \cdot\one\{T_{i+1}<\infty\} \Bigr] 
(1-r^{-\bb})^{j-1-i} .
\end{split}
\label{YY-aux-18}\eeq
Split the last expectation  as
\begin{align}
&E_x\Bigl[\; \prod_{k=1}^{i-1}\one_{H_k}\cdot              
(T_{i+1}-T_i)\one_{H_i}\cdot\one\{T_{i+1}<\infty\} \Bigr]\nn\\
&\le
E_x\Bigl[\; \prod_{k=1}^{i-1}\one_{H_k}\cdot              
(T_{i+1}-S_i)\one_{H_i} \cdot\one\{S_{i}<\infty\}\Bigr]
+\; E_x\Bigl[\; \prod_{k=1}^{i-1}\one_{H_k}\cdot              
(S_i-T_i)\one_{H_i}\cdot\one\{T_{i}<\infty\} \Bigr]\nn\\
&\le 
E_x\Bigl[\; \prod_{k=1}^{i-1}\one_{H_k}\cdot\one\{S_{i}<\infty\}\cdot              
E_{Y_{S_i}}(T_1) \Bigr]
+\; E_x\Bigl[\; \prod_{k=1}^{i-1}\one_{H_k}\cdot \one\{T_{i}<\infty\}\cdot             
E_{Y_{T_i}}(S_1\cdot\one_{H_1}) \Bigr]\nn\\
&\le E_x\Bigl[\; \prod_{k=1}^{i-1}\one_{H_k}\cdot \one\{T_{i-1}<\infty\}\Bigr]
\bigl((\log r)^{2\log K} + r^{\abb}\bigr)\nn\\
&\le (1-r^{-\bb})^{i-1} \bigl((\log r)^{2\log K} + 
r^{\abb}\bigr).              
\label{YY-aux-20}
\end{align}
In the second-last inequality above,
 before applying \eqref{YY-aux-14} 
to the $H_k$'s,  
$E_{Y_{S_i}}(T_1)\le (\log r)^{2\log K}$ comes
from \eqref{ell-escape}. The other expectation is 
estimated by iterating Corollary \ref{YY-aux-cor-2.1}  again
 with $z\in B_r\smallsetminus B_{r_0}$, as was done in calculation
\eqref{YY-aux-11.5}:
\begin{align*}
E_{z}[S_1\cdot\one_{H_1}]
&=\sum_{m=0}^\infty P_z\{S_1>m\,,\,H_1\}
\le\sum_{m=0}^\infty P_z\{\text{ $Y_k\in B_r\smallsetminus B_{r_0}$
for $k\le m$ }\}\\
&\le r^{\mm}\sum_{j=0}^\infty P_z\{\text{ $Y_k\in B_r\smallsetminus B_{r_0}$
for $k\le jr^{\mm}$ }\}
\le r^{\abb}. 
\end{align*} 
Insert the bound from line \eqref{YY-aux-20} back up into
\eqref{YY-aux-18} to get the bound 
\begin{align*}
E_x[T_j\,,\, T_j\le U<S_j] \le 
 (2 (\log r)^{2\log K} + r^{\abb})j   (1-r^{-\bb})^{j-2}.
\end{align*}
Finally, bound  the second-last sum in \eqref{YY-aux-11}:
\begin{align*}
\sum_{j=1}^\infty E_x[T_j\,,\, T_j\le U<S_j] 
\le \bigl(2 (\log r)^{2\log K}r^{6} + r^{13}\bigr)
 (1-r^{-\bb})^{-1}.
\end{align*}
Take $r$ large enough so that $r^{-\bb}<1/2$. 
Combine the above bound
 with \eqref{YY-aux-11} and \eqref{YY-aux-17} 
  to get 
\[
E_x[U]\le r^{\cbb} + 
4 (\log r)^{2\log K}r^{6} + 2r^{13} \le 4r^{13} 
\]
when $r$ is large enough.
\end{proof} 

For the remainder of the proof we work with 
$B=B_r$ for $r=n^{\e_1}$. 
The above estimate gives us one part of the argument for 
\eqref{goal-Y-1}, namely that the Markov chain $Y$ exits
$B=[-n^{\e_1}, n^{\e_1}]^d$ fast enough.

Let $0=V_0<U_1<V_1<U_2<V_2<\dotsm$ be the successive entrance
times $V_i$ into $B$  and exit times $U_i$ from $B$ for the 
Markov chain $Y$, assuming that $Y_0=z\in B$. It is possible that 
some $V_i=\infty$.  But if $V_i<\infty$ then also $U_{i+1}<\infty$
due to assumption \eqref{Y-ell-ass},  as already observed.
The time intervals spent in $B$ are $[V_i, U_{i+1})$ each of length
at least 1. Thus, by applying Lemma \ref{YY-aux-lm-6}, 
\beq\begin{split}
\sum_{k=0}^{n-1} P_z(Y_k\in B) &\le 
\sum_{i=0}^n E_z\bigl[\, (U_{i+1}-V_i) \one\{V_i\le n\}\bigr]\\
&\le \sum_{i=0}^n E_z\bigl[\, E_{Y_{V_i}} (U_1) \one\{V_i\le n\}\bigr]\\
&\le Cn^{13\e_1} E_z\biggl[\, \sum_{i=0}^n   \one\{V_i\le n\}\biggr].
\end{split} 
\label{temp-Y-2}
\eeq

Next we bound the expected number of returns to $B$ by the number of
 excursions outside $B$ that fit in a time of length $n$:
\begin{align} 
E_z\biggl[\, \sum_{i=0}^n   \one\{V_i\le n\}\biggr]
&=
E_z\biggl[\, \sum_{i=0}^n   \one\Bigl\{\,
\sum_{j=1}^i (V_j-V_{j-1})\le n\Bigr\}\biggr]\nn\\
&\le E_z\biggl[\, \sum_{i=0}^n   \one\Bigl\{\,
\sum_{j=1}^i (V_j-U_{j})\le n\Bigr\}\biggr]. 
\label{line-a7}
\end{align}

According to the usual notion of stochastic dominance, we say
 the random vector $(\xi_1,\dotsc,\xi_n)$
dominates $(\eta_1,\dotsc,\eta_n)$ if 
\[ Ef(\xi_1,\dotsc,\xi_n)\ge Ef(\eta_1,\dotsc,\eta_n) \]
for any function $f$ that is 
coordinatewise nondecreasing.  If the process
 $\{\xi_i:1\le i\le n\}$ is adapted to the filtration
$\{\cG_i:1\le i\le n\}$, and 
$P[\xi_i> a\vert\cG_{i-1}]\ge 1-F(a)$ for some
distribution function $F$,  then the 
 $\{\eta_i\}$ can be taken i.i.d.\ $F$-distributed.  

\begin{lemma} There exist positive constants $c_1$, $c_2$
 such that the following holds:   
the excursion lengths $\{V_j-U_j:1\le j\le n\}$ 
stochastically dominate i.i.d.\ variables $\{\eta_j\}$ whose
common distribution satisfies 
$\bfP\{\eta\ge a\}\ge c_1a^{-1/2}$ for $1\le a\le c_2n^{2\e_1(p_1-2)}$. 
\label{stoch-lm} 
\end{lemma} 

\begin{proof} 
Since $P_z\{V_j-U_j\ge a\vert \cF_{U_j}\}=P_{Y_{U_j}}\{V\ge a\}$
where $V$ means first entrance time into $B$, we shall bound
$P_x\{V\ge a\}$ below uniformly over 
$x\notin B$. 
Fix such an $x$ and an index $1\le j\le d$ such that $x^j\notin[-r,r]$. 
  As before we work through 
the case $x^j>r$ because the argument for the other case
$x^j<-r$  is the same.

Let $\wbar=\inf\{n\ge 1: \Ybar^j_n\le r\}$ 
be the first time the one-dimensional random walk $\Ybar^j$
 enters the half-line 
$(-\infty, r]$.   If both $Y$ and $\Ybar$ start at $x$ and 
stay coupled together  until time $\wbar$, then $V\ge\wbar$.  This 
 way we bound $V$ from below.   Since the random walk
is symmetric and 
can be translated, we can move the origin to $x^j$ and use 
classic results about the first entrance  time
into the left half-line,  $\Tbar=\inf\{ n\ge 1: \Ybar^j_n<0\}$. 
Thus 
\beq\label{rwtail}
P_{x^j}\{\wbar\ge a\}\ge P_{r+1}\{\wbar\ge a\}= P_0\{\Tbar\ge a\}
\ge \frac{\alpha_5}{\sqrt{a}}
\eeq
for a constant $\alpha_5$. 
The last inequality follows for  one-dimensional symmetric walks
 from basic random walk theory. For example, combine 
equation (7) on p.~185 of \citep{spitzer} with a Tauberian theorem 
such as Theorem 5 on p.~447 of Feller \citep{fell-2}.   Or see directly 
Theorem 1a on p.~415 of \citep{fell-2}.  

Now start both $Y$ and $\Ybar$ from $x$.
 Apply Lemma 
\ref{YY-aux-lm-1} and recall that  $r=n^{\e_1}$.   
\begin{align*}
P_x\{V\ge a\}&\ge P_{x,x}\{V\ge a, \text{ $Y_k=\Ybar_k$ for $k=1,\dotsc,\wbar$ }\}\\
&\ge P_{x,x}\{\wbar\ge a, \text{ $Y_k=\Ybar_k$ for $k=1,\dotsc,\wbar$ }\}\\
&\ge P_{x^j}\{\wbar\ge a\} - P_{x,x}\{\text{ $Y_k\ne\Ybar_k$ for some
 $k\in\{1,\dotsc,\wbar\}$ }\}\\
&\ge \frac{\alpha_5}{\sqrt{a}} - Cn^{\e_1(2-p_1)} \ge \frac{\alpha_5}{2\sqrt{a}}   
\end{align*}
if $a\le \alpha_5^2(2C)^{-2}n^{2\e_1(p_1-2)}$. 
This lower bound is independent 
of $x$.  We have proved the lemma. 
\end{proof}

We can assume that the random variables $\eta_j$ given by
the lemma satisfy $1\le \eta_j\le c_2n^{2\e_1(p_1-2)}$, and we can 
assume that $c_2\le 1$ and 
$\e_1$ is small enough to have 
\beq  2\e_1(p_1-2)\le 1 \label{etap1}\eeq
 because this merely weakens the conclusion 
of the lemma. 
 For the renewal process
determined by $\{\eta_j\}$ write  
\[
S_0=0\, ,\; S_k=\sum_{j=1}^k \eta_j\,,
\quad\text{and}\quad  K(n)=\inf\{ k: S_k > n\}
\]
for the renewal times 
and the number of renewals up to time $n$ (counting the 
renewal $S_0=0$).  
Since the random variables are
bounded, Wald's identity gives 
\[  \bfE K(n) \cdot \bfE\eta = 
\bfE S_{K(n)} \le n+c_2n^{2\e_1(p_1-2)} \le 2n,
\]
while 
\[
\bfE\eta \ge \int_1^{c_2n^{2\e_1(p_1-2)}}  \frac{c_1}{\sqrt{s}}\,ds \ge c_3 
n^{{\e_1(p_1-2)}}.
\]
Together these give 
\[ \bfE K(n) \le \frac{2n}{\bfE\eta} \le C_2n^{1-{\e_1(p_1-2)}}. 
\]

Now we pick up the development from line \eqref{line-a7}. 
Since the negative of the function of $(V_j-U_j)_{1\le i\le n}$ 
in the expectation on line \eqref{line-a7} is nondecreasing, 
the stochastic domination of Lemma \ref{stoch-lm} gives
 an upper  bound of  \eqref{line-a7} in terms of the i.i.d.\ $\{\eta_j\}$.  
Then 
we use the renewal bound from  above.
\begin{align*} 
E_z\biggl[\, \sum_{i=0}^n   \one\{V_i\le n\}\biggr]
&\le E_z\biggl[\, \sum_{i=0}^n   \one\Bigl\{\,
\sum_{j=1}^i (V_j-U_{j})\le n\Bigr\}\biggr]\\
&\le \bfE\biggl[\, \sum_{i=0}^n   \one\Bigl\{\,
\sum_{j=1}^i \eta_j\le n\Bigr\}\biggr]
=  \bfE K(n) \le C_2n^{1-\e_1(p_1-2)}.  
\end{align*}
Returning back to \eqref{temp-Y-2} to collect the bounds, we
have shown that 
\begin{align*}
\sum_{k=0}^{n-1} P_z\{Y_k\in B\} &\le 
Cn^{13\e_1} E_z\biggl[\, \sum_{i=0}^n   \one\{V_i\le n\}\biggr] 
\le Cn^{1+13\e_1-\e_1(p_1-2)}
= Cn^{1-\eta}.
\end{align*}
Since $p_1>15$ by assumption, $\eta=\e_1(p_1-15)>0$. 
We can  satisfy 
\eqref{etap1} with $\e_1=(1/2)(p_1-2)^{-1}$ in which case
the last bound is $Cn^{(1/2)+ 13/(2p_1-4)}$. 

\section{Replacing direction of transience}
\label{vectorapp} 

Hypotheses \eqref{dir-trans} and \mom are made for a specific 
vector $\uhat$. 
This appendix shows that, at the expense of a further
factor in the moment required,  the assumption that
 $\uhat$  has integer coordinates entails no
loss of generality. 
This appendix also uses the assumption
\step that the magnitude of a step is bounded by $\M$. 
 We learned the proof below from
Berger and Zeitouni \cite{berg-zeit-07-}.

\def\regenw{\sigma} 
\def\mompp{{p_3}} 
\def\backw{\hat\beta} 

Assume some vector  $\what\in\R^d$ satisfies
 $P_0\{X_n\cdot\what\to\infty\}=1$.  Let $\{\regenw_k\}_{k\ge 0}$
be the regeneration times in the direction $\what$.
Assume $E_0(\regenw_1^{\mompp})<\infty$ for some $\mompp>6$.
As explained in Section \ref{prelim}, 
transience and moments on $\what$ imply the law of large numbers
\beq
\frac{X_n}{n}\to v=
\frac{E_0[X_{\regenw_1}\vert\backw=\infty]}
{E_0[{\regenw_1}\vert\backw=\infty]}
\qquad\text{$P_0$-almost surely,} 
\label{AAdefv}\eeq 
where $\backw=\inf\{n\ge 0: X_n\cdot\what<X_0\cdot\what\}$ is the
first backtracking time in the direction $\what$.  The limiting
velocity $v$ satisfies $\what\cdot v>0$.

\begin{proposition}  Suppose $\uhat\in\R^d$ satisfies
$\uhat\cdot v>0$. 
Then 
 \[P_0\{X_n\cdot\uhat\to\infty\}=1.\]
For  the first regeneration time $\tau_1$
 in the direction $\uhat$ we have the estimate
 $E_0(\tau_1^\momp)<\infty$ for $1\le \momp<\mompp/2-2$. 
\label{momtranspr}
\end{proposition} 

From this lemma we can choose a $\uhat$ with rational 
coordinates and then scale it by a suitable integer to get the
integer vector $\uhat$ assumed in 
 \eqref{dir-trans} and Hypothesis \amom.  To get $\momp>176d $
as required by \mom of course puts an even larger demand on $\mompp$. 

\begin{proof}[Proof of Proposition \ref{momtranspr}]

\medskip

{\bf Step 1.} Transience in direction $\uhat$.
  

Given $n$ choose $k=k(n)$ so that $\regenw_{k-1}<n\le \regenw_k$.
Then by the bounded step Hypothesis \step 
\beq
\Bigl\lvert \,\frac1n{X_n\cdot\uhat}\;-\; 
\frac1n{X_{\regenw_k}\cdot\uhat}\, \Bigr\rvert\  \le\ 
 \frac1n \lvert\uhat\rvert \M (\regenw_k-\regenw_{k-1}).
\label{AAaux5.5}\eeq 
By the moment assumption on $\regenw_1$ 
the right-hand side converges $P_0$-a.s. to zero 
while $n^{-1}{X_{\regenw_k}\cdot\uhat}\to v\cdot\uhat>0$, 
and so in particular $X_n\cdot\uhat\to\infty$. 
From this follows that the regeneration times $\{\tau_k\}$
in direction $\uhat$ are finite. 

\medskip

{\bf Step 2.} Moment bound on the height $X_{\tau_1}\cdot\uhat$ 
 of the first $\uhat$-regeneration slab.

Let $\beta$ be the $\uhat$-backtracking time 
as defined in \eqref{defbeta} and 
\[
M=\sup_{0\le n\le\beta} X_n\cdot\uhat.
\]
Lemma 1.2 in Sznitman \cite{szni-02} shows how the construction
of the regeneration time leads to 
stochastic domination of  $X_{\tau_1}\cdot\uhat$ under $P_0$  by a
sum of geometrically many i.i.d.\ terms, each distributed
like $M$ plus a fixed constant under the measure 
$P_0(\,\cdot\,\vert\beta<\infty)$.
Hence to prove $E_0[(X_{\tau_1}\cdot\uhat)^p\,]<\infty$ 
 it suffices to prove $E_0(M^p\vert\beta<\infty)<\infty$. 
We begin with a lemma that helps control the tail
probabilities  $P\{M>m\vert\beta<\infty\}$. For the arguments it turns out
convenient to 
multiply $m$ by the constant $\abs{\uhat}\M$.

\begin{lemma}
 There exist $\delta_0>0$ such that this holds:
if $\delta\in(0,\delta_0)$ there exists an $m_0=m_0(\delta)<\infty$ 
such that
for $m\ge m_0$ the event $\{M>m\abs{\uhat}\M,\,\beta<\infty\}$ 
lies in the union of these
three events:
\begin{align}
&\regenw_{[\delta m]}\ge m, \label{AAevent1}\\
&\regenw_{k}-\regenw_{k-1} \ge \delta k 
\quad\text{for some $k>[\delta m]$,} \label{AAevent2}\\
&\lvert X_{\regenw_{k}}-E_0(X_{\regenw_{k}})\rvert  \ge \delta k 
\quad\text{for some $k>[\delta m]$.} \label{AAevent4}
\end{align}
\label{AAauxlemma1}\end{lemma}

\begin{proof}[Proof of Lemma \ref{AAauxlemma1}]
Assume 
that  $\beta<\infty$ and  $M>m\abs{\uhat}\M$, 
but conditions \eqref{AAevent1}--\eqref{AAevent4} fail simultaneously. 
We derive  a contradiction  from this.  

Fix $k\ge 1$ so that 
\beq
\regenw_{k-1}< \beta\le\regenw_{k}.
\label{AAaux7}\eeq
Since the maximum step size is $\M$, at least $m$ steps are needed
to realize the event $M>m\abs{\uhat}\M$ and so $\beta>m$. Thus negating
\eqref{AAevent1} implies 
$k>\delta m$. 
The step bound, \eqref{AAaux7} and the negation of \eqref{AAevent2}
 imply
\beq
X_{\regenw_k}\cdot\uhat \le X_\beta\cdot\uhat + 
\abs{\uhat}\M (\regenw_k-\regenw_{k-1}) < \abs{\uhat}\M \delta k. 
\label{AAaux11}\eeq

Introduce the shorthands  
\beq 
a=E_0(X_{\regenw_1})\quad\text{and}\quad
b=E_0(\regenw_1\vert\backw=\infty). 
\label{AAauxdefab}\eeq
By the i.i.d.\ property of the regeneration slabs from the 
second one onwards [recall the discussion around \eqref{regenslab}]
$E_0(\regenw_k-\regenw_{k-1})=b$ and 
$E_0(X_{\regenw_k})=a+b(k-1)v$ for $k\ge 1$. 
Thus  negating \eqref{AAevent4} gives 
\beq
\begin{split}
X_{\regenw_k}\cdot\uhat &= \bigl(X_{\regenw_k}-a-b(k-1)v\bigr) \cdot\uhat
+a\cdot\uhat +b(k-1)v\cdot\uhat\\
&\ge -\delta k \abs{\uhat}- \abs{a}\cdot\abs{\uhat}+ b(k-1)v\cdot\uhat. 
\end{split}
\label{AAaux12}
\eeq 

Since $v\cdot\uhat>0$,  
comparison of \eqref{AAaux11} and \eqref{AAaux12} reveals that it
is possible to first fix $\delta>0$ small enough and then 
$m_0$ large enough so that, if $m\ge m_0$, then $k>\delta m$ 
forces a contradiction between 
\eqref{AAaux11} and \eqref{AAaux12}. 
This concludes the proof of Lemma \ref{AAauxlemma1}. 
\end{proof}

Next we observe that the union of 
\eqref{AAevent1}--\eqref{AAevent4} has probability $\le$
$Cm^{1-\mompp/2}$. The assumptions of  
$\what$-directional transience and $E_0(\regenw_1^{\mompp})<\infty$
imply that $P_0(\backw=\infty)>0$ and hence 
(by the i.i.d slab property again) for $k\ge 2$, 
\beq
E_0[(\regenw_k-\regenw_{k-1})^\mompp\,]
=E_0[\regenw_1^{\mompp}\vert\backw=\infty]
\le \frac{E_0(\regenw_1^{\mompp})}{P_0(\backw=\infty)}<\infty. 
\label{AAaux13}\eeq

For the next calculation, recall that for i.i.d.\ mean zero summands 
and $p\ge 2$ the Burkholder-Davis-Gundy inequality \cite{burk-73}
 followed by Jensen's inequality  gives
\[
E\Bigl[\,\Bigl\lvert\sum_{j=1}^n Z_j\Bigr\rvert^p\,\Bigr]
\le  
E\Bigl[\,\Bigl(\;\sum_{j=1}^n Z_j^2\Bigr)^{p/2}\,\Bigr]
\le
n^{p/2}E(Z_1^p).
\]
Recall $a$  and $b$ from \eqref{AAauxdefab}. 
Shrink $\delta$ further (this can be done at the expense of increasing
$m_0$ in Lemma \ref{AAauxlemma1}) so that $\delta b<1/4$. 
\beq
\begin{split}
P_0\{\regenw_{[\delta m]}\ge m\} &\le P_0\{\regenw_1\ge m/2\}
+ P_0\Bigl\{\; \sum_{k=2}^{[\delta m]} 
(\regenw_k-\regenw_{k-1}-b)\ge m/4 \Bigr\}
\le Cm^{-\mompp/2}.
\end{split} \label{AAaux14}\eeq

For the second estimate use \eqref{AAaux13}.
\[
\sum_{k>[\delta m]} P_0\{ \regenw_k-\regenw_{k-1}\ge \delta k\}
\le \sum_{k>[\delta m]} C(\delta k)^{-\mompp} 
\le Cm^{1-\mompp}.
\]

For the third estimate use Chebychev and
for the sum of i.i.d\ pieces  repeat the 
 Burkholder-Davis-Gundy estimate:
\beq
\begin{split}
&P_0\{\,\lvert X_{\regenw_{k}}-E_0(X_{\regenw_{k}})\rvert  \ge \delta k\}
\le 
 P_0\{\,\lvert X_{\regenw_{1}}-a\rvert  \ge \delta k/2\}\\
&\qquad +
P_0\{\,\lvert X_{\regenw_{k}}-X_{\regenw_{1}}-(k-1)bv\rvert  
\ge \delta k/2\}
\ \le \  Ck^{-\mompp/2}.
\end{split}\label{AAaux14.5}\eeq
Summing these bounds over $k>[\delta m]$ gives 
$Cm^{1-\mompp/2}$.

Collecting the above bounds for the events 
\eqref{AAevent1}--\eqref{AAevent4} and utilizing 
Lemma \ref{AAauxlemma1} gives the intermediate  bound
$P_0(M>m\vert\backw<\infty)\le Cm^{1-\mompp/2}$ for large enough $m$.
Hence $E_0(M^{p}\vert\backw<\infty)<\infty$ for $p<\mompp/2-1$.  
By the already mentioned appeal to Lemma 1.2 
in Sznitman \cite{szni-02}
we can conclude {\bf Step 2} with the bound
\beq
E_0[ (X_{\tau_1}\cdot\uhat)^p]<\infty \qquad
\text{for $p<\mompp/2-1$. } 
\label{AAaux15}\eeq

\medskip

{\bf Step 3.} Moment bound for $\tau_1$. 
We insert one more lemma. 

\begin{lemma}
For $\ell\ge 1$: 
\[
P_0\{ \text{$\lvert X_n-nv\rvert\ge \delta n$ for some $n\ge \ell$}\} 
\le C\ell^{1-\mompp/2}.
\]
\label{AAauxlemma2}
\end{lemma}

\begin{proof}[Proof of Lemma \ref{AAauxlemma2}] 
Fix a small $\eta>0$. 
\begin{align*}
&P_0\{ \text{$\lvert X_n-nv\rvert\ge \delta n$ for some $n\ge \ell$}\} \\
&\le P_0\{\regenw_{[\eta\ell]}\ge \ell\} \;+\;
\sum_{j>\eta\ell}
P_0\{ \text{$\lvert X_n-nv\rvert\ge \delta n$ for some 
$n\in[\regenw_{j-1},\regenw_j]$}\}.
\end{align*}
If $\eta$ is small enough
the first probability above is bounded by $C\ell^{-\mompp/2}$ as 
in \eqref{AAaux14}.  For a term in the sum, note first that if 
$\regenw_{j-1}\ge \eta j$ then the parameter $n$ in the probability
satisfies $n\ge \eta j$.  Then replace time $n$ with time 
$\regenw_j$ at the expense of an error of a constant times
$\regenw_j-\regenw_{j-1}$:
\begin{align*}
&P_0\{ \text{$\lvert X_n-nv\rvert\ge \delta n$ for some 
$n\in[\regenw_{j-1},\regenw_j]$}\}\\
&\le P_0\{\regenw_{j-1}<\eta j\} 
+ P_0\{ \lvert X_{\regenw_j}-\regenw_jv\rvert\ge \delta\eta j/2\}
+P_0\{\, (\abs{v}+\M)(\regenw_{j}-\regenw_{j-1}) \ge \delta\eta j/2\,\}.
\end{align*}
The first probability after the inequality gives again 
$Cj^{-\mompp/2}$ as in \eqref{AAaux14}
 if $\eta$ is small enough.   In the second one
the summands $X_{\regenw_j}-X_{\regenw_{j-1}}-(\regenw_j-\regenw_{j-1})v$
are i.i.d.\ mean zero for $j\ge 2$ so we can argue in the same
 spirit as in 
\eqref{AAaux14.5} to get $Cj^{-\mompp/2}$.  The last probability 
gives $Cj^{-\mompp}$ by the moments of $\regenw_j-\regenw_{j-1}$.
Adding the bounds gives the conclusion. 
\end{proof}

Now we finish the proof of Proposition \ref{momtranspr}. 
To get a contradiction, suppose that
$\momp<\mompp/2-2$ and 
 $E_0(\tau_1^{\momp})=\infty$. 
 Pick
 $\e\in(0,\mompp/2-2-\momp)$.  
Then  there exists 
a subsequence $\{k_j\}$ such that 
$P_0(\tau_1> k_j)\ge k_j^{-\momp-\e}$. 
With the above lemma and the choice of $\e$
we have, for large enough $k_j$
\begin{align*}
P_0\{\tau_1> k_j\,,\, &\text{$\lvert X_n-nv\rvert< \delta n$ 
for all $n\ge k_j$}\} 
\ge k_j^{-\momp-\e}-Ck_j^{1-\mompp/2}
\ge Ck_j^{-\momp-\e}
\end{align*} 
On the event above 
\[
X_{\tau_1}\cdot\uhat \ge \tau_1v\cdot\uhat -\delta\tau_1\abs{\uhat} 
\ge \delta_1 k_j
\]
for another small $\delta_1>0$ if $\delta$ is small enough.  Thus we 
have 
\[
P_0\{ X_{\tau_1}\cdot\uhat \ge \delta_1 k_j\}
\ge Ck_j^{-\momp-\e}.
\] 
From this 
\[
\delta_1^{-p}E_0[(X_{\tau_1}\cdot\uhat)^{p}\,]
\ge C \sum_j k_j^{p-1-\momp-\e}.
\]
This sum diverges and contradicts \eqref{AAaux15} if $p$ is
chosen to satisfy 
$\mompp/2-1>p\ge 1+\momp+\e$ which can be done by the earlier
choice of $\e$. 
  This contradiction implies
that  $E_0(\tau_1^{\momp})<\infty$ and completes the proof
of Proposition \ref{momtranspr}. 
\end{proof}

\section{Completion of a technical proof}
\label{Yapp}

In this appendix we finish the proof of  Lemma \ref{Yapplm1}
by deriving  the bounds 
\eqref{Yline-1}--\eqref{Yline-3}.  

\smallskip

{\sl Proof of \eqref{Yline-1}.}  
By Hypothesis (\HR), there exist two nonzero
vectors $w\ne z$ such that $\E\pi_{0,z}\pi_{0,w}>0$. We will distinguish 
several cases. In each case we describe two paths that the two walkers can 
take with positive probability. The two paths will start at 0, reach a fresh common 
level at distinct points, will not backtrack below level 0, and will not have
a chance of a joint regeneration at any previous positive
 level. Then the two walks can 
regenerate with probability $\ge$ $\eta>0$ (Lemma \ref{common-beta-lemma}). 
Note that the part of 
the environment responsible for the two paths and  the part responsible
for regeneration lie in separate half-spaces. Since $\P$ is
product, a positive lower bound for \eqref{YXaux1} is obtained 
and  \eqref{Yline-1} thereby proved.

{\bf Case 1:} $w\cdot\uhat>0$, $z\cdot\uhat>0$, and they are noncollinear.
Let one walk take enough $w$-steps and the other enough $z$-steps.

{\bf Case 2:} $w\cdot\uhat>0$, $z\cdot\uhat>0$, and they are collinear.
Since the walk is not confined to a line (Hypothesis \areg),
there must exist a vector $y$ that is not collinear with
$z,w$ such that $\E\pi_{0y}>0$. 

{\bf Subcase 2.a:} $y\cdot\uhat<0$. Exchanging $w$ and $z$, if necessary,
we can assume $w\cdot\uhat<z\cdot\uhat$.
Let $n>0$ and $m>0$ be such that $nw\cdot\uhat+my\cdot\uhat=0$.
Let one walk take $n$ $w$-steps then $m$ $y$-steps, coming back to level 0,
then $n$ $w$-steps and a $z$-step. The other walk takes $n-1$ $w$-steps
(staying with the first walk), a $z$-step, then a $w$-step.

%
{\bf Subcase 2.b:} $y\cdot\uhat\ge0$. Let $n\ge0$ and $m>0$ be such that 
$nw\cdot\uhat=my\cdot\uhat$. One walk takes a $w$-step, $m$ $y$-steps,
then a $z$-step. The other walk takes a $z$-step,
then $n+1$ $w$-steps. Whenever the walks are on a common level, they
will be at distinct points.

{\bf Case 3:} $w\cdot\uhat=0$ while $z\cdot\uhat>0$. 
The first walk takes a $w$-step then a $z$-step. The second walk takes a $z$-step. 
The case when $w\cdot\uhat>0$ and $z\cdot\uhat=0$ is similar.

{\bf Case 4:} $w\cdot\uhat=z\cdot\uhat=0$. By $\uhat$-transience, 
there exists a $y$ with $y\cdot\uhat>0$
and $\E\pi_{0y}>0$. One walk takes a $w$-step, the other a $z$-step, then both
take a $y$-step.

The rest of the cases treat the situation when $w\cdot\uhat<0$ or $z\cdot\uhat<0$.
Exchanging $w$ and $z$, if necessary, we can assume that $w\cdot\uhat<0$.

{\bf Case 5:} $w\cdot\uhat<0$, $z\cdot\uhat>0$, and they are noncollinear. 
This can be resolved as in 
the proof of Lemma \ref{common-beta-lemma} for $x\ne0$, since now paths
intersections do not matter. More precisely, let $n>0$ and $m>0$ be such that
$nw\cdot\uhat=mz\cdot\uhat$. The first walk takes $m$ $z$-steps, $n$ $w$-steps,
backtracking all the way back to level 0, then $m+1$ $z$-steps. The other walk just
takes $m+1$ $z$-steps.

{\bf Case 6:} $w\cdot\uhat<0$, $z\cdot\uhat>0$, and they are collinear. 
Since the one-dimensional case is excluded, there must exist a vector $y$ noncollinear
with them and such that $\E\pi_{0y}>0$. 

{\bf Subcase 6.a:} $y\cdot\uhat>0$. Let $m>0$ and $n>0$ be such that 
$ny\cdot\uhat+mw\cdot\uhat=0$. Let $k$ be a minimal integer such that 
$kz\cdot\uhat+y\cdot\uhat+mw\cdot\uhat\ge0$. 
The first walk takes $k$ $z$-steps, a $y$-step, $m$ $w$-steps,
$n$ $y$-step, and a $z$-step. The other walk takes $k$ $z$-steps, a $y$-step,
staying so far with the first walk, then splits away and takes a $z$-step.

{\bf Subcase 6.b:} $y\cdot\uhat=0$. Let $m>0$ and $n>0$ be such that 
$nz\cdot\uhat+mw\cdot\uhat=0$. 
The first walk takes $n$ $z$-steps, a $y$-step, $m$ $w$-steps, backtracking
all the way back to level 0, a $y$-step, then takes $n+1$ $z$-step. 
The other walk takes 
$n$ $z$-steps, a $y$-step, staying with the first walk, then takes a $z$-step.

{\bf Subcase 6.c:} $y\cdot\uhat<0$. Let $k>0$, $\ell>0$, $m>0$, and $n>0$ be such that
$\ell z\cdot\uhat=k(w+y)\cdot\uhat$ and $mz\cdot\uhat=ny\cdot\uhat$. The first
walk takes $\ell+m$ $z$-steps, $n$ $y$-steps, $k$ $w$-steps, $k$ $y$-steps,
backtracking back to level 0, then $\ell+m+1$ $z$-steps. The second walk
takes $\ell+m$ $z$-steps, $n$ $y$-steps, staying with the other walk, then
$m+1$ $z$-steps.

{\bf Case 7:} $w\cdot\uhat<0$, $z\cdot\uhat<0$, and they are collinear.
Since the one-dimensional case is excluded, there exists a $u$ noncollinear
with them and such that $\E\pi_{0u}>0$. Furthermore, by $\uhat$-transience,
there exists a $y$ such that $y\cdot\uhat>0$ and $\E\pi_{0y}>0$. It could be
the case that $y=u$.

{\bf Subcase 7.a:} $y$ is not collinear with $w$ and $z$. Let $k$ be the minimal
integer such that $w\cdot\uhat+ky\cdot\uhat>0$. Let $n>0$ and $m>0$ be such that
$n y\cdot\uhat+mz\cdot\uhat=0$. Let $\ell$ by the minimal integer such that
$\ell y\cdot\uhat+w\cdot\uhat+m z\cdot\uhat\ge0$. The first walk takes 
$\ell$ $y$-steps,
a $z$-step, a $w$-step, $m-1$ $z$-steps, then $n+k$ $y$-steps. The other walk
takes $\ell$ $y$-steps, a $w$-step, then $k$ $y$=steps.

{\bf Subcase 7.b:} $y$ is collinear with $w$ and $z$ and $u\cdot\uhat\le0$.
Let $m>0$ and $n>0$ be such that $m(z\cdot\uhat+u\cdot\uhat)=ny\cdot\uhat$. 
Let $k$ be minimal such that $ky\cdot\uhat+w\cdot\uhat+2u\cdot\uhat>0$.
Let $\ell$ be the minimal integer such that 
$\ell y\cdot\uhat+mz\cdot\uhat+w\cdot\uhat+(m+2)u\cdot\uhat\ge0$.
The first walk takes $\ell$ $y$-steps, a $u$-step, a $z$-step, a $w$-step,
$m-1$ $z$-steps, $m+1$ $u$-steps, then $n+k$ $y$-steps.
The other walk takes also $\ell$ $y$-steps and a $u$-step, but then splits
from the first walk taking a $w$-step, a $u$-step, and $k$ $y$-steps.

The subcase when $y$ is collinear with $w$ and $z$ and $u\cdot\uhat>0$ is done
by using $u$ in place of  $y$ in the argument of Subcase 7.a.

{\bf Case 8:}  $w\cdot\uhat<0$, $z\cdot\uhat\le0$, and they are not collinear.
By $\uhat$-transience, 
$\exists y$ 
:  $y\cdot\uhat>0$ and 
$\E\pi_{0y}>0$.

{\bf Subcase 8.a:} $y$ is not collinear with $w$ nor with $z$
and $a y+bw+z\ne 0$ for all integers $a,b>0$.
Let $m>0$ and $n>0$ be such that $mw\cdot\uhat+ny\cdot\uhat=0$. Let $\ell$ be
the minimal integer such that 
$\ell y\cdot\uhat+m w\cdot\uhat+z\cdot\uhat\ge0$.
Let $k$ be the minimal integer such that $z\cdot\uhat+ky\cdot\uhat>0$.
The first walk takes $\ell$ $y$-steps, $m$ $w$-steps, a $z$-step, then $n+k$ $y$-steps.
The other walk takes also $\ell$ $y$-steps, a $z$-step, then $k$ $y$-steps.

{\bf Subcase 8.b:} $y$ is not collinear with $w$ nor with $z$, 
 there exist  integers $a,b>0$ such that   $a y+bw+z=0$ and 
 $z\cdot \uhat=0$.  
One walk takes $a$ $y$-steps, one $z$-step, and one $y$-step.
The other walk takes $a$ $y$-steps, $b$ 
$w$-steps, then $(a+1)$ $y$-steps.

{\bf Subcase 8.c:}  $y$ is not collinear with $w$ nor with $z$, 
 there exist  integers $a,b>0$ such that   $a y+bw+z=0$ and 
 $z\cdot \uhat<0$. 
Pick $k,n>0$  such that $kw\cdot\uhat=nz\cdot\uhat$.  Pick $i,j>0$
 so that $iy\cdot\uhat+jkw\cdot\uhat=0$. 
The first walk takes 
 $i$ $y$-steps, then $jk$ $w$-steps 
followed by $(i+1)$ $y$-steps.
The second walk takes  $i$ $y$-steps, then
 $jn$ $z$-steps followed by $(i+1)$ $y$-steps.

In subcases 8.b and 8.c 
 there are no self-intersections because the pairs $y, w$
and $y,z$  are not collinear.
Also, the two paths cannot intersect because an intersection
together with   $z=-a y-bw$ would force 
 $y$ and $w$ to be collinear.

{\bf Subcase 8.d:} $y$ is collinear with $z$ or with $w$. Exchanging $z$ and
$w$, if necessary, and noting that if $z\cdot\uhat=0$ then $y$ cannot be collinear
with $z$, we can assume that $y$ is collinear with $w$. 
Let $k>0$ and $\ell>0$ be the minimal integers such that 
$kw\cdot\uhat+\ell y\cdot\uhat=0$.
Let $a\ge0$ and $b>0$ be the minimal integers such that $ay\cdot\uhat+bz\cdot\uhat=0$.
Let $m$ be the smallest integer such that $my\cdot\uhat+2z\cdot\uhat>0$. 
Let $n$ be the
smallest integer such that $ny\cdot\uhat+(b+2)z\cdot\uhat+kw\cdot\uhat>0$. 
Now, the first walk takes $n$ $y$-steps, one $z$-step,
$k$ $w$-steps, $(b+1)$ $z$-steps, then $\ell+m+a$ $y$-steps.
The other walk takes $n$ $y$-steps, two $z$-steps, and then $m$
$y$-steps.

\smallskip

{\sl Proof of \eqref{Yline-2}.} 
We appeal here to the construction done in the
 proof of Lemma \ref{common-beta-lemma}.   For 
$x\in\V_d\smallsetminus \{0\}$ the paths constructed
there gave us a bound 
\begin{align*}
P_x[ \,\lvert Y_1\rvert >L\,]=
P_{0,x}\{\beta=\tilde\beta=\infty,\, 
\lvert  \Xtil_{\mutil_1}-X_{\mu_1}\rvert >L\}\ge \delta(x)>0 
\end{align*}
for any given $L$. 
There was a stage in that proof where $x$ may have been 
replaced by $-x$, so the above bound is valid for either $x$ or $-x$.
But  translation shows that 
\[
P_x[ \,\lvert Y_1\rvert =a\,] = P_{-x}[ \,\lvert Y_1\rvert =a\,]
\]
and so we have the estimate for all $x\in\V_d\smallsetminus \{0\}$.  
Considering only 
finitely many $x$ inside a ball gives a uniform lower bound
$\delta=\min_{\abs{x}\le L}\delta(x)>0$. 

\smallskip

{\sl Proof of \eqref{Yline-3}.}  
The proof of Lemma \ref{common-beta-lemma}
 gave us
two paths $\ppath_1=\{0=x_0, x_1, \dotsc, x_{m_1}\}$
 and  $\ppath_2=\{0=y_0, y_1, \dotsc, y_{m_2}\}$ with
positive probability and these additional  properties:
the paths do not backtrack below level $0$, the final points 
$x_{m_1}$ and $y_{m_2}$ are distinct but  
 on a common level $\ell=x_{m_1}\cdot\uhat=y_{m_2}\cdot\uhat>0$,  
  and no level strictly 
between $0$ and $\ell$ can serve as a level
of joint regeneration for the paths. 
  
To recall more specifically from the proof of
Lemma \ref{common-beta-lemma}, these paths were constructed from
two nonzero, noncollinear vectors 
$z,w\in\mathcal J=\{x: \E\pi_{0,x}>0\}$ such that 
$z\cdot\uhat>0$.  If also $w\cdot\uhat>0$, then take 
$\ppath_1=\{(iz)_{0\le i\le m}\}$ and 
$\ppath_2=\{(iw)_{0\le i\le n}\}$ where $m,n$ 
are the minimal positive integers such that $mz\cdot\uhat=nw\cdot\uhat$. 
 In the case $z\cdot\uhat>0\ge w\cdot\uhat$
these paths were given by 
 $\ppath_1=\{(iz)_{0\le i\le m}, (mz+iw)_{1\le i\le n},
(mz+nw+iz)_{1\le i\le m+1}\}$
 and  $\ppath_2=\{(iz)_{0\le i\le m+1}\}$ where now 
$m\ge 0$ and $n>0$ are minimal for  $mz\cdot\uhat=-nw\cdot\uhat$. 

Take $L$ large enough so that 
 $\lvert z_1-z_2\rvert> L$ guarantees that
 paths $z_1+\ppath_1$ and $z_2+\ppath_2$ 
cannot intersect.  Let $\bhat=y_{m_2}-x_{m_1}\in\V_d\smallsetminus\{0\}$. 
 Then by the independence of
environments and \eqref{common-beta},  for $\abs{x}>L$, 
\begin{align*}
&P_x[Y_1-Y_0=\bhat]\ge P_{0,x}\{\beta=\tilde\beta=\infty,\,
 \Xtil_{\mutil_1}=x+y_{m_2},\, X_{\mu_1}=x_{m_1}\}\\
&\ge P_{0,x}\{ X_{0,m_1}=\ppath_1,\, \Xtil_{0,m_2}=x+\ppath_2,\,
\beta\circ\theta^{m_1}=\tilde\beta\circ\theta^{m_2}=\infty\}\\
&\ge \Bigl(\,\prod_{i=0}^{m_1-1}\E\pi_{x_i,x_{i+1}}\Bigr)
\Bigl(\,\prod_{i=0}^{m_2-1}\E\pi_{y_i,y_{i+1}}\Bigr)\cdot\eta>0. 
\end{align*}
The lower bound is independent of $x$.  
The same lower  bound for $P_x[Y_1-Y_0=-\bhat]$ comes by letting
$X$ follow  $\ppath_2$ and $\Xtil$ follow  $x+\ppath_1$. 
This completes the proof of Lemma \ref{Yapplm1}. 


\bibliographystyle{acmtrans-ims}
\bibliography{tmfrefs}

\end{document}